\documentclass[11pt]{article}

\usepackage{amsmath,amssymb,amsfonts,enumitem,amsthm,accents,pstricks,pst-node,pstricks-add,mathrsfs,xy,graphicx}
\xyoption{all}
\usepackage{hyperref}

\numberwithin{equation}{section}

\newtheorem{thm}{Theorem}[section]
\newtheorem{lmm}[thm]{Lemma}
\newtheorem{prp}[thm]{Proposition}
\newtheorem{crl}[thm]{Corollary}
\newtheorem{ques}[thm]{Question}
\theoremstyle{definition}
\newtheorem{dfn}[thm]{Definition}
\newtheorem{eg}[thm]{Example}
\newtheorem{rmk}[thm]{Remark}

\def\BE#1{\begin{equation}\label{#1}}
\def\EE{\end{equation}}
\def\eref#1{(\ref{#1})}

\def\lra{\longrightarrow}

\def\lhra{\ensuremath{\lhook\joinrel\relbar\joinrel\rightarrow}}
\def\xlra#1{\xrightarrow{{#1}}}

\def\ov#1{\overline{#1}}

\def\wt#1{\widetilde{#1}}
\def\sf#1{\textsf{#1}}

\def\tn#1{\textnormal{#1}}
\def\wh#1{\widehat{#1}}

\def\al{\alpha}
\def\be{\beta}
\def\de{\delta}
\def\ep{\epsilon}
\def\ga{\gamma}
\def\io{\iota}

\def\na{\nabla}
\def\om{\omega}
\def\si{\sigma}
\def\th{\theta}
\def\ve{\varepsilon}

\def\vph{\varphi}
\def\vp{\varpi}
\def\vt{\vartheta}
\def\ze{\zeta}

\def\Ga{\Gamma}
\def\De{\Delta}
\def\La{\Lambda}
\def\Om{\Omega}

\def\Si{\Sigma}

\def\cA{\mathcal A}
\def\C{\mathbb C}
\def\cC{\mathcal C}
\def\fD{\mathfrak D}
\def\nD{\tn{D}}
\def\bE{\mathbb E}

\def\cN{\mathcal N}
\def\cO{\mathcal O}
\def\P{\mathbb P}
\def\cP{\mathcal P}
\def\cR{\mathfrak R}
\def\R{\mathbb R}
\def\cR{\mathcal R}
\def\Q{\mathbb Q}
\def\bS{\mathbb S}

\def\cT{\mathcal T}

\def\cZ{\mathcal Z}
\def\Z{\mathbb Z}

\def\nh{\tn{h}}

\def\ne{\textnormal{e}}
\def\fI{\mathfrak i}

\def\AK{\textnormal{AK}}
\def\Aux{\textnormal{Aux}}
\def\Aut{\textnormal{Aut}}
\def\codim{\textnormal{codim}}
\def\cok{\textnormal{cok}}

\def\deR{\textnormal{deR}}
\def\Dom{\textnormal{Dom}}
\def\End{\textnormal{End}}

\def\id{\textnormal{id}}
\def\Im{\textnormal{Im}}

\def\nd{\textnormal{d}}
\def\hor{\textnormal{hor}}

\def\PD{\textnormal{PD}}
\def\cPD{\mathcal{PD}}

\def\supp{\textnormal{supp}}

\def\Symp{\textnormal{Symp}}

\def\ver{\textnormal{ver}}

\def\i{\infty}
\def\w{\wedge}

\def\eset{\emptyset}
\def\prt{\partial}
\def\1{\mathbf 1}

\makeatletter
\newcommand{\tpitchfork}{%
  \vbox{
    \baselineskip\z@skip
    \lineskip-.52ex
    \lineskiplimit\maxdimen
    \m@th
    \ialign{##\crcr\hidewidth\smash{$-$}\hidewidth\crcr$\pitchfork$\crcr}
  }%
}
\makeatother

\begin{document}

\title{Normal crossings singularities for \\
symplectic topology: structures}
\author{Mohammad F.~Tehrani\thanks{Partially supported by NSF grant DMS-2003340}, 
Mark McLean\thanks{Partially supported by NSF grant DMS-1811861}, and 
Aleksey Zinger\thanks{Partially supported by NSF grant DMS-1901979}}

\date{\today}

\maketitle

\begin{abstract}
\noindent
Our previous papers introduce topological notions of normal crossings symplectic divisor and variety, 
show that they are equivalent, in a suitable sense, to the corresponding geometric notions, 
and establish a topological smoothability criterion for normal crossings symplectic varieties. 
The present paper constructs a blowup, a complex line bundle, and 
a logarithmic tangent bundle naturally 
associated with a normal crossings symplectic divisor 
and determines the Chern class of the last bundle. 
These structures have applications in constructions and analysis of various moduli spaces.
As a corollary of the Chern class formula for the logarithmic tangent bundle, 
we refine Aluffi's formula for the Chern class of the tangent bundle of the blowup
at a complete intersection to account for the torsion and
extend it to the blowup at the deepest stratum
of an arbitrary normal crossings divisor.
\end{abstract}

\tableofcontents

\section{Introduction}
\label{intro_sec}
Divisors, i.e.~subvarieties of codimension~1 over the ground field, 
and related structures, are among the central objects of study in algebraic geometry. 
They appear in the study of curves (as dual objects), 
singularities (particularly in the Minimal Model Program), 
and semistable degenerations of smooth varieties (as the singular locus). 
The complex line bundle $\cO_X(V)$ corresponding to a Cartier divisor  $V\!\subset\!X$ and 
the log tangent bundle~$TX(-\log V)$ (or dually the sheaf of log 1-forms $\Om^1_X(\log V)$) 
corresponding to a normal crossings (or NC) divisor 
are among such useful and well-studied structures. 
They play important roles in the relative/log Gromov-Witten theories of Li~\cite{Jun1, Jun2}, 
Gross-Siebert~\cite{GS}, and Abramovich-Chen~\cite{AC,QChen}.
Divisors also play an important role in symplectic topology,
including as representatives of the Poincare duals of symplectic forms~\cite{Donaldson},
in symplectic sum constructions~\cite{Gf,MW}, 
in relative Gromov-Witten theory and symplectic sum formulas 
\cite{Tian,SympSum,Brett15,Brett19-2,Frel},
in affine symplectic geometry \cite{MAff,McLean}, 
in homological mirror symmetry \cite{Au, Sheridan},
and in relative Fukaya category \cite{DF1,DF2}.\\

\noindent
A {\it smooth} symplectic divisor is simply a symplectic submanifold of real codimension~2. 
Applications involving a smooth symplectic divisor $V\!\subset X\!$ typically rely on
the Symplectic Neighborhood Theorem \cite[Theorem~3.4.10]{MS1}. 
It provides an identification~$\Psi$ (which we call a \textsf{symplectic regularization}) 
of a neighborhood of~$V$ in~$X$ with a neighborhood 
of~$V$ in its normal bundle~$\cN_XV$.
Such an identification can then be used to construct an auxiliary $V$-compatible 
``geometric" data/structure, such as a $V$-compatible almost complex structure $J$ 
or a complex line bundle~$\cO_X(V)$ over~$X$ with first Chern class~$[V]_X$. 
One then shows that an invariant defined using~$J$ or the deformation equivalence class 
of an object constructed using~$\Psi$ depends only on~$(X,V,\om)$. 
This approach is relatively straightforward to carry out in the case of smooth divisors.\\

\noindent
Singular symplectic divisors/varieties and structures associated with them are 
generally hard to define and work with because there is no direct analogue of 
the Symplectic Neighborhood or Darboux Theorem in this setting.
Following an alternative approach,
\cite{SympDivConf,SympDivConf2} introduce topological notions of NC symplectic divisor 
and variety and geometric notions of regularization for NC symplectic divisors and varieties.
The latter is basically a ``nice" neighborhood identification of the divisor/singular locus,
analogous to that provided by the Symplectic Neighborhood Theorem in the smooth case.
Every NC symplectic divisor/variety is deformation equivalent to one admitting a regularization. 
For this reason, in our approach, we work with the entire deformation equivalence classes of 
NC symplectic divisors/varieties, as opposed to a fixed NC symplectic divisor/variety;
see the end of Section~\ref{NCLreg_subs}.\\

\noindent
As stated in~\cite{SympNCSumm}, a regularization of an NC symplectic variety $V\!\!\subset\!X$
can be used to construct an associated complex line bundle~$\cO_X(V)$
and a log tangent bundle~$TX(-\log V)$.
The present paper carries out these constructions in detail.
While the constructions of~$\cO_X(V)$ and~$TX(-\log V)$
involve the auxiliary data of regularizations and other choices, 
their deformation equivalence classes 
depend on the deformation equivalence class of the symplectic structure only.\\

\noindent
We denote the smooth locus of an NC symplectic divisor $V\!\subset\!X$ by~$V^*$. 
For $r\!\in\!\Z^{\ge0}$, let $V^{(r)}\!\subset\!V$ be $r$-fold locus of~$V$,
i.e.~the locus that locally is the intersection of at least $r$ branches of~$V$;
see~\eref{Vrdfn_e}.
For example, $V^{(0)}\!=\!X$, $V^{(1)}\!=\!V$, $V^{(2)}$ is the singular locus of~$V$,
and $V^*\!=\!V^{(1)}\!-\!V^{(2)}$.
The subspace $V^{(r)}\!-\!V^{(r+1)}$ is a smooth submanifold of~$X$ of codimension~$2r$.
We denote its inclusion into~$X$ by~$\io_{V^{(r)}-V^{(r+1)}}$.
If $X$ is compact and of (real) dimension~$2n$, $\io_{V^{(r)}-V^{(r+1)}}$
is a pseudocycle in~$X$ of dimension $2(n\!-\!r)$; see~\cite{pseudo}.
Thus, $V^*$ and $V^{(r)}\!-\!V^{(r+1)}$ determine homology classes
$$[V]_X\in H_{2n-2}(X;\Z) \qquad\hbox{and}\qquad 
\big[V^{(r)}\big]_X\in H_{2(n-r)}(X;\Z),$$
respectively.

\begin{prp}\label{OXV_prp}
Let $(X,\om)$ be a symplectic manifold  and $V\!\subset\! X$ be an NC symplectic divisor. 
\begin{enumerate}[label=(\arabic*),leftmargin=*]

\item\label{cOdfn_it} An $\om$-regularization~$\cR$ for $V\!\subset\!X$ 
determines a complex line bundle~$(\cO_{\cR;X}(V),\fI_{\cR})$  
over~$X$ with a smooth section~$s_{\cR}$ so that $s_{\cR}^{-1}(0)\!=\!V$ and
$$\nD s_{\cR}\!: \cN_X(V^*)\lra \cO_{\cR;X}(V)\big|_{V^*}$$
is an orientation-preserving isomorphism along the smooth locus $V^*$ of~$V$.

\item\label{cOinv_it} The deformation equivalence class~$(\cO_{X}(V),\fI)$ 
of~$(\cO_{X}(V),\fI_{\cR})$  depends only on the deformation equivalence class of $(X,V,\om)$.

\item\label{cOsplit_it} If $V'\!\subset\!X$ is another NC symplectic divisor so that 
$V\!\cup\!V'\!\subset\!X$ is also an NC symplectic divisor
and $V\!\cap\!V'$ contains no open subspace of~$V$, then
\BE{cOsplit_e} 
\big(\cO_{X}(V\!\cup\!V'),\fI\big)\approx 
\big(\cO_{X}(V),\fI\big)\!\otimes\!\big(\cO_{X}(V'),\fI\big).\EE

\end{enumerate}
\end{prp}

\vspace{.1in}

\noindent
The complex line bundle $\cO_X(V)$ appears in the smoothability criterion for 
SC symplectic varieties in~\cite{SympSumMulti} and for general NC symplectic varieties 
in~\cite{SympSumMulti2}.
For a simple (normal) crossings (or SC) symplectic divisor $V\!=\!\bigcup_{i\in S}\!V_i$
as in Definition~\ref{SCD_dfn},  
\eref{cOsplit_e} gives
$$\cO_X(V)\cong \bigotimes_{i\in S}\cO_X(V_i)\lra X.$$
If $X$ is compact and of (real) dimension~$2n$, the stated properties of~$s_{\cR}$ imply~that
\BE{ccOV_e}c_1\big(\cO_X(V)\!\big)=\PD_X\big([V]_X\big)\in H^2(X;\Z).\EE
If $X$ is not compact, this identity holds with $[V]_X$ denoting the element 
of the Borel-Moore homology of~$X$ determined by~$V^*$; see~\cite{BM}.

\begin{thm}\label{TXV_thm}
Let $(X,\om)$ be a symplectic manifold  and $V\!\subset\!X$ be an NC symplectic divisor.  
\begin{enumerate}[label=(\arabic*),leftmargin=*]

\item\label{loddfn_it} 
An $\om$-regularization~$\cR$ for $V\!\subset\!X$ determines a vector bundle~$T_{\cR}X(-\log V)$ 
over~$X$ with a smooth vector bundle homomorphism 
\BE{Ninc_e}\io_{\cR}\!: T_{\cR}X(-\log V)\lra TX\EE 
so that for every $r\!\in\!\Z^{\ge0}$
\begin{gather*}
T\big(V^{(r)}\!-\!V^{(r+1)}\big)\subset T_{\cR}X(-\log V)\big|_{V^{(r)}-V^{(r+1)}},
\quad  \io_{\cR}|_{T(V^{(r)}-V^{(r+1)})}=\nd\io_{V^{(r)}-V^{(r+1)}},\\
\hbox{and}\qquad  
\io_{\cR}\big(T_{\cR}X(-\log V)\big|_{V^{(r)}-V^{(r+1)}}\big)=T(V^{(r)}\!-\!V^{(r+1)}).
\end{gather*}

\item\label{logJ_it}
An $\cR$-compatible almost complex structure~$J$ on~$X$ determines a complex structure~$\fI_{\cR,J}$ 
on the vector bundle $T_{\cR}X(-\log V)$ so that the bundle homomorphism~\eref{Ninc_e} is $\C$-linear.

\item\label{loginv_it}
The deformation equivalence class~$(TX(-\log V),\fI)$ of~$(T_{\cR}X(-\log V),\fI_{\cR,J})$ 
depends only on the deformation equivalence class of $(X,V,\om)$.

\item\label{logsplit_it} If $V'\!\subset\!X$ is a smooth submanifold so that 
$V\!\cup\!V'\!\subset\!X$ is also an NC symplectic divisor and 
$V\!\cap\!V'$ contains no open subspace of~$V$, then
\BE{logsplit_e} 
\big(TX(-\log(V\!\cup\!V')\!),\fI\big)\!\oplus\!\big(\cO_X(V'),\fI\big)
\approx \big(TX(-\log V),\fI\big)\!\oplus\!(X\!\times\!\C).\EE

\item\label{cTXV_it} We have 
\BE{cTXV_e} c\big(TX(-\!\log V),\fI\big)=
\frac{c(TX,\om)}{1\!+\!\PD_X([V^{(1)}]_X)\!+\!\PD_X([V^{(2)}]_X)
\!+\!\ldots}\in H^*(X;\Q).\EE
The above equality holds in $H^*(X;\Z)$ if $V\!\subset\!X$ is an SC divisor.
\end{enumerate}
\end{thm}

\noindent 
As shown in~\cite{Frt}, 
the complex vector bundle $(TX(-\log V),\fI)$ plays the same role
in the deformation-obstruction theory of pseudoholomorphic curves relative to 
an NC symplectic divisor as the complex vector bundle $(TX,J)$ 
in the standard deformation-obstruction theory of pseudoholomorphic curves;
see \cite[Chapter~3]{MS2}, for example.
This provides a symplectic topology perspective on the constructions of log stable maps
in \cite{Brett15,GS,AC}.
In~\cite{FS}, the vector bundle $TX(-\log V)$ is used to define Seiberg-Witten invariants
of a closed oriented 4-manifold~$X$ relative to a smooth oriented Riemann surface~$V$.
This perspective on the standard constructions of relative Seiberg-Witten invariants
reveals additional structures.\\

\noindent
For an SC symplectic divisor $V\!=\!\bigcup_{i\in S}\!V_i$ as in Definition~\ref{SCD_dfn}, 
\eref{logsplit_e} gives
\BE{TXlogSC_e} \big(TX(-\log V),\fI\big)\oplus
\bigoplus_{i\in S}\big(\cO_X(V_i),\fI\big)\approx 
(TX,J)\!\oplus\!(X\!\times\!\C^S,\fI).\EE
This immediately implies that
\BE{cTXV_eSC}\begin{split}
c\big(TX(-\!\log V)\!\big)
&=\frac{c(TX,\om)}{\prod_{i\in S}\big(1\!+\!\PD_X([V_i]_X)\big)}\\
&=\frac{c(TX,\om)}{1\!+\!\PD_X([V^{(1)}]_X)\!+\!\PD_X([V^{(2)}]_X)
\!+\!\ldots}\in H^*(X;\Z).
\end{split}\EE
The direct sum vector bundle on the left-hand side of~\eref{TXlogSC_e} does not even
exist as a vector bundle in the general NC case;
see the example in the second half of Section~\ref{ChernQvsZ_sec}.
In Section~\ref{ChNC_sec}, we instead establish the de Rham cohomology analogue of~\eref{cTXV_e}  
by expressing the Chern classes on the two sides of~\eref{cTXV_e} 
in terms of the curvatures of connections in the vector bundles~$TX$ and~$TX(-\!\log V)$.
We construct a $2k$-form~$\tau_k$ on~$X$, supported in a neighborhood of~$V^{(k)}$ and
representing $\PD_X([V^{(k)}]_X)$ in $H^{2k}_{\deR}(X)$, and $\C$-linear connections~$\na$ in~$TX$ 
and~$\na'$ in~$TX(-\log V)$ so that the curvature of~$\na$ is the correct combination of 
the curvature of~$\na'$ and~$\tau_1,\ldots,\tau_r$ to yield~\eref{cTXV_e}; 
see~\eref{FtoC_e}, Lemma~\ref{PDVk_lmm}, and Proposition~\ref{DFChern_prp}.
Our proofs of~\eref{logsplit_e} and~\eref{cTXV_e} are carried out in the almost complex category.\\

\noindent
When $V\!\subset\!X$ is either an NC complex divisor in a complex manifold 
or an NC almost complex divisor in an almost complex manifold compatible with a regularization for~$V$,
in the sense defined in Section~\ref{NCLreg_subs}, 
we can pass to the blowup~$\wt{X}$ of~$X$ along the deepest stratum~$V^{(r)}$ of~$V$
and to the proper transform~$\ov{V}$ of~$V$.
We can then compare the log tangent bundles for~$(X,V)$ and~$(\wt{X},\ov{V})$
and their Chern classes via~\eref{logsplit_e}, \eref{cTXV_e}, and 
Lemma~\ref{blowuplog_lmm} below.
This immediately yields Corollary~\ref{blowupchern_crl} below,
except for the refinement in the vanishing torsion case.
This corollary refines  \cite[Lemma~1.3]{Aluffi} in the SC case
and extends it to the general NC case.

\begin{lmm}\label{blowuplog_lmm}
Let $(X,J)$ be an almost complex manifold, $V\!\subset\!X$ be an NC almost complex divisor
with a regularization~$\cR$, and $r\!\in\!\Z^+$ be such that $V^{(r+1)}\!=\!\eset$.
If \hbox{$\pi\!:(\wt{X},\wt{J})\lra\!(X,J)$} is the blowup of~$(X,J)$ along~$V^{(r)}$ 
determined by~$\cR$, $\bE$ is the exceptional divisor,
and $\ov{V}\!\subset\!\wt{X}$ is the proper transform of~$X$, then
$\wt{V}\!\equiv\!\ov{V}\!\cup\!\bE$ is an NC almost complex divisor in~$(\wt{X},\wt{J})$.
In this case,
there are a regularization~$\wt\cR$ of~$\wt{V}$ in~$(\wt{X},\wt{J})$ 
and an isomorphism
\BE{blowuplog_e}\nd^{\log}_{\cR\wt\cR}\pi\!:  T_{\wt{\cR}}\wt{X}(-\log\wt{V})
\xlra{~\cong~}\pi^*\big(T_\cR X(-\log V)\!\big)\EE
so that the diagram
$$\xymatrix{T_{\wt{\cR}}\wt{X}(-\log\wt{V})
\ar[d]^{\nd_{\cR\wt\cR}^{\log}\pi}\ar[rr]^>>>>>>>>>>>>{\io_{\wt\cR}}
&&T\wt{X}\ar[d]^{\nd \pi}\\
\pi^*(T_\cR X(-\log V)\!)\ar[rr]^>>>>>>>>>>{\pi^*\io_{\cR}} 
&&\pi^*TX}$$
commutes.   
The first and third claims above also hold in the category of 
complex manifolds with NC divisors.
\end{lmm}

\begin{crl}\label{blowupchern_crl}
With the assumptions as in Lemma~\ref{blowuplog_lmm},
\BE{blowupchern_e}\begin{split}
&\frac{c(T\wt{X})}{(1\!+\!\PD_{\wt{X}}([\ov{V}^{(1)}]_{\wt{X}})
\!+\!\PD_{\wt{X}}([\ov{V}^{(2)}]_{\wt{X}})\!+\!\ldots)(1\!+\!\PD_{\wt{X}}([\bE]_{\wt{X}}))}\\
&\hspace{1.5in}
=\pi^*\bigg(\frac{c(TX)}{1\!+\!\PD_X([V^{(1)}]_X)\!+\!\PD_X([V^{(2)}]_X)\!+\!\ldots}
\bigg)
\in H^*(\wt{X};\Q)\,.
\end{split}\EE
The above equality holds in $H^*(\wt{X};\Z)$ if $V\!\subset\!X$ is an SC divisor
or the torsion in $H_*(\bE;\Z)$ lies in the kernel of the homomorphism~$\wt\io_*$
induced by the inclusion \hbox{$\bE\!\lra\!\wt{X}$}.
\end{crl}

\noindent
The statement of Lemma~\ref{blowuplog_lmm} in the complex category is well-known;
its proof is recalled at the end of Section~\ref{AG_subs}.
We establish this lemma in the almost complex category with regularizations in 
Sections~\ref{AlCBl_sub}-\ref{blowuplog_subs}.
The regularizations are used to construct the bundles $TX(-\log V)$ and
$T\wt{X}(\!-\log(\ov{V}\!\cup\!\bE)\!)$ and the blowup~$\wt{X}$.\\

\noindent
We can also pass to {\it a} blowup~$\wt{X}$ of~$X$ along the deepest stratum~$V^{(r)}$ of~$V$
and the proper transform $\ov{V}\!\subset\!\wt{X}$ of~$V$ 
if $V\!\subset\!X$ is an NC symplectic divisor
with~$V^{(r)}$ admitting a tubular symplectic neighborhood 
that contains the disk subbundle of~$\cN_XV^{(r)}$ of a fixed radius; see Section~\ref{SymplBl_subs}.
This is automatically the case if~$V^{(r)}$ is the compact.
If so, each deformation equivalence class~$[\om]$ of symplectic forms on~$(X,V)$ 
determines a deformation equivalence class~$[\wt\om]$ 
of symplectic forms on the blowup~$(\wt{X},\wt{V})$ of~$(X,V)$ along~$V^{(r)}$, 
a homotopy class of blowdown maps
$$\pi\!:(\wt{X},\wt{V})\lra(X,V),$$
and a homotopy class of isomorphisms
\BE{blowuplog_e2}\nd^{\log}\pi\!:  T\wt{X}(-\log\wt{V})
\xlra{~\cong~}\pi^*\big(TX(-\log V)\!\big)\EE
between the log tangent bundles associated with~$[\om]$ and~$[\wt\om]$.
If $V^{(r)}$ is not compact, 
the existence of  a tubular symplectic neighborhood that contains the disk subbundle of~$\cN_XV^{(r)}$
of a fixed radius is unclear, even after deforming the symplectic form.
We suspect that the answers to the following closely related questions are negative in general;
the affirmative answer to Question~\ref{SymplBl_ques2} would imply 
the affirmative answer to Question~\ref{SymplBl_ques}.

\begin{ques}\label{SymplBl_ques}
Let $(V,\om)$ be a symplectic manifold, $\cN\!\lra\!V$ be a direct sum of Hermitian line
bundles $(L_i,\rho_i,\na^{(i)})$ determining fiberwise symplectic forms~$\Om_i$ on~$L_i$,
and $\cN'\!\subset\!\cN$ be a neighborhood of~$V$.
Is there a deformation of~$\om$ through symplectic forms and of $(\rho_i,\na^{(i)})$ through
Hermitian structures so that the induced 2-form~$\wt\om$ as in~\eref{ombund_e}
on the total space of~$\cN$ is symplectic on the unit ball subbundle of~$\cN$,
with respect to the deformed metric, and this subbundle is contained in~$\cN'$?
\end{ques}

\begin{ques}\label{SymplBl_ques2}
Let $(V,\om)$ be a symplectic manifold, $J$ be an $\om$-compatible almost complex structure, 
and $C\!:V\!\lra\!\R^+$ be a smooth function.
Are there a symplectic form~$\om'$ on~$V$ deformation equivalent to~$\om$ 
and an $\om'$--compatible almost complex structure~$J'$ so~that
$$\om'(v,J'v)\ge C(x)\om(v,Jv) \qquad\forall~x\!\in\!V,~v\!\in\!T_xV.$$
\end{ques}

\vspace{.15in}

\noindent
We review the complex geometric constructions of the complex line bundle $(\cO_X(V),\fI)$ and
the complex vector bundle $(TX(-\log V),\fI)$ associated with an NC divisor~$V$ 
in a complex manifold~$V$ and some of their properties in Section~\ref{AG_subs}.
As a warmup to the general case, we construct these bundles for a smooth symplectic divisor
in Section~\ref{smooth_subs}.
For the reader's convenience,
Sections~\ref{SCdfn_subs} and~\ref{SCreg_subs} recall the notions of SC symplectic divisor and 
regularization, respectively, introduced in~\cite{SympDivConf}.
Section~\ref{ConSC_subs} contains the constructions of the vector bundles~$\cO_X(V)$
and $TX(-\log V)$ for an SC symplectic divisor $V\!\subset\!X$ and establishes
Proposition~\ref{OXV_prp} and the first three statements of Theorem~\ref{TXV_thm} 
in this setting.
Section~\ref{logsplit_subs} establishes Theorem~\ref{TXV_thm}\ref{logsplit_it}
for an SC symplectic divisor~$V$. 
The constructions and proofs in the SC case illustrate the arguments
in the general NC case, which are more notionally involved.
Sections~\ref{NCLdfn_subs}-\ref{ConLNC_subs} and \ref{NCGdfn_subs}-\ref{ConGNC_subs}
are the analogues of Sections~\ref{SCdfn_subs}-\ref{ConSC_subs} in
the local and global perspectives, respectively, on the NC symplectic divisors introduced
in~\cite{SympDivConf2}.  
Section~\ref{ConLNC_subs} also shows why the proof of Theorem~\ref{TXV_thm}\ref{logsplit_it}
for SC symplectic divisors in Section~\ref{logsplit_subs} immediately extends to
the general NC case.
Sections~\ref{blowup_sec} and~\ref{ChNC_sec} establish 
Lemma~\ref{blowuplog_lmm} and Theorem~\ref{TXV_thm}\ref{cTXV_it}, respectively.
In Section~\ref{ChernQvsZ_sec}, we establish the remaining statement of 
Corollary~\ref{blowupchern_crl} and show 
that~\eref{cTXV_e} and \eref{blowupchern_e} do not need 
to hold with $\Z$-coefficients for arbitrary NC divisors.

\section{Standard settings}
\label{stand_sec}

\subsection{NC complex divisors}
\label{AG_subs}

\noindent
Let $X$ be a complex manifold of (complex) dimension~$n$
with structure sheaf~$\cO_X$ (the sheaf of local holomorphic functions).
An \sf{NC divisor} in~$X$ is a subvariety $V\!\subset\!X$ locally defined by an equation of the~form
\BE{CNCdfn_e}z_1\!\ldots\!z_k = 0 \EE
in a holomorphic coordinate chart $(z_1,\ldots,z_n)$ on~$X$.
The sheaf of local meromorphic functions with simple poles along the smooth locus of~$V$
is freely generated in such a coordinate chart by the meromorphic function $1/z_1\!\ldots\!z_k$
as a module over~$\cO_X$.
Since this sheaf is locally free of rank~1, 
it is the sheaf of local holomorphic sections of a holomorphic line bundle~$\cO_X(V)$. 
The constant function~1 on~$X$ determines a holomorphic section~$s$ of this sheaf
satisfying the properties of~$s_{\cR}$ in Proposition~\ref{OXV_prp}\ref{cOdfn_it}.
It is immediate that~\eref{cOsplit_e} holds as~well.
The dual of~$\cO_X(V)$ is the holomorphic line bundle~$\cO_X(-V)$;
the sheaf of its local holomorphic sections is freely generated in a coordinate chart
as above by the holomorphic function $z_1\!\ldots\!z_k$.\\

\noindent
In a local chart as in~\eref{CNCdfn_e},
the sheaf $\cT X\!\equiv\!\cO(TX)$ of local holomorphic sections of the tangent bundle~$TX$ is generated by 
the coordinate vector fields $\prt_{z_1},\ldots,\prt_{z_n}$.
The \sf{logarithmic tangent sheaf} $\cT X(-\log V)$ is the subsheaf of~$\cT X$ generated~by
the vector fields 
$$\prt^{\log}_{z_1}\equiv z_1\prt_{z_1},\quad \ldots, \quad \prt^{\log}_{z_k}\equiv z_k\prt_{z_k},
\quad  \prt_{z_{k+1}},\quad \ldots,\quad \prt_{z_n}.$$
The dual of this subsheaf is the sheaf of logarithmic 1-forms $\Om^1_X(\log V)$ is the sheaf generated by 
$$\frac{\nd z_1}{z_1},\quad\ldots,\quad \frac{\nd z_k}{z_k}, 
\quad\nd z_{k+1},\quad\ldots,\quad \nd z_n.$$
Since $\cT X(-\log V)$ is locally free, 
it is the sheaf of local holomorphic sections of a holomorphic vector bundle $TX(-\log V)$. 
The inclusion of~$\cT X(-\log V)$ into~$\cT X$ gives rise to a holomorphic homomorphism 
$$\io\!: TX(-\log V) \lra TX$$
that realizes every section of $TX(-\log V)$ as a section of $TX$ with values in~$TV$ along~$V$.\\

\noindent
The normalization $\io\!:\wt{V}\!\lra\!V\!\subset\!X$ of~$V$ is an immersion.
The inclusion of the sheaf~$\Om^1_X$ of 1-forms on~$X$ into~$\Om^1_X(\log V)$
and the Poincare residue~map induce an exact sequence
$$0\lra \Om_X^1 \lra  \Om_X^1(\log V) \lra \io_*\cO_{\wt{V}}\lra 0$$
of sheaves on~$X$,
where $\io_*\cO_{\wt{V}}$ is the direct image (or push-forward) sheaf of
the structure sheaf $\cO_{\wt{V}}$ of~$\wt{V}$.
Therefore,
\BE{OmXc_e}c\big(\Om_X^1(\log V)\!\big)=c\big(\Om_X^1\big)\,c\big(\io_*\cO_{\wt{V}}\big).\EE
If $V\!=\!\bigcup_{i\in S}V_i$ is an SC divisor, then 
$\io_*\cO_{\wt{V}}\!=\!\bigoplus_{i\in S}\cO_{V_i}$ as sheaves on~$X$.
Furthermore, there is an exact sequence
$$0\lra \cO_X(-V_i) \lra \cO_X \lra \cO_{V_i}\lra0$$
of sheaves on $X$ for each $i\!\in\!S$.
Thus, 
$$c\big(\io_*\cO_{\wt{V}})\prod_{i\in S}\!\!\big(1\!-\!\PD_X\big([V_i]_X\big)\!\big)
=1\in H^2(X;\Z)$$
in this case. We thus obtain~\eref{cTXV_eSC} in the complex setting.\\

\noindent
For an arbitrary NC divisor $V\!\subset\!X$,
the derived direct image sheaf $\io_!\cO_{\wt{V}}$ of~$\cO_{\wt{V}}$
coincides with the direct image sheaf~$\io_*\cO_{\wt{V}}$ 
because the higher derived functors  for an immersion vanish.
Along with the Grothendieck-Riemann-Roch theorem, this gives
$$\tn{ch}\big(\io_*\cO_{\wt{V}}\big)\,\tn{td}(X)
=\tn{ch}\big(\io_!\cO_{\wt{V}}\big)\,\tn{td}(X)
=\io_*\big(\tn{ch}(\cO_{\wt{V}})\,\tn{td}(\wt{V})\!\big)
=\io_*\big(\tn{td}(\wt{V})\!\big),$$
where $\tn{ch}$ is the Chern character and $\tn{td}$ is the Todd class;
see \cite[Theorem~1.3]{Ryan}.
Thus,
\BE{chFormula_e}
\tn{ch}\big(\iota_*\cO_{\wt{V}}\big)
=\frac{\iota_*(\tn{td}(\wt{V}))}{\tn{td}(X)}\in H^*(X,\Q).\EE
This formula holds only with $\Q$-coefficients because the Chern character is 
a map from the K-theory of~$X$ to the rational Chow group of~$X$.
The proof of~\eref{chFormula_e} in \cite[Corollary~5.22]{Ryan} uses blowups 
to reduce the problem to SC divisors;
we are not aware of a direct way for obtaining~\eref{chFormula_e}.
For the purposes of computing $c(TX(-\log V)\!)$ via~\eref{OmXc_e}, 
it is still necessary to translate $\tn{ch}(\io_*\cO_{\wt{V}})$ into $c(\io_*\cO_{\wt{V}})$.
Nevertheless, it is feasible to directly study the change in $c(TX(-\log V)\!)$ 
under blowups, as is done in the proofs of Lemma~\ref{blowuplog_lmm}
in Section~\ref{blowuplog_subs}, and 
to relate~$c(TX(-\log V)\!)$ to~$c(TX)$ in the spirit of~\eref{OmXc_e},
as is done in the proof of~\eref{cTXV_e} in Section~\ref{ChNC_sec}.\\

\noindent
Let $r\!\in\!\Z^+$, $\pi\!:\wt{X}\lra\!X$, 
and $\bE,\ov{V}\!\subset\!\wt{X}$ be as in Lemma~\ref{blowuplog_lmm}
and $\wt{U}\!\equiv\!\pi^{-1}(U)$ be the preimage of a coordinate chart~$U$ as 
in the sentence~\eref{CNCdfn_e}.
For each $i\!=\!1,\ldots,k$,
$$\wt{U}_i\equiv\big\{\!\big(z_i,(u_j)_{j\in[k]-i},(z_j)_{j\in[n]-[k]}\big)\!:
\big(u_1z_1,\ldots,u_{i-1}z_i,z_i,u_{i+1}z_i,\ldots,u_kz_i,(z_j)_{j\in[n]-[k]}\big)\!\in\!U\big\},$$
where $[k]\!\equiv\!\{1,\ldots,k\}$, is a coordinate chart on~$\wt{U}\!\subset\!\wt{X}$;
these charts cover~$\wt{U}$.
Since
\begin{gather*}
\ov{V}\!\cap\!\wt{U}_i=\big(u_1\ldots u_{i-1}u_{i+1}\ldots u_k\!=\!0\big),\quad
\bE\!\cap\!\wt{U}_i=(z_i\!=\!0),\\
\hbox{and}\qquad
\frac{\nd z_j}{z_j}=\frac{\nd u_j}{u_j}+\frac{\nd z_i}{z_i}\qquad \forall~j\in [k]\!-\!i,
\end{gather*}
we obtain that
$$\pi^*\Om^1_X(\log V)= \Om^1_{\wt{X}}\big(\log(\ov{V}\!\cup\!\bE)\!\big).$$
This establishes Lemma~\ref{blowuplog_lmm} in the complex setting.

\subsection{Smooth symplectic divisors}
\label{smooth_subs}

\noindent
It is fairly straightforward to adapt the constructions of $\cO_X(V)$ and $TX(-\log V)$ 
in Section~\ref{AG_subs} via the Symplectic Neighborhood Theorem.
Before doing so below, we carefully formulate the relevant notions.\\ 

\noindent
Let $V$ be a smooth manifold.
For a vector bundle $\pi\!:\cN\!\lra\!V$, we denote by $\ze_{\cN}$ 
\sf{the radial vector field} on the total space of~$\cN$; it is given~by
$$\ze_{\cN}(v)=(v,v)\in\pi^*\cN =T\cN^{\ver} \lhra T\cN\,.$$
Let $\Om$ be a fiberwise 2-form on~$\cN\!\lra\!V$.
A connection~$\na$ on~$\cN$ induces a projection $T\cN\!\lra\!\pi^*\cN$ and thus
determines an extension~$\Om_{\na}$ of~$\Om$ to a 2-form on 
the total space of~$\cN$.
If $\om$ is a closed 2-form on~$V$, the 2-form
\BE{ombund_e2}
\wt\om \equiv \pi^*\om+\frac12\nd\io_{\ze_{\cN}}\Om_{\na}
\equiv \pi^*\om+\frac12\nd\big(\Om_{\na}(\ze_{\cN},\cdot)\big)\EE
on the total space of $\cN$ is also closed and
restricts to~$\Om$ on $\pi^*\cN\!=\!T\cN^{\ver}$.
If $\om$ is a symplectic form on~$V$ and $\Om$ is a fiberwise symplectic form on~$\cN$,
then~$\wt\om$ is a symplectic form on a neighborhood of~$V$ in~$\cN$.\\

\noindent
We call $\pi\!:(L,\rho,\na)\!\lra\!V$ a \sf{Hermitian line bundle} if
$L\!\lra\!V$ is a smooth complex line bundle,
$\rho$ is a Hermitian metric on~$L$, 
and $\na$ is a $\rho$-compatible connection on~$L$.
We use the same notation~$\rho$ to denote the square of the norm function on~$L$
and the Hermitian form on~$L$ which is $\C$-antilinear in the second input.
Thus,
$$\rho(v)\equiv\rho(v,v), \quad 
\rho(\fI v,w)=\fI\rho(v,w)=-\rho(v,\fI w) 
\qquad\forall~(v,w)\!\in\!L\!\times_V\!L.$$
Let $\rho^{\R}$  denote the real part of the form~$\rho$.\\

\noindent
A Riemannian metric on an oriented  real vector bundle \hbox{$L\!\lra\!V$} of rank~2
determines a complex structure on the fibers of~$V$.
A \sf{Hermitian structure} on an oriented  real vector bundle \hbox{$L\!\lra\!V$} of rank~2
is a pair $(\rho,\na)$ such that $(L,\rho,\na)$ is a Hermitian line bundle
with the complex structure~$\fI_{\rho}$ determined by the Riemannian metric~$\rho^{\R}$.
If $\Om$ is a fiberwise symplectic form on an oriented vector bundle \hbox{$L\!\lra\!V$} of rank~2,
an \sf{$\Om$-compatible Hermitian structure} on~$L$ is a Hermitian structure $(\rho,\na)$ on~$L$ 
such that $\Om(\cdot,\fI_{\rho}\cdot)=\rho^{\R}(\cdot,\cdot)$.

\begin{dfn}\label{smreg_dfn}
Let $X$ be a manifold and $V\!\subset\!X$ be a submanifold
with normal bundle $\cN_XV\!\lra\!V$. 
A (smooth) \sf{regularization for~$V$ in~$X$} is a diffeomorphism $\Psi\!:\cN'\!\lra\!X$
from a neighborhood of~$V$ in~$\cN_XV$ onto a neighborhood of~$V$ in~$X$ such
that $\Psi(x)\!=\!x$ and the isomorphism
$$ \cN_XV|_x=T_x^{\ver}\cN_XV \lhra T_x\cN_XV
\stackrel{\nd_x\Psi}{\lra} T_xX\lra \frac{T_xX}{T_xV}\equiv\cN_XV|_x$$
is the identity for every $x\!\in\!V$.
\end{dfn}

\noindent
Let  $V$ be a closed symplectic submanifold of a  symplectic submanifold in~$(X,\om)$. 
The normal bundle~$\cN_XV$ of~$V$ in~$X$ then inherits a fiberwise symplectic form~$\om|_{\cN_XV}$
from~$\om$ via the isomorphism
$$\pi_{\cN_XV}\!:
TV^{\om}\!\equiv\! \big\{v\!\in\!T_xX\!:\,x\!\in\!V,\,\om(v,w)\!=\!0~\forall\,w\!\in\!T_xV\big\}
\stackrel{\approx}{\lra}\frac{TX|_V}{TV}\!\equiv\!\cN_XV\,.$$
The symplectic form~$\om|_V$ on~$V$,
the fiberwise 2-form $\Om\!\equiv\!\om|_{\cN_XV}$ on~$\cN_XV$,
and a connection~$\na$ on~$\cN_X V$ thus determine a 2-form~$\wt\om_{\na}$ 
on the total space of~$\cN_X V$ via~\eref{ombund_e2}.
By the Symplectic Neighborhood Theorem, there exists 
a regularization $\Psi\!:\cN'\!\lra\!X$ for~$V$ in~$X$ so that $\Psi^*\om\!=\!\wt\om_{\na}|_{\cN'}$.\\

\noindent
Suppose in addition that $V$ is of codimension~2, i.e.~$V$ is a smooth symplectic divisor
in~$(X,\om)$. 
If $(\rho,\na)$ is an $\om|_{\cN_XV}$-compatible Hermitian structure on~$\cN_XV$, 
the triple $\cR\!=\!(\!(\rho,\nabla),\Psi)$ is an $\om$-regularization for~$V$ in~$X$ 
in the sense of Definition~\ref{TransCollregul_dfn} 
and determines a fiberwise complex structure~$\fI_{\rho}$ on~$\cN_XV$. 
Let
\BE{OXVSmooth_e}\begin{split}
&\cO_{\cR;X}(V)= \big(\{\Psi^{-1}\}^{\!*}\pi^*\cN_XV\!\sqcup\! 
(X\!-\!V)\!\times\!\C\big)\big/\!\!\sim\,\lra 
\Psi(\cN')\!\cup\!(X\!-\!V)\!=\!X,\\
&\hspace{.3in}\{\Psi^{-1}\}^{\!*}\pi^*\cN_XV
\ni\big(\Psi(v),v,cv\big)\sim\big(\Psi(v),c\big)\in (X\!-\!V)\!\times\!\C,
\end{split}\EE
where $\pi\!:\cN'\!\lra\!V$ is the bundle projection map.
This defines a smooth complex line bundle over~$X$.
The smooth section~$s_{\cR}$ of this bundle given~by
$$s_{\cR}(x)=\begin{cases}[x,v,v],&\hbox{if}~x\!=\!\Psi(v),~v\!\in\!\cN';\\
[x,1],&\hbox{if}~x\!\in\!X\!-\!V;
\end{cases}$$
satisfies the properties stated in Proposition~\ref{OXV_prp}\ref{cOdfn_it}.\\

\noindent
For each $v\!\in\!\cN_XV$, the connection~$\na$ determines an injective homomorphism
\BE{nhndfn_e}\nh_{\na;v}\!:T_{\pi(v)}V\lra T_v(\cN_XV)\EE
with the image complementary to the image of $\cN_XV$.
Let
\begin{equation*}\begin{split}
&T_\cR X(-\log V)=\Big(\!\!
\big(\!(\{\Psi^{-1}\}^{\!*}\pi^*TV)\!\oplus\!\Psi(\cN')\!\times\!\C\big)  
\!\sqcup\!T(X\!-\!V)\!\!\Big)\!\Big/\!\!\!\sim\,\lra \Psi(\cN')\!\cup\!(X\!-\!V)\!=\!X,\\
&~(\{\Psi^{-1}\}^{\!*}\pi^*TV)\!\oplus\!\Psi(\cN')\!\times\!\C
\ni\big(\Psi(v),v,w\big)\!\oplus\!\big(\Psi(v),c\big)
\sim\nd_v\Psi\big(\nh_{\na;v}(w)\!+\!cv)\in T(X\!-\!V).
\end{split}\end{equation*}
This defines a smooth vector bundle over~$X$.
The smooth bundle homomorphism~\eref{Ninc_e} defined~by
$$\io_{\cR}(\dot{x})=\begin{cases}
\nd_v\Psi\big(\nh_{\na;v}(w)\!+\!cv),&\hbox{if}~
\dot{x}\!=\!\big[(\Psi(v),v,w)\!\oplus\!(\Psi(v),c)\big],~v\!\in\!\cN';\\
\dot{x},&\hbox{if}~\dot{x}\!\in\!T(X\!-\!V);
\end{cases}$$
satisfies the properties stated in Theorem~\ref{TXV_thm}\ref{loddfn_it}.\\

\noindent
An almost complex structure $J$ on $V$ and the fiberwise complex structure $\fI_{\rho}$
on~$\cN_XV$ determine an almost complex structure~$J_{\cR}$ on the total space of~$\cN_XV$
via the connection~$\na$.
We call an almost complex structure~$J$ on~$X$ \sf{$\cR$-compatible} if
$J$~preserves $TV\!\subset\!TX|_V$ and $\Psi$ intertwines~$J$ and 
$J_{\cR}\!\equiv\!(J|_{TV})_{\cR}$, i.e.
$$J(TV)\subset TV \qquad\hbox{and}\qquad 
J\!\circ\!\nd\Psi=\nd\Psi\!\circ\!J_{\cR}\big|_{\cN'}\,.$$
Such an almost complex structure~$J$ induces a fiberwise complex structure $\fI_{\cR,J}$
on~$T_\cR X(-\log V)$ satisfying Theorem~\ref{TXV_thm}\ref{logJ_it}.
It can be constructed by pasting together $J_{\cR}\!\circ\!\{\nd\Psi\}^{-1}$
and an almost complex structure on $X\!-\!V$.\\

\noindent
We denote by $\Symp^+(X,V)$ the space of symplectic forms on~$X$ that restrict to
symplectic forms on~$V$,
by $\Aux(X,V)$ the space of pairs $(\om,\cR)$ consisting of \hbox{$\om\!\in\!\Symp^+(X,V)$}
and an $\om$-compatible regularization~$\cR$ for~$V$ in~$X$,
and by $\AK(X,V)$ the space of triples $(\om,\cR,J)$ consisting of \hbox{$(\om,\cR)\!\in\!\Aux(X,V)$}
and an almost complex structure~$J$ on~$X$ compatible with~$\om$ and~$\cR$.
Since the Symplectic Neighborhood Theorem can be applied with families of symplectic forms
parametrized by compact manifolds, the projection
\BE{Aux2Symp_e} \Aux(X,V)\lra\Symp^+(X,V), \qquad  (\om,\cR)\lra\om,\EE 
is a weak homotopy equivalence.
It is straightforward, by adapting the proof of \cite[Prp~4.1]{MS1}, for example,
to show that the projection
\BE{AK2Aux_e} \AK(X,V)\lra \Aux(X,V), \qquad  (\om,\cR,J)\lra(\om,\cR),\EE 
is also a weak homotopy equivalence.\\

\noindent
The above constructions of the complex line bundle $(\cO_X(V),\fI)$ and 
the vector bundle $TX(-\log V)$ can be applied with compact families in~$\Aux(X,V)$.
The construction of the complex vector bundle $(TX(-\log V),\fI_{\cR,J})$
can be applied with compact families in~$\AK(X,V)$.
Along with the previous paragraph, this confirms the statements of
 Proposition~\ref{OXV_prp}\ref{cOinv_it}	and Theorem~\ref{TXV_thm}\ref{loginv_it} 
for smooth symplectic divisors~$V$.\\
	
\noindent
The constructions of the  complex line bundle $(\cO_X(V),\fI)$ and 
the vector bundle $TX(-\log V)$ do not involve the symplectic form~$\om$ directly.
The first construction can be carried out for any closed codimension~2 submanifold~$V$ 
of a smooth manifold~$X$ endowed with a complex structure on the normal bundle~$\cN_XV$
and a smooth regularization~$\Psi$.
The constructions of the vector bundle~$TX(-\log V)$ and of the complex structure~$\fI_{\cR,J}$
on~it require in addition a connection on~$\cN_XV$ in the first case and 
also an $\cR$-compatible almost complex structure~$J$ on~$X$ in the second~case.\\

\noindent
Corollary~\ref{RadVecFld_crl} below is used later in this paper.
We deduce it from the following observation.

\begin{lmm}\label{RadVecFld_lmm}
Suppose $V$ is a smooth manifold, $\pi\!:\cN\!\lra\!V$ is a vector bundle, and 
$\na$ is a connection in~$\cN$.
Let $T\cN^{\hor}\!\subset\!T\cN$ be the horizontal tangent subbundle determined by~$\na$. 
If $\wt\na\!\equiv\!\pi^*\na$ is the connection in $\pi^*\cN\!\lra\!\cN$ determined by~$\na$, then
\BE{RadVecFld_e0a}
\wt\na\ze_{\cN}\big|_{T\cN^{\hor}}\!=\!0\!: T\cN^{\hor}\lra\pi^*\cN.\EE
If in addition $\cN$ is a complex vector bundle (and $\na$ is a complex linear connection), then
\BE{RadVecFld_e0b}
\wt\na\ze_{\cN}\big|_{T\cN^{\ver}}\!\circ\!\fI\!=\!\fI\wt\na\ze_{\cN}\big|_{T\cN^{\ver}}
\!:T\cN^{\ver}\lra\pi^*\cN.\EE
\end{lmm}

\begin{crl}\label{RadVecFld_crl}
Suppose $(V,J)$ is an almost complex manifold, $\pi\!:\cN\!\lra\!V$ is a complex line bundle,  
$\na$ is a connection in~$\cN$, and $\wt\na\!\equiv\!\pi^*\na$.
If 
$$\Phi\!: (\cN\!-\!V)\!\times\!\C \stackrel{\approx}{\lra}\pi^*\cN\big|_{\cN-V},  
\qquad \Phi(v,c)=(v,cv),$$
then $\Phi^*\wt\na\!-\!\nd$ is a $(1,0)$-form on~$\cN\!-\!V$
with respect to the almost complex structure~$J_{\na}$ on~$\cN\!-\!V$
determined by~$J$ and~$\na$.
\end{crl}

\begin{proof} The 1-form $\Phi^*\wt\na\!-\!\nd$ is given by
\BE{RadVecFld_e2}\big\{\!\Phi^*\wt\na\!-\!\nd\big\}1
=\Phi^{-1}\!\circ\!\wt\na(\Phi\!\circ\!1)
=\Phi^{-1}\!\circ\!\wt\na\ze_{\cN}.\EE
The almost complex structure~$J_{\na}$ restricts to $\{\nd\pi\}^*J$ on~$T\cN^{\hor}$
and to $\pi^*\fI$ on~$T\cN^{\ver}$.
The claim thus follows from Lemma~\ref{RadVecFld_lmm}.
\end{proof}

\begin{proof}[{\bf{\emph{Proof of Lemma~\ref{RadVecFld_lmm}}}}] 
Suppose $U$ is an open subset of~$V$ and $\xi_1,\ldots,\xi_n\!\in\!\Ga(U;\cN)$ is
a frame for~$\cN$ over~$U$. 
Let $\th^k_l\!\in\!\Ga(U;T^*U)$ be such that 
\BE{RadVecFld_e1}\na\xi_l=\sum_{k=1}^{k=n}\xi_k\th^k_l 
\equiv \sum_{k=1}^{k=n}\th^k_l\!\otimes\!\xi_k\in \Ga\big(U;T^*U\!\otimes\!\cN\big)
\quad\forall~l\!=\!1,\ldots,n.\EE
The frame $\xi_1,\ldots,\xi_n$ determines an identification $\cN|_U\!=\!U\!\times\!\R^n$
so~that 
\BE{RadVecFld_e3}\begin{split}
\ze_{\cN}(v)&=\sum_{l=1}^{l=n}c_l\xi_l(x),\\
T_v\cN^{\hor}&=\big\{\!
\big(\dot{x},-\sum_{l=1}^{l=n}c_l\th^1_l(\dot{x}),\ldots
-\sum_{l=1}^{l=n}c_l\th^n_l(\dot{x})\!\big)\!: \dot{x}\!\in\!T_xV\big\}
\end{split}
\quad\forall\,v\!\equiv\!(x,c_1,\ldots,c_n)\!\in\!\cN|_U;\EE
see the proof of \cite[Lemma~1.1]{anal}.\\

\noindent
For each $l\!=\!1,\ldots,n$, let $\wt\xi_l\!=\!\pi^*\xi_l\!\in\!\Ga(\cN|_U;\pi^*\cN)$.
By the definition of~$\wt\na$ and~\eref{RadVecFld_e1},
$$\wt\na\wt\xi_l=\sum_{k=1}^{k=n}(\th^k_l\!\circ\!\nd\pi)\!\otimes\!\wt\xi_k
\quad\forall~l\!=\!1,\ldots,n.$$
Thus,
\BE{RadVecFld_e5}
\wt\na\ze_{\cN}\big|_{(x,c_1,\ldots,c_n)}=
\sum_{l=1}^{l=n}
\sum_{k=1}^{k=n}c_l(\th^k_l\!\circ\!\nd\pi)\!\otimes\!\wt\xi_k
+\sum_{l=1}^{l=n}(\nd c_l)\!\otimes\!\wt\xi_l.\EE
Along with the second statement in~\eref{RadVecFld_e3}, this gives~\eref{RadVecFld_e0a}.\\

\noindent
The first summand on the right-hand side of~\eref{RadVecFld_e5} vanishes on~$T\cN^{\ver}$.
If $\cN$ is a complex vector bundle, the above applies with~$\R$ replaced by~$\C$.  
The second summand  on the right-hand side of~\eref{RadVecFld_e5} is $\C$-linear on~$T\cN^{\ver}$
in this case. 
This gives~\eref{RadVecFld_e0b}.
\end{proof}

\begin{rmk}\label{RadVecFld_rmk}
By the proof of Lemma~\ref{RadVecFld_lmm}, the 1-form in~\eref{RadVecFld_e2} is given~by
$$(\Phi^*\wt\na\!-\!\nd)\big|_{(x,z)}=\th_1^1|_x\!+\!\frac{\nd z}{z}\,.$$
Thus, the curvature $F^{\Phi^*\wt\na}$ of the connection $\Phi^*\wt\na$ 
on $(\cN\!-\!V)\!\times\!\C$ is given~by
$$F^{\Phi^*\wt\na}=\nd(\pi^*\th_1^1)=\pi^*F^{\na}\,.$$
\end{rmk}

\section{SC symplectic divisors}
\label{SC_sec}

\noindent
For $N\!\in\!\Z^{\ge0}$, let 
$$[N]=\{1,\ldots,N\}\,.$$
If $\cN\!\lra\!V$ is a vector bundle, $\cN'\!\subset\!\cN$, and $V'\!\subset\!V$, we define
\hbox{$\cN'|_{V'}\!=\!\cN|_{V'}\cap\cN'$}.

\subsection{Definitions}
\label{SCdfn_subs}

\noindent
Let $X$ be a (smooth) manifold. 
For a collection $\{V_i\}_{i\in S}$ of submanifolds of~$X$ and $I\!\subset\!S$, let
$$V_I\equiv \bigcap_{i\in I}\!V_i\subset X, \quad
V_{I,\partial}=\bigcup_{I'\supsetneq I}\!\! V_{I'}, \quad
V_I^{\circ}=V_{I}\!-\!V_{I,\partial}\,.$$
Such a collection 
is called \sf{transverse} if any subcollection $\{V_i\}_{i\in I}$ of these submanifolds
intersects transversely, i.e.~the homomorphism
\BE{TransVerHom_e}
T_xX\oplus\bigoplus_{i\in I}T_xV_i\lra \bigoplus_{i\in I}T_xX, \qquad
\big(v,(v_i)_{i\in I}\big)\lra (v\!+\!v_i)_{i\in I}\,,\EE
is surjective for all $x\!\in\!V_I$. 
By the Inverse Function Theorem,
each subspace $V_I\!\subset\!X$ is then a submanifold of~$X$ 
of codimension
$$\codim_XV_I=\sum_{i\in I}\codim_XV_i$$
and the homomorphisms 
\BE{cNorient_e2}
\begin{split}
\cN_XV_I\lra \bigoplus_{i\in I}\cN_XV_i\big|_{V_I}\quad&\forall~I\!\subset\!S,\qquad
\cN_{V_{I-i}}V_I\lra \cN_XV_i\big|_{V_I} \quad\forall~i\!\in\!I\!\subset\!S,\\
&\bigoplus_{i\in I-I'}\!\!\!\cN_{V_{I-i}}V_I\lra 
\cN_{V_{I'}}V_I \quad\forall~I'\!\subset\!I\!\subset\!S
\end{split}\EE
induced by inclusions of the tangent bundles are isomorphisms.\\

\noindent
As detailed in \cite[Section~2.1]{SympDivConf}, 
a transverse collection $\{V_i\}_{i\in S}$ of oriented submanifolds of
an oriented manifold~$X$
of even codimensions  induces an orientation on each submanifold $V_I\!\subset\!X$
with $|I|\!\ge\!2$; we call it \sf{the intersection orientation of~$V_I$}.
If $V_I$ is zero-dimensional, it is a discrete collection of points in~$X$
and the homomorphism~\eref{TransVerHom_e} is an isomorphism at each point $x\!\in\!V_I$;
the intersection orientation of~$V_I$ at $x\!\in\!V_I$
then corresponds to a plus or minus sign, depending on whether this isomorphism
is orientation-preserving or orientation-reversing.
We call the original orientations of 
$X\!=\!V_{\eset}$ and $V_i\!=\!V_{\{i\}}$ \sf{the intersection orientations}
of these submanifolds~$V_I$ of~$X$ with $|I|\!<\!2$.\\

\noindent
Suppose $(X,\om)$ is a symplectic manifold and $\{V_i\}_{i\in S}$ is a transverse collection 
of submanifolds of~$X$ such that each $V_I$ is a symplectic submanifold of~$(X,\om)$.
Each $V_I$ then carries an orientation induced by $\om|_{V_{I}}$,
which we call the \sf{$\om$-orientation}.
If $V_I$ is zero-dimensional, it is automatically a symplectic submanifold of~$(X,\om)$;
the $\om$-orientation of~$V_I$ at each point $x\!\in\!V_I$ corresponds to the plus sign 
by definition.
By the previous paragraph, the $\om$-orientations of~$X$ and~$V_i$ with $i\!\in\!I$
also induce intersection orientations on all~$V_I$.

\begin{dfn}\label{SCD_dfn}
Let $(X,\om)$ be a symplectic manifold.
A \sf{simple crossings} (or \sf{SC}) \sf{symplectic divisor} 
in~$(X,\om)$ is a finite transverse union 
$V\!=\!\bigcup_{i\in S}V_i$ of closed submanifolds of~$X$ of codimension~2 such that 
$V_I$ is a symplectic submanifold of~$(X,\om)$ for every $I\!\subset\!S$
and the intersection and $\om$-orientations of~$V_I$ agree.
\end{dfn}

\noindent
The singular locus $V_{\prt}\!\subset\!V$ of an SC symplectic divisor 
$V\!\subset\! X$ is the union
$$V_{\prt}\equiv \bigcup_{I\subset S, |I|\ge2}\!\!\!\!\!\! V_I~.$$
An SC symplectic divisor $V$ with $|S|\!=\!1$ is a smooth symplectic divisor in the usual sense. 
If $(X,\om)$ is a 4-dimensional symplectic manifold, 
a finite transverse union $V\!=\!\bigcup_{i\in S}V_i$ of closed symplectic submanifolds of~$X$ 
of codimension~2 is an SC symplectic divisor if all points of the pairwise intersections
$V_{i_1}\!\cap\!V_{i_2}$ with $i_1\!\neq\!i_2$ are positive. 
By \cite[Example~2.7]{SympDivConf}, it is not sufficient to consider 
the deepest (non-empty) intersections in higher~dimensions.

\begin{dfn}\label{SCdivstr_dfn}
Let $X$ be a manifold and $V\!=\!\bigcup_{i\in S}V_i$ be a finite transverse union of 
closed submanifolds of~$X$ of codimension~2.
A \sf{symplectic structure on $V$ in~$X$} is a symplectic form~$\om$ 
on~$X$ such that $V_I$ is a symplectic submanifold of $(X,\om)$ for all $I\!\subset\!S$.
\end{dfn}

\noindent
For $X$ and $\{V_i\}_{i\in S}$ as in Definition~\ref{SCdivstr_dfn}, 
we denote by $\Symp(X,\{V_i\}_{i\in S})$ the space of all symplectic structures 
on $\{V_i\}_{i\in S}$ in~$X$ and by 
$$\Symp^+\big(X,\{V_i\}_{i\in S}\big)\subset \Symp\big(X,\{V_i\}_{i\in S}\big)$$
the subspace of the symplectic forms~$\om$ such that $\{V_i\}_{i\in S}$
is an SC symplectic divisor in~$(X,\om)$.

\subsection{Regularizations}
\label{SCreg_subs}

\noindent
Let $V$ be a smooth manifold with a 2-form $\om$ and
$(L_i,\rho_i,\na^{(i)})_{i\in I}$ be a finite collection of Hermitian line bundles over~$V$.
If each $(\rho_i,\na^{(i)})$ is compatible with a fiberwise symplectic form~$\Om_i$ on~$L_i$
and
$$(\cN,\Om,\na)\equiv\bigoplus_{i\in I}\big(L_i,\Om_i,\na^{(i)}\big),$$
then the 2-form~\eref{ombund_e2} is given~by 
\BE{ombund_e}
\wt\om=\om_{(\rho_i,\na^{(i)})_{i\in I}}
\equiv  \pi^*\om+\frac12
\bigoplus_{i\in I} \pi_{I;i}^*\nd\big((\Om_i)_{\na^{(i)}}(\ze_{L_i},\cdot)\big),\EE
where $\pi_{I;i}\!:\cN\!\lra\!L_i$ is the component projection map.\\

\noindent
If in addition $\Psi\!:V'\!\lra\!V$ is a smooth map and 
$(L_i',\rho_i',\na'^{(i)})_{i\in I}$ is a finite collection of Hermitian line
bundles over~$V'$, we call a (fiberwise) vector bundle isomorphism
$$\wt\Psi\!: \bigoplus_{i\in I}L'_i\lra \bigoplus_{i\in I}L_i$$
covering~$\Psi$ a \sf{product Hermitian isomorphism} if
$$\wt\Psi\!: (L_i',\rho_i',\na'^{(i)}) \lra 
\Psi^*(L_i,\rho_i,\na^{(i)})$$
is an isomorphism of Hermitian line bundles over~$V'$ for every $i\!\in\!I$.\\

\noindent
If $V$ is a symplectic submanifold of a symplectic manifold $(X,\om)$,
we denote the restriction of~$\om|_{\cN_XV}$ to a subbundle $L\!\subset\!\cN_XV$
by~$\om|_L$.

\begin{dfn}\label{sympreg1_dfn}
Let $X$ be a  manifold, $V\!\subset\!X$ be a  submanifold, and
$$\cN_XV=\bigoplus_{i\in I}L_i$$
be a fixed splitting into oriented rank~2 subbundles. 
If $\om$ is a symplectic form on~$X$ such that $V$ is a symplectic submanifold
and $\om|_{L_i}$ is nondegenerate for every $i\!\in\!I$, then
an \sf{$\om$-regularization for~$V$ in~$X$} is a tuple $((\rho_i,\na^{(i)})_{i\in I},\Psi)$, 
where $(\rho_i,\na^{(i)})$ is an $\om|_{L_i}$-compatible Hermitian structure on~$L_i$
for each $i\!\in\!I$ and $\Psi$ is a regularization for~$V$ in~$X$
in the sense of Definition~\ref{smreg_dfn}, such that 
$$\Psi^*\om=\om_{(\rho_i,\na^{(i)})_{i\in I}}\big|_{\Dom(\Psi)}.$$
\end{dfn}

\vspace{.15in}

\noindent
Suppose $\{V_i\}_{i\in S}$ is a transverse collection of codimension~2 submanifolds of~$X$.
For each $I\!\subset\!S$, the last isomorphism in~\eref{cNorient_e2} with $I'\!=\!\eset$
provides 
a natural decomposition 
\BE{cNIIprdfn_e0}
\pi_I\!:\cN_XV_I\!=\!\bigoplus_{i\in I}\cN_{V_{I-i}}V_I  \lra V_I
\EE
of the normal bundle of~$V_I$ in~$X$ into oriented rank~2 subbundles. 
We take this decomposition as given for the purposes of applying Definition~\ref{sympreg1_dfn}.
If in addition $I'\!\subset\!I$, let
\BE{cNIIprdfn_e}\pi_{I;I'}\!:\cN_{I;I'}\equiv 
\bigoplus_{i\in I-I'}\!\!\!\cN_{V_{I-i}}V_I=\cN_{V_{I'}}V_I\lra V_I\EE
be the bundle projection.
There are canonical identifications
\BE{cNtot_e}
\cN_{I;I-I'}=\cN_XV_{I'}|_{V_I}, \quad
\cN_XV_I=\pi_{I;I'}^*\cN_{I;I-I'}=\pi_{I;I'}^*\cN_XV_{I'}
\qquad\forall~I'\!\subset\!I\!\subset\![N].
\EE
The first equality in the second statement above
is used in particular in~\eref{overlap_e}. 

\begin{dfn}\label{TransCollReg_dfn}
Let $X$ be a manifold and $\{V_i\}_{i\in S}$ be a transverse collection 
of submanifolds of~$X$.
A \sf{system of regularizations for}  $\{V_i\}_{i\in S}$ in~$X$ is a~tuple 
$(\Psi_I)_{I\subset S}$, where $\Psi_I$ is a regularization for~$V_I$ in~$X$
in the sense of Definition~\ref{smreg_dfn}, such~that
\BE{Psikk_e}
\Psi_I\big(\cN_{I;I'}\!\cap\!\Dom(\Psi_I)\big)=V_{I'}\!\cap\!\Im(\Psi_I)
\quad\hbox{and}\quad \Im(\Psi_I)\!\cap\!\Im(\Psi_J)=\Im(\Psi_{I\cup J})\EE
for all $I'\!\subset\!I\!\subset\!S$ and $J\!\subset\!S$.
\end{dfn}

\noindent
Given a system of regularizations as in Definition~\ref{TransCollReg_dfn}
and $I'\!\subset\!I\!\subset\!S$, let
\BE{wtPsiIIdfn_e0}
\cN_{I;I'}' = \cN_{I;I'}\!\cap\!\Dom(\Psi_I), \qquad
\Psi_{I;I'}\equiv \Psi_I\big|_{\cN_{I;I'}'}\!: \cN_{I;I'}'\lra V_{I'}\,.\EE
The map $\Psi_{I;I'}$ is a regularization for $V_I$ in~$V_{I'}$.
As explained in \cite[Section~2.2]{SympDivConf}, $\Psi_I$ determines
an isomorphism
\BE{wtPsiIIdfn_e}
 \fD\Psi_{I;I'}\!:  \pi_{I;I'}^*\cN_{I;I-I'}\big|_{\cN_{I;I'}'}
\lra \cN_XV_{I'}\big|_{V_{I'}\cap\Im(\Psi_I)}\EE
of vector bundles covering~$\Psi_{I;I'}$ and
respecting the natural decompositions of 
$\cN_{I;I-I'}\!=\!\cN_XV_{I'}|_{V_I}$ and $\cN_XV_{I'}$.
By the last assumption in Definition~\ref{smreg_dfn}, 
$$\fD\Psi_{I;I'}\big|_{\pi_{I;I'}^*\cN_{I;I-I'}|_{V_I}}\!=\!\id\!:\,
\cN_{I;I-I'}\lra \cN_XV_{I'}|_{V_I}$$
under the canonical identification of $\cN_{I;I-I'}$ with $\cN_XV_{I'}|_{V_I}$.

\begin{dfn}\label{TransCollregul_dfn}
Let $X$ be a manifold and  $\{V_i\}_{i\in S}$ be a transverse 
collection of submanifolds of~$X$. 
\begin{enumerate}[label=(\arabic*),leftmargin=*]

\item\label{SCregul_it}  A \sf{regularization for $\{V_i\}_{i\in S}$ in~$X$} 
is a system of regularizations $(\Psi_I)_{I\subset S}$ 
for $\{V_i\}_{i\in S}$ in~$X$ such~that
\BE{overlap_e}\begin{split}
&\Dom(\Psi_I)\subset \pi_{I;I'}^*\cN_{I;I-I'}\big|_{\cN_{I;I'}'}, \quad 
\fD\Psi_{I;I'}\big(\Dom(\Psi_I)\big)
=\Dom(\Psi_{I'})\big|_{V_{I'}\cap\Im(\Psi_I)}, \\
&\hspace{1.8in} \Psi_I=\Psi_{I'}\circ\fD\Psi_{I;I'}|_{\Dom(\Psi_I)}
\end{split}\EE
for all $I'\!\subset\!I\!\subset\!S$.

\item\label{sympregul_it} 
Suppose in addition that $V=\bigcup_{i\in S} V_i$ is an SC symplectic divisor in $(X,\om)$. 
An \sf{$\om$-regularization for $V$ in~$X$}  is a~tuple
$$(\cR_I)_{I\subset S} \equiv  
\big((\rho_{I;i},\na^{(I;i)})_{i\in I},\Psi_I\big)_{I\subset S}$$
such that $\cR_I$ is an $\om$-regularization for~$V_I$ in~$X$ for each $I\!\subset\!S$,
$(\Psi_I)_{I\subset S}$ is a regularization for $\{V_i\}_{i\in S}$ in~$X$,
and the induced vector bundle isomorphisms~\eref{wtPsiIIdfn_e}
are product Hermitian isomorphisms for all $I'\!\subset\!I\!\subset\!S$.

\end{enumerate}
\end{dfn}

\vspace{.15in}

\noindent
If $(\Psi_I)_{I\subset S}$ is a regularization for $\{V_i\}_{i\in S}$ in~$X$, 
then 
\BE{SCDcons_e2}\begin{split}
&\hspace{1in}
\cN_{I;I''}',\pi_{I;I''}^*\cN_{I;I-I''}|_{\cN_{I;I''}'}\subset
 \pi_{I;I'}^*\cN_{I;I-I'}\big|_{\cN_{I;I'}'},\\ 
&\Psi_{I;I''}=\Psi_{I';I''}\circ\fD\Psi_{I;I'}\big|_{\cN_{I;I''}'}\,, 
\quad
\fD\Psi_{I;I''}=\fD\Psi_{I';I''}\circ
\fD\Psi_{I;I'}\big|_{\pi_{I;I''}^*\cN_{I;I-I''}|_{\cN_{I;I''}'}}
\end{split}\EE
for all $I''\!\subset\!I'\!\subset\!I\!\subset\!S$.\\

\noindent
An almost complex structure $J$ on $X$ preserving $TV_I\!\subset\!TX|_{V_I}$ and 
an $\om$-regularization $\cR_I$ for~$V_I$ in~$X$ as in Definition~\ref{sympreg1_dfn}
determine an almost complex structure~$J_{\cR;I}$ on the total space of~$\cN_XV_I$ 
via the connection $\na^{(I)}\!\equiv\!\bigoplus_{i\in I}\na^{(I;i)}$.
We call an almost complex structure~$J$ on~$X$ \sf{compatible} with 
an $\om$-regularization $(\cR_I)_{I\subset S}$ 
as in Definition~\ref{TransCollregul_dfn}\ref{sympregul_it} 
if  
$$J(TV_I)\subset TV_I \quad\hbox{and}\quad 
J\!\circ\!\nd\Psi_I=\nd\Psi_I\!\circ\!J_{\cR;I}\big|_{\Dom(\Psi_I)}
\qquad\forall\,I\!\subset\!S\,.$$
The notion of regularization of Definition~\ref{TransCollregul_dfn}\ref{sympregul_it}
readily extends to families of symplectic forms;
see \cite[Definition~2.12(2)]{SympDivConf}.
We define the spaces $\Aux(X,V)$ of pairs $(\om,\cR)$ and $\AK(X,V)$ of triples $(\om,\cR,J)$
as in Section~\ref{smooth_subs}.
By \cite[Theorem~2.17]{SympDivConf}, the map~\eref{Aux2Symp_e} is a weak homotopy equivalence
in the present setting as well.
On the other hand, it is still straightforward to show that 
the map~\eref{AK2Aux_e} is also a weak homotopy equivalence in the present setting.

\subsection{Constructions}
\label{ConSC_subs}

\noindent
We now construct the bundles $\cO_X(V)$ and $TX(-\log V)$ for an SC symplectic divisor~$V$ 
in a symplectic manifold~$(X,\om)$.
We fix an $\om$-regularization~$\cR$ for $V$ in $X$ as in 
Definition~\ref{TransCollregul_dfn}\ref{sympregul_it}.
For the purposes of constructing a complex structure on $TX(-\log V)$, we also 
fix an almost complex structure~$J$ on~$X$ compatible with~$\om$ and~$\cR$.\\

\noindent
For $I'\!\subset\!I\!\subset\!S$, let $\pi_{I}$, $\pi_{I;I'}$, $\cN_{I;I'}'$, $\Psi_{I;I'}$,
and $\fD\Psi_{I;I'}$
be as in~\eref{cNIIprdfn_e0}, \eref{cNIIprdfn_e}, \eref{wtPsiIIdfn_e0}, and~\eref{wtPsiIIdfn_e}.
In what follows, we write an element $v_I\!\equiv\!(v_i)_{i\in I}$ of~$\cN_XV_{I}$ as
$$v_I=(v_{I;I'},v_{I;I-I'}) \quad\hbox{with}\quad
v_{I;I'}\!\equiv\!(v_i)_{i\in I-I'}\!\in\!\cN_{I;I'}
~~\hbox{and}~~  v_{I;I-I'}\!\equiv\!(v_i)_{i\in I'} \in \cN_{I;I-I'}.$$
We denote by $\na^{(I)}$ and $\na^{(I;I')}$
the connections on~$\cN_XV_{I}$ and~$\cN_{I;I'}$
induced by the connections~$\na^{(I;i)}$ on the direct summands of these vector bundles.
Let
$$\nh_{\na^{(I)};v_I}\!:T_{\pi_I(v_I)}V_I\lra T_{v_I}(\cN_XV_I)
\quad\hbox{and}\quad
\nh_{\na^{(I;I')};v_{I;I'}}\!:T_{\pi_{I;I'}(v_{I;I'})}V_I\lra T_{v_{I;I'}}\cN_{I;I'}$$
be the corresponding injective homomorphisms as in~\eref{nhndfn_e}.
Define
\begin{gather}\notag
\Pi_{I}\!: \cN_{X}V_I\lra\bigotimes_{i\in I}\cN_{V_{I-i}}\!V_{I}, \qquad
\Pi_{I}\big((v_i)_{i\in I}\big)=\bigotimes_{i\in I}v_i\,,\\
\label{UIcirc_e}
U_{I}^{\circ}=\Psi_{I}\big(\Dom(\Psi_{I})|_{V_I^{\circ}}\big)
=\Im\,\Psi_{I}-\!\bigcup_{J\not\subset I}\!\!V_J\,;
\end{gather}
the last equality follows from \eref{Psikk_e} and \eref{overlap_e}.
For every $I\!\subset\!S$, let
\BE{cOlofdfnloc_e}\begin{split}
\cO_{\cR;I}(V)&= \big\{\Psi_{I}^{-1}|_{U_{I}^{\circ}}\big\}^{\!*}
\pi_{I}^{\,*}\bigg(\!\bigotimes_{i\in I}\cN_{V_{I-i}}V_{I}\!\bigg)\lra U_{I}^{\circ},\\
T_{\cR}U_I^\circ (-\log V)&= 
\Big(\!\big\{\Psi_{I}^{-1}|_{U_{I}^{\circ}}\big\}^{\!*}\pi_{I}^{\,*}TV_I\Big) 
\!\oplus\! \big(\!U_{I}^{\circ}\!\times\!\C^{I}\!\big) \lra U_{I}^{\circ}.
\end{split}\EE
The complex structures $\fI_{\rho_{I;i}}$ on
\hbox{$\cN_{V_{I-i}}V_{I}\!=\!\cN_XV_i|_{V_I}$} encoded in~$\cR$ determine 
a complex structure on the complex line bundle~$\cO_{\cR;I}(V)$.
The almost complex structure~$J|_{TV_I}$ on~$V_I$ and the standard complex structure on~$\C^I$
determine a complex structure on the vector bundle~$T_{\cR}U_I^\circ (-\log V)$.\\

\noindent
Let $I'\!\subset\!I\!\subset\!S$.
If $x\!\in\!U_I^{\circ}\!\cap\!U_{I'}^{\circ}$, then 
\begin{gather*}
x=\Psi_I(v_I)=\Psi_{I'}\big(\fD\Psi_{I;I'}(v_{I;I'},v_{I;I-I'})\!\big)
\quad\hbox{with}\quad\\
v_I=\big(v_{I;I'},v_{I;I-I'}\big)\!\equiv\!
\big(\!(v_i)_{i\in I-I'},(v_i)_{i\in I'}\big)\!\in\!\cN_{I;I'}\!\oplus\!\cN_{I;I-I'}
~~\hbox{s.t.}~~v_i\!\neq\!0~\forall\,i\!\in\!I\!-\!I'.
\end{gather*}
Since $\fD\Psi_{I;I'}$ is a product Hermitian isomorphism, it follows that the~map
\begin{gather}\label{thetaII_e}
\th_{I'I}\!: \cO_{\cR;I}(V)\big|_{U_I^{\circ}\cap U_{I'}^{\circ}}
\lra \cO_{\cR;I'}(V)\big|_{U_I^{\circ}\cap U_{I'}^{\circ}},\\
\notag
\th_{I'I}\big(x,v_I,\Pi_I(v_{I;I'},w_{I;I-I'})\!\big)
=\Big(x,\fD\Psi_{I;I'}(v_I),
\Pi_{I'}\big(\fD\Psi_{I;I'}(v_{I;I'},w_{I;I-I'})\!\big)\!\Big),
\end{gather}
is a well-defined isomorphism of complex line bundles.
The~map
\begin{gather}\label{varTII_e}
\vt_{I'I}\!: T_{\cR}U_I^\circ (-\log V)\big|_{U_I^{\circ}\cap U_{I'}^{\circ}}\lra
T_{\cR}U_{I'}^\circ (-\log V)\big|_{U_I^{\circ}\cap U_{I'}^{\circ}},\\
\notag
\begin{split}
\vt_{I'I}\Big(\!(x,v_I,w)\!\oplus\!\big(x,(c_i)_{i\in I}\big)\!\!\Big)
=\Big(\!x,\fD\Psi_{I;I'}(v_I),\nd_{v_{I;I'}}\Psi_{I;I'}\!
\big(\nh_{\na^{(I;I')};v_{I;I'}}(w)\!+\!\!\sum_{i\in I-I'}\!\!\!c_iv_i\!\big)\!\!\Big)&\\
\!\oplus\!\big(x,(c_i)_{i\in I'}\big)&,
\end{split}
\end{gather}
is similarly a well-defined isomorphism of vector bundles.
Since $J$ is an $\cR$-compatible almost complex structure on~$X$, this isomorphism is $\C$-linear.
By~\eref{SCDcons_e2}, 
\BE{overlapcond_e1}\begin{split} 
\th_{I''I}\big|_{U_I^{\circ}\cap U_{I'}^{\circ}\cap U_{I''}^{\circ}}
&=\th_{I''I'}\big|_{U_I^{\circ}\cap U_{I'}^{\circ}\cap U_{I''}^{\circ}}\!\circ\!
\th_{I'I}\big|_{U_I^{\circ}\cap U_{I'}^{\circ}\cap U_{I''}^{\circ}},\\
\vt_{I''I}\big|_{U_I^{\circ}\cap U_{I'}^{\circ}\cap U_{I''}^{\circ}}
&=\vt_{I''I'}\big|_{U_I^{\circ}\cap U_{I'}^{\circ}\cap U_{I''}^{\circ}}\!\circ\!
\vt_{I'I}\big|_{U_I^{\circ}\cap U_{I'}^{\circ}\cap U_{I''}^{\circ}}
\end{split}\EE
for all $I''\!\subset\!I'\!\subset\!I$.\\

\noindent
Let $I,K\!\subset\!S$.
By~\eref{Psikk_e} and~\eref{overlap_e}, $U_I^{\circ}\!\cap\!U_K^{\circ}\!\subset\!U_{I\cup K}^{\circ}$.
If $I\!\not\subset\!K$, the~maps
\begin{equation*}\begin{split}
\th_{IK}\!=\!\th_{I(I\cup K)}|_{U_I^{\circ}\cap U_K^{\circ}}\!\circ\!
\th_{K(I\cup K)}^{~-1}|_{U_I^{\circ}\cap U_K^{\circ}}\!:
\cO_{\cR;K}(V)\big|_{U_I^{\circ}\cap U_K^{\circ}}
&\lra \cO_{\cR;I}(V)\big|_{U_I^{\circ}\cap U_K^{\circ}},\\
\vt_{IK}\!=\!\vt_{I(I\cup K)}|_{U_I^{\circ}\cap U_K^{\circ}}\!\circ\!
\vt_{K(I\cup K)}^{~-1}|_{U_I^{\circ}\cap U_K^{\circ}}\!:
T_{\cR}U_K^\circ (-\log V)\big|_{U_I^{\circ}\cap U_K^{\circ}}&\lra
T_{\cR}U_I^\circ (-\log V)\big|_{U_I^{\circ}\cap U_K^{\circ}}
\end{split}\end{equation*}
are thus well-defined isomorphisms of complex vector bundles.
By~\eref{overlapcond_e1}, the collections $\{\th_{IK}\}_{I,K\subset S}$ 
and $\{\vt_{IK}\}_{I,K\subset S}$
satisfy the cocycle condition.
The first collection thus determines a complex line bundle
\BE{COXVSC_e}\begin{split}
&\pi\!: \cO_{\cR;X}(V)\!=\!\bigg(\bigsqcup_{I\subset S}
\cO_{\cR;I}(V)\!\!\bigg)\!\Big/\!\!\sim\,\,\lra X,\quad
\pi\big([x,v_I,\Pi_I(w_I)]\big)=x,\\
&\hspace{.2in}
\cO_{\cR;I}(V)\big|_{U_I^{\circ}\cap U_K^{\circ}}
\ni \th_{IK}(v)\sim v\in \cO_{\cR;K}(V)\big|_{U_I^{\circ}\cap U_K^{\circ}} 
\quad \forall~I,K\subset\!S.
\end{split}\EE
The second collection similarly determines a complex vector bundle
\BE{LogTXVSCasQ_e}\begin{split}
&\pi\!: T_\cR X(-\log V)\!=\!\bigg(
\bigsqcup_{I\subset S}T_{\cR}U_I^\circ (-\log V)\!\!\bigg)\!\Big/\!\!\sim\,\,\lra X,\quad
\pi\Big(\!\big[\!(x,v_I,w)\!\oplus\!\big(x,(c_i)_{i\in I}\big)\!\big]\!\Big)=x,\\
&\hspace{.4in}
T_{\cR}U_I^\circ (-\log V)\big|_{U_I^{\circ}\cap U_K^{\circ}}
\ni \vt_{IK}(v)\sim v\in
T_{\cR}U_K^\circ (-\log V)\big|_{U_I^{\circ}\cap U_K^{\circ}}
\quad \forall~I,K\subset\!S.
\end{split}\EE

\vspace{.2in}

\noindent
The smooth section~$s_{\cR}$ of the complex line bundle~\eref{COXVSC_e} given~by
\BE{LS_e} s_{\cR}(x)=\big[x,v_I,\Pi_I(v_I)\big]
\quad\forall~x\!=\!\Psi_I(v_I)\!\in\!U_I^{\circ},~I\!\subset\!S,\EE
satisfies the properties stated in Proposition~\ref{OXV_prp}\ref{cOdfn_it}.
The smooth bundle homomorphism~\eref{Ninc_e} defined~by
\BE{ioRSC_e}\io_{\cR}\Big(\!\big[(x,v_I,w)\!\oplus\!\big(x,(c_i)_{i\in I}\big)\!\big]\!\Big)
=\nd_{v_I}\Psi_I\!
\Big(\nh_{\na^{(I)};v_I}(w)\!+\!\!\sum_{i\in I}c_iv_i\!\Big)
\quad\forall~x\!=\!\Psi_I(v_I)\!\in\!U_I^{\circ},~I\!\subset\!S,\EE
satisfies the properties stated in Theorem~\ref{TXV_thm}\ref{loddfn_it}\ref{logJ_it}.
By the same reasoning as at the end of Section~\ref{smooth_subs},
the bundles~\eref{COXVSC_e} and~\eref{LogTXVSCasQ_e} also satisfy
the properties in Proposition~\ref{OXV_prp}\ref{cOinv_it}	and Theorem~\ref{TXV_thm}\ref{loginv_it}.
Along with Proposition~\ref{OXV_prp}\ref{cOinv_it}, 
Lemma~\ref{cOsplitNC_lmm} below implies the claim of Proposition~\ref{OXV_prp}\ref{cOsplit_it} 
for SC symplectic divisors.

\begin{lmm}\label{cOsplitNC_lmm}
Suppose $V$ and $V'$ are SC symplectic divisors in a symplectic manifold $(X,\om)$ so that 
$V\!\cup\!V'\!\subset\!X$ is also an SC symplectic divisor
and $V\!\cap\!V'$ contains no open subspace of~$V$.
An $\om$-regularization
$$\wt\cR \equiv \big((\rho_{I;i},\na^{(I;i)})_{i\in I},\wt\Psi_I\big)_{I\subset S\sqcup S'}$$
for~$V\!\cup\!V'$ in~$X$
as in Definition~\ref{TransCollregul_dfn}\ref{sympregul_it} determines
$\om$-regularizations~$\cR$ for~$V$ in~$X$ and $\cR'$ for~$V'$ in~$X$
and an isomorphism
\BE{cOsplitNC_e}\psi_{\cR\cR'}\!:
\big(\cO_{\wt\cR;X}(V\!\cup\!V'),\fI_{\wt\cR}\big)\lra
\big(\cO_{\cR;X}(V),\fI_{\cR}\big)\!\otimes\!\big(\cO_{\cR';X}(V'),\fI_{\cR'}\big)\EE
natural with respect to the restrictions to the open subsets of~$X$.
\end{lmm}

\begin{proof}
Let $V\!=\!\bigcup_{i\in S}\!V_i$ and $V'\!=\!\bigcup_{i\in S'}\!V_i'$.
If $\wt\cR\!=\!(\cR_I)_{I\subset S\sqcup S'}$, then
$\cR\!\equiv\!(\cR_I)_{I\subset S}$ is an $\om$-regularization for~$V$ in~$X$
and $\cR'\!\equiv\!(\cR_I)_{I\subset S'}$ is an $\om$-regularization for~$V'$ in~$X$.
By the assumptions,
$$\wt{V}_{I\sqcup K}\!\equiv\!(V\!\cup\!V')_{I\sqcup K}=V_I\!\cap\!V_K', \qquad
\cN_X\wt{V}_{I\sqcup K}
=\cN_XV_I\big|_{\wt{V}_{I\sqcup K}}\!\oplus\!\cN_XV_K'\big|_{\wt{V}_{I\sqcup K}}
\quad\forall~I\!\subset\!S,~K\!\subset\!S'.$$

\vspace{.15in}

\noindent
For $I'\!\subset\!I\!\subset\!S'$ (resp.~$I'\!\subset\!I\!\subset\!S\!\sqcup\!S'$),
we denote by $U_I'^{\circ}\!\subset\!X$ and~$\th'_{I'I}$
(resp.~$\wt{U}_I^{\circ}\!\subset\!X$ and~$\wt\th_{I'I}$)
the analogues of the open subsets~$U_I^{\circ}$ and  the transition maps~$\th_{I'I}$
in~\eref{thetaII_e} for the trivialization~$\cR'$ (resp.~$\wt\cR$).
Thus,
$$U_I^{\circ}=\bigcup_{K\subset S'}\!\!\!\wt{U}_{I\sqcup K}^{\circ}~~\forall~I\!\subset\!S
\qquad\hbox{and}\qquad
U_K'^{\circ}=\bigcup_{I\subset S}\!\!\wt{U}_{I\sqcup K}^{\circ}~~\forall~K\!\subset\!S'.$$
For  $I\!\subset\!S$ and $K\!\subset\!S'$, define
\begin{gather*}
\psi_{\cR\cR';IK}\!:\cO_{\wt\cR;I\sqcup K}(V\!\cup\!V')
\lra \cO_{\cR;I}(V)\big|_{\wt{U}_{I\sqcup K}^{\circ}}
\!\otimes_{\C}\!\cO_{\cR';K}(V')\big|_{\wt{U}_{I\sqcup K}^{\circ}},\\
\begin{split}
\psi_{\cR\cR';IK}\big(x,v_{I\sqcup K},\Pi_I(w_{I\sqcup K})\!\big)
&=\Big(x,\fD\wt\Psi_{I\sqcup K;I}(v_{I\sqcup K}),
\Pi_I\big(\fD\wt\Psi_{I\sqcup K;I}(w_{I\sqcup K})\!\big)\!\!\Big)\\
&\hspace{.7in}\!\otimes\!\Big(x,\fD\wt\Psi_{I\sqcup K;K}(v_{I\sqcup K}),
\Pi_K\big(\fD\wt\Psi_{I\sqcup K;K}(w_{I\sqcup K})\!\big)\!\!\Big).\\
\end{split}
\end{gather*}  
By~\eref{SCDcons_e2} for the regularization~$\wt\cR$, 
\begin{equation*}\begin{split}
&\big\{\vt_{I'I}\!\otimes\!\vt_{K'K}'\big\}\!\circ\!\psi_{\cR\cR';IK}
\!=\!\psi_{\cR\cR';I'K'}\!\circ\!\wt\vt_{(I'\sqcup K')(I\sqcup K)}\!:\\
&\hspace{.5in}
\cO_{\wt\cR;I\sqcup K}(V\!\cup\!V')
\big|_{\wt{U}_{I\sqcup K}^{\circ}\cap\wt{U}_{I'\sqcup K'}^{\circ}}
\lra 
\cO_{\cR;I'}(V)\big|_{\wt{U}_{I\sqcup K}^{\circ}\cap\wt{U}_{I'\sqcup K'}^{\circ}}
\!\otimes_{\C}\!\cO_{\cR';K'}(V')\big|_{\wt{U}_{I\sqcup K}^{\circ}\cap\wt{U}_{I'\sqcup K'}^{\circ}}
\end{split}\end{equation*}
for all $I'\!\subset\!I\!\subset\!S$ and $K'\!\subset\!K\!\subset\!S'$.
Along with~\eref{COXVSC_e}, this implies that 
the collection $\{\psi_{\cR\cR';IK}\}$ 
induces a well-defined bundle homomorphism~\eref{cOsplitNC_e}.
\end{proof}

\subsection{Proof of Theorem~\ref{TXV_thm}\ref{logsplit_it}}
\label{logsplit_subs}

\noindent
We now establish~\eref{logsplit_e} under the assumption that $V$ is an SC symplectic divisor.
The reasoning in the general case is identical; see the end of Section~\ref{ConLNC_subs}.
As noted after Theorem~\ref{TXV_thm}, \eref{logsplit_e} implies~\eref{cTXV_e}
if $V$ is an SC symplectic divisor.\\

\noindent
Suppose $X$ is a smooth manifold, $V\!\subset\!X$ is a submanifold, and
$\psi\!:E\!\lra\!F$ is a homomorphism of complex vector bundles over~$X$ that vanishes 
on a complex subbundle $L'\!\subset\!E|_V$.
The restriction of~$\psi$ to an extension of~$L'$ to a complex subbundle~$L$ of~$E$ 
over a neighborhood~$U$
of $V\!\subset\!X$ then determines a section~$\psi_{L'}$ of 
the complex vector bundle $L^*\!\otimes_{\C}\!F$ over~$U$ that vanishes on~$V$
and well-defined derivative bundle homomorphisms
\BE{nDpsiLdfn_e}
\nD_x\psi_{L'}\!:\cN_XV|_x\lra L_x'^*\!\otimes_{\C}\!(\cok\,\psi_x)
\quad\hbox{and}\quad \nD_x\psi_{L'}\!:\cN_XV|_x\!\otimes_{\R}\!L_x'\lra \cok\,\psi_x\EE
for each $x\!\in\!V$; these homomorphisms do not depend on the extension of~$L'$.
Lemma~\ref{bundleisom_lmm} below, proved at the end of this section, is the key topological input 
we~use to establish~\eref{logsplit_e}. 
 
\begin{lmm}\label{bundleisom_lmm}
Suppose $X$ is a smooth manifold, $V\!\subset\!X$ is a closed submanifold of codimension~2
with a complex structure on~$\cN_XY$, 
$\psi\!:E\!\lra\!F$ is a homomorphism of complex vector bundles over~$X$,
and $L\!\lra\!X$ is a complex vector bundle so that 
$$
L|_V\!=\!\ker\!\big(\psi\!:E|_V\!\lra\!F|_V\big)\subset E|_V.$$
If $\psi$ is an isomorphism over~$X\!-\!V$, 
the first homomorphism in~\eref{nDpsiLdfn_e} with $L'\!=\!L|_V$ is $\C$-linear
for every $x\!\in\!V$, 
and the second homomorphism \eref{nDpsiLdfn_e} is surjective,
 then there exists an isomorphism
\BE{wtpsidfn_e}\wt\psi\!: E\oplus\cO_X(V)\!\otimes_{\C}\!L \lra F\!\oplus\!L\EE
of complex vector bundles over~$X$.
\end{lmm}

\begin{lmm}\label{ioRRpr_lmm}
Let $V$ and $V'\!\equiv\!\bigcup_{i\in S'}\!V_i'$, $\wt\cR$, and $\cR$ 
be as in Lemma~\ref{cOsplitNC_lmm}.
There exists a unique vector bundle homomorphism
\BE{iocRcRpr_e}\io_{\cR\wt\cR}\!:
T_{\wt\cR}X\big(\!-\log(V\!\cup\!V')\!\big)\lra T_\cR X(-\log V) \EE
so that $\io_{\wt\cR}\!=\!\io_{\cR}\!\circ\!\io_{\cR\wt\cR}$.
For every $K\!\subset\!S'$, 
\begin{enumerate}[label=(\arabic*),leftmargin=*]

\item
the kernel of~$\io_{\cR\wt\cR}$ over $V_K'^{\circ}$ 
is a trivial subbundle $L_K'\!\approx\!V_K'^{\circ}\!\times\!\C^K$;

\item\label{ioRRprcom_it}
the cokernel of~$\io_{\cR\wt\cR}$ over $V_K'^{\circ}$
is canonically isomorphic to $\cN_XV_K'^{\circ}$;

\item the composition of the above isomorphism with 
the first homomorphism in~\eref{nDpsiLdfn_e} with \hbox{$\psi\!=\!\io_{\cR\wt\cR}$} and 
$L'\!=\!L_K'$,
$$\cN_XV_K'^{\circ}|_x \xlra{\nD_x(\io_{\cR\wt\cR})_{L_K'}}
 (L_K')_x^*\!\otimes_{\C}\!\big(\cok\,\io_{\cR\wt\cR}\big)_{\!x} \xlra{~\approx~}
(L_K')_x^*\!\otimes_{\C}\!\cN_XV_K'^{\circ}|_x\,,$$
is $\C$-linear with respect to the complex structures on 
its domain and target determined by the regularization~$\wt\cR$ for every 
$x\!\in\!V_K'^{\circ}$;

\item the second homomorphism in~\eref{nDpsiLdfn_e} with 
\hbox{$\psi\!=\!\io_{\cR\wt\cR}$} and $L'\!=\!L_K'$,
$$\nD_x(\io_{\cR\wt\cR})_{L_K'}\!: 
\cN_XV_K'^{\circ}|_x\!\otimes_{\R}\!L_K'|_x\lra \big(\cok\,\io_{\cR\wt\cR}\big)_{\!x}$$
is surjective for every  \hbox{$x\!\in\!V_K'^{\circ}$}.

\end{enumerate}
\end{lmm}

\vspace{.15in}

\noindent
Suppose $J$ is an almost complex structure on~$X$ compatible with~$\wt\cR$ (and thus with~$\cR$).
Since \hbox{$\io_{\wt\cR}\!=\!\io_{\cR}\!\circ\!\io_{\cR\wt\cR}$}
and the homomorphisms $\io_{\wt\cR}$ and $\io_{\cR}$ are $\C$-linear isomorphisms over $X\!-\!V\!\cup\!V'$,
the homomorphism~\eref{iocRcRpr_e} is $\C$-linear with respect 
to the complex structures~$\fI_{\wt\cR,J}$ on the domain and~$\fI_{\cR,J}$ on the target. 
The isomorphism in Lemma~\ref{ioRRpr_lmm}\ref{ioRRprcom_it} 
is $\C$-linear with respect to the complex structures
on its target and domain determined by~$\wt\cR$.
If $V'\!\subset\!X$ is a smooth symplectic divisor, the bundle homomorphism 
$\psi\!=\!\io_{\cR\wt\cR}$ thus satisfies the conditions of Lemma~\ref{bundleisom_lmm}
with~$V$ replaced by~$V'$ and $L\!=\!X\!\times\!\C$.
The conclusion of this lemma is the claim of~\eref{logsplit_e}.

\begin{proof}[{\bf{\emph{Proof of Lemma~\ref{ioRRpr_lmm}}}}] 
We continue with the notation in the proof of Lemma~\ref{cOsplitNC_lmm}
and denote by~$\wt\vt_{I'I}$ the analogues of the transition maps~$\vt_{I'I}$
in~\eref{varTII_e} for the regularization~$\wt\cR$. 
For  $I\!\subset\!S$ and $K\!\subset\!S'$, define
\begin{gather*}
\io_{\cR\wt\cR;IK}\!:T_{\wt\cR}\wt{U}_{I\sqcup K}^{\circ}\big(\!-\log(V\!\cup\!V')\!\big)
\lra T_{\cR}U_I^\circ (-\log V)\big|_{\wt{U}_{I\sqcup K}^{\circ}}, \\
\begin{split}
&\io_{\cR\wt\cR;IK}\Big(\!(x,v_{I\sqcup K},w)\!\oplus\!
\big(x,(c_i)_{i\in I\sqcup K}\big)\!\!\Big)\\
&\hspace{.5in}=\Big(\!x,\fD\wt\Psi_{I\sqcup K;I}(v_{I\sqcup K}),
\nd_{v_{I\sqcup K;I}}\wt\Psi_{I\sqcup K;I}
\big(\nh_{\na^{(I\sqcup K)};v_{I\sqcup K}}(w)\!+\!\sum_{i\in K}c_iv_i\!\big)\!\!\Big)
\!\oplus\!\big(x,(c_i)_{i\in I}\big).
\end{split}
\end{gather*}  
By~\eref{SCDcons_e2} for the regularization~$\wt\cR$, 
\begin{equation*}\begin{split}
&\vt_{I'I}\!\circ\!\io_{\cR\wt\cR;IK}\!=\!\io_{\cR\wt\cR;I'K'}
\!\circ\!\wt\vt_{(I'\sqcup K')(I\sqcup K)}\!:
T_{\wt\cR}\wt{U}_{I\sqcup K}^\circ\big(\!-\log(V\!\cup\!V')\!\big)
\big|_{\wt{U}_{I\sqcup K}^{\circ}\cap\wt{U}_{I'\sqcup K'}^{\circ}}\\
&\hspace{3.5in}\lra T_{\cR}U_{I'}^\circ (-\log V)
\big|_{\wt{U}_{I\sqcup K}^{\circ}\cap\wt{U}_{I'\sqcup K'}^{\circ}}
\end{split}\end{equation*}
for all $I'\!\subset\!I\!\subset\!S$ and $K'\!\subset\!K\!\subset\!S'$.
Along with~\eref{LogTXVSCasQ_e}, this implies that
 the collection $\{\io_{\cR\wt\cR;IK}\}$ 
induces a well-defined bundle homomorphism~\eref{iocRcRpr_e}.\\

\noindent
It is immediate from the definitions and \eref{SCDcons_e2} for the regularization~$\wt\cR$
that \hbox{$\io_{\wt\cR}\!=\!\io_{\cR}\!\circ\!\io_{\cR\wt\cR}$}.
If $I\!\subset\!S$, $K\!\subset\!S'$, and $x\!\in\!V_K'\!\cap\!\wt{U}_{I\sqcup K}^{\circ}$
with $x\!=\!\wt\Psi_{I\cup K}(v_{I\sqcup K;K})\!=\!\Psi_I(v_{I\sqcup K;K})$, then
$$\ker(\io_{\cR\wt\cR;IK})_x=
\big\{\!(x,v_{I\sqcup K;K},0)\!\big\}\!\oplus\!\big(\{x\}\!\times\!\{0\}^I\!\times\!\C^K\big).$$ 
The homomorphism
\begin{gather*}
T_{\cR}U_I^{\circ}(-\log V)\big|_{V_K'\cap \wt{U}_{I\sqcup K}^{\circ}}\lra 
\cN_XV_K'^{\circ}\big|_{V_K'\cap \wt{U}_{I\sqcup K}^{\circ}},\\
 (x,v_{I\sqcup K;K},w)\!\oplus\!\big(x,(c_i)_{i\in I\sqcup K}\big)
\lra \nd_{v_{I\sqcup K;K}}\Psi_I\!\big(h_{\na^{(I)};v_{I\sqcup K;K}}(w)\!\big)\!+\!
T_xV_K'^{\circ},
\end{gather*}
induces an isomorphism $\cok(\io_{\cR\wt\cR})_x\!\lra\!\cN_XV_K'^{\circ}|_x$.
The composition of this isomorphism with
the second homomorphism in~\eref{nDpsiLdfn_e} with $\psi\!=\!\io_{\cR\wt\cR}$ and 
$L'\!=\!L_K'$ is given~by
$$\nD_x(\io_{\cR\wt\cR})_{L_K'}\big(\!(w_i)_{i\in K}\!\otimes\!\big(x,(c_i)_{i\in K}\big)\!\big)
=\sum_{i\in K}c_iw_i\,.$$
In particular, this homomorphism is surjective, as claimed.
\end{proof}

\begin{proof}[{\bf{\emph{Proof of Lemma~\ref{bundleisom_lmm}}}}] 
Let $\Psi\!:\cN'\!\lra\!X$ be a regularization for $V$ in~$X$ as in Definition~\ref{smreg_dfn} and
\hbox{$U\!=\!\Psi(\cN')$}.
Let $\cO_X(V)\!\lra\!X$ be the complex line bundle~\eref{OXVSmooth_e}
and $E'\!\subset\!E|_V$ be a complex subbundle complementary to~$L|_V$.
Extend~$E'$ to a subbundle of~$E|_U$, which we still denote by~$E'$.
By the assumptions of the lemma, $\psi$ is injective on $E'$ and $\psi(E')\!\subset\!F|_U$
is a complex subbundle.
Let $Q\!\subset\!F|_U$ be a complex subbundle complementary to~$\psi(E')$.
After shrinking~$\cN'$ if necessary, we can extend the identification of~$L|_V$ with a subbundle
of~$E|_V$ to an identification of $L|_U$ with a subbundle of~$E|_U$ so that 
$\psi(L|_U)\!\subset\!Q$.
Since $\psi$ vanishes on~$L|_V$, the derivative of the associated section of~$L^*|_U\!\otimes_{\C}\!Q$
induces a homomorphism
$$\nD\psi_L\!:\cN_XV\lra L^*|_V\!\otimes_{\C}\!Q|_V\subset (L^*\!\otimes_{\C}\!F)\big|_V$$
of real vector bundles over~$V$.
By the assumptions on~\eref{nDpsiLdfn_e} with $L'\!=\!L|_V$, this homomorphism is $\C$-linear
and induces an isomorphism
$$\nD\psi_L\!:\cN_XV\!\otimes_{\C}\!L|_V\lra Q|_V\subset F|_V$$
of an isomorphism of complex vector bundles over~$V$.
Using parallel transport, we identify $L|_U$ and~$Q$
with the vector bundles $\{\Psi^{-1}\}^*\pi^*L$ and $\{\Psi^{-1}\}^*\pi^*Q$, respectively.\\

\noindent
Choose a smooth function $\be\!:X\!\lra\![0,1]$  such~that
$\be\!=\!1$ on a neighborhood $U'\!\subset\!U$ of~$V$ and \hbox{$\supp\,\be\!\subset\!U$}.
We define~\eref{wtpsidfn_e} over $X\!-\!\supp\,\be\!\subset\!X\!-\!V$ and 
on $E|_U\!=\!E'\!\oplus\!(L|_U)$ by
\begin{gather*}
\wt\psi(e,w)=\big(\psi(e),w\big)
\quad\forall~(e,w)\in
\big(E\!\oplus\!\cO_X(V)\!\otimes_{\C}\!L\big)\big|_{X-\supp\,\be}\!=\!
\big(E\!\oplus\!L\big)\big|_{X-\supp\,\be},\\
\wt\psi(e'\!+\!e'',0)=\big(\psi(e')\!+\!\psi(e''),-\be(x)e''\big)
\quad \forall~e'\!\in\!E'_x,~e''\!\in\!L_x,~x\!\in\!U.
\end{gather*}
We extend this definition to $(\cO_X(V)\!\otimes_{\C}\!L)|_U$ by 
\begin{equation*}\begin{split}
\wt\psi\big(0,(x,v,w)\!\big)&=
\big(x,v,\big(\be(x)D\psi_L(v\!\otimes\!w),(1\!-\!\be(x)\!)w\big)\!\big)
\in \{\Psi^{-1}\}^*\pi^*(Q\!\oplus\!L)\\
&\hspace{1in}
\forall~(x,v,w)\!\in\!\big(\cO_X(V)\!\otimes_{\C}\!L\big)|_{U-V}
=\{\Psi^{-1}\}^*\pi^*L\big|_{U-V},\\
\wt\psi\big(0,(x,v,w)\!\big) &=\big(x,v,\big(\be(x)D\psi_L(w),0\big)\!\big)
\in \{\Psi^{-1}\}^*\pi^*(Q\!\oplus\!L)\\
&\hspace{1in}
\forall~(x,v,w)\!\in\!\big(\cO_X(V)\!\otimes_{\C}\!L\big)\big|_{U'}
\!=\!\big\{\Psi^{-1}\big\}^*\pi^*(\cN_XV\!\otimes\!L)\big|_{U'}.
\end{split}\end{equation*}
The bundle homomorphism~$\wt\psi$ is well-defined and smooth.
It remains to verify that the homomorphism
\BE{wtpsicomp_e}\wt\psi\!:
\big(L\!\oplus\!\cO_X(V)\!\otimes_{\C}\!L\big)|_U
\lra   Q\!\oplus\!(L|_U)\EE
is injective over $\supp\,\be\!\subset\!U$ if $\supp\,\be$ is sufficiently small.\\

\noindent
Over $U'$, the homomorphism~\eref{wtpsicomp_e} is given~by
\begin{equation*}\begin{split}
&\hspace{.3in}
\wt\psi\!:\big\{\Psi^{-1}\big\}^*\pi^*\!\big(L\!\oplus\!\cN_XV\!\otimes\!L\big)\big|_{U'}
\lra \big\{\Psi^{-1}\big\}^*\pi^*\!(Q\!\oplus\!L),\\
&\wt\psi\big(x,v,(e'',v'\!\otimes\!w)\!\big)=
\big(\psi(x,v,e''),0\big)\!+\!
\big(x,v,\big(D\psi_L(v'\!\otimes\!w),-e''\big)\!\big),
\end{split}\end{equation*}
and is thus injective.
Over $\supp\,\be\!-\!V\!\subset\!U\!-\!V$, \eref{wtpsicomp_e} is given~by
\BE{wtpsicomp_e3}\begin{split}
&\hspace{.6in}
\wt\psi\!:\big\{\Psi^{-1}\big\}^*\pi^*\!\big(L\!\oplus\!L\big)
\big|_{\supp\,\be-V}\lra \big\{\Psi^{-1}\big\}^*\pi^*\!(Q\!\oplus\!L),\\
&\wt\psi\big(x,v,(e'',w)\!\big)=
\big(\psi(x,v,e''),0\big)\!+\!
\big(x,v,\big(\be(x)D\psi_L(v\!\otimes\!w),-\be(x)e''\!+\!(1\!-\!\be(x)\!)w\big)\!\big).
\end{split}\EE
Choose norms on $L|_U$ and~$Q$.
Let $\cC\!:V\!\lra\!\R^+$ and
$\ve\!:(U,V)\!\lra\!(\R,0)$ be smooth functions such~that
\begin{equation*}\begin{split}
\cC\big(\pi(v)\!\big)\big|\nD\psi_L(v\!\otimes\!w)\big|&\ge|v||w|, \\
\big|\psi(\Psi(v),v,w)\!-\!(\Psi(v),v,\nD\psi_L(v\!\otimes\!w)\!)\big|&\le \ve(v)||v||w|
\end{split}
\qquad\forall~(v,w)\!\in\!\cN'\!\oplus\!L|_V\,.
\end{equation*}
If $(x,v,(e'',w)\!)$ with $x\!=\!\Psi(v)\!\in\!\supp\,\be\!-\!V$
lies in the kernel of~\eref{wtpsicomp_e3}, then
\begin{gather*}
\be(x)e''=(1\!-\!\be(x)\!)w, \quad 
D\psi_L\big(v\!\otimes\!(\be(x)w\!+\!e'')\!\big)
=D\psi_L\big(v\!\otimes\!e''\big)-\psi(x,v,e''),\\
|v|\big|\be(x)w\!+\!e''\big|\le 
\cC\big(\pi(v)\!\big)\big|\nD\psi_L\big(v\!\otimes\!(\be(x)w\!+\!e'')\!\big)\!\big|
\le \cC\big(\pi(v)\!\big)\ve(v)||v||e''|.
\end{gather*}
If $\cC(\pi(v))\ve(v)\!<\!1$, this implies that $e'',w\!=\!0$.
Thus, the homomorphism~\eref{wtpsicomp_e} is injective everywhere over~$U$
if the support of $\be$ is sufficiently small.
We conclude that~$\wt\psi$ is an isomorphism everywhere over~$X$.
\end{proof}

\section{NC symplectic divisors: local perspective}
\label{NCL_sec}

\noindent
Arbitrary normal crossings (or NC)~divisors are spaces that are locally SC~divisors.
This local perspective, reviewed below, 
makes it straightforward to define NC divisors and  regularizations for~them.
It is also readily applicable to local statements, 
such as Theorem~\ref{TXV_thm}\ref{logsplit_it} and Lemma~\ref{blowuplog_lmm}.\\

\noindent
For a set $S$, denote by $\cP(S)$ the collection of subsets of~$S$.
If in addition $i\!\in\!S$, let 
$$\cP_i(S)=\big\{I\!\in\!\cP(S)\!:\,i\!\in\!S\big\}.$$

\subsection{Definitions}
\label{NCLdfn_subs}

\noindent
We begin by extending the definitions of Section~\ref{SCdfn_subs} to the general NC setting.

\begin{dfn}\label{NCD_dfn}
Let $(X,\om)$ be a symplectic manifold.
A subspace  $V\!\subset\!X$ is an \sf{NC symplectic divisor in~$(X,\om)$} 
if for every $x\!\in\!X$ there exist an open neighborhood~$U$ of~$x$ in~$X$ and 
a finite transverse collection $\{V_i\}_{i\in S}$ of closed submanifolds  of~$U$ 
of codimension~2 such~that 
$$V\!\cap\!U= \bigcup_{i\in S}\!V_i$$
is an SC symplectic divisor in~$(U,\om|_U)$.
\end{dfn}

\noindent
Every NC divisor $V\!\subset\!X$ is a closed subspace; 
its \sf{singular locus} $V_{\prt}\!\subset\!V$ consists of the points $x\!\in\!V$
such that there exists a chart $(U,\{V_i\}_{i\in S})$ as in Definition~\ref{NCD_dfn} 
and  $I\!\subset\!S$ with $|I|\!=\!2$ and $x\!\in\!V_I$. 
Figure~\ref{nbhds_fig} shows an NC divisor~$V$, a chart around a singular point of~$V$,
and a chart around a smooth point of~$V$.\\

\begin{figure}
\begin{pspicture}(-3,-1)(11,1)
\psset{unit=.3cm}
\psline[linewidth=.1](13,-2)(20,-2)
\psline[linewidth=.1](15,-4)(15,3)
\pscircle*(15,-2){.3}
\pscircle*(19,-2){.3}
\psarc(20,3){5}{-90}{180}
\psline[linewidth=.1,linestyle=dashed,dash=3pt](16,-1)(14,-1)(14,-3)(16,-3)(16,-1)
\psline[linewidth=.1,linestyle=dashed,dash=3pt](20,-1)(18,-1)(18,-3)(20,-3)(20,-1)
\psline[linewidth=.1,linestyle=dashed,dash=3pt](10,2)(6,2)(6,-2)(10,-2)(10,2)
\psline[linewidth=.1](10,0)(6,0)
\psline[linewidth=.1](8,-2)(8,2)
\psline[linewidth=.1]{->}(11,0)(13,-1)
\rput(10.5,2.5){\small{U}}\rput(26,2.5){\small{V}}
\rput(8.6,1.2){\tiny{$V_1$}}
\rput(9.3,-.5){\tiny{$V_2$}}
\pscircle*(8,0){.3}
\rput(7.3,-.5){\tiny{$V_{12}$}}
\end{pspicture}
\caption{A non-SC normal crossings divisor.}
\label{nbhds_fig}
\end{figure}
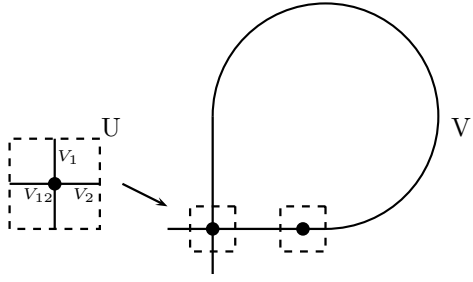

\noindent
For each chart $(U, \{V_i\}_{i\in S})$ as in Definition~\ref{NCD_dfn}
and each $x\!\in\!U$, let
$$S_x=\big\{i\!\in\!S\!:\,x\!\in\!V_i\big\}\,.$$
The cardinality $|x|\!\equiv\!|S_x|$ is independent of the choice of a chart around~$x$.
For each~$r\!\in\!\Z^{\ge0}$, let
\BE{Vrdfn_e}V^{(r)}\equiv\big\{x\!\in\!X\!:|x|\!\ge\!r\big\}.\EE
If $(U',\{V_i'\}_{i\in S'})$ is another chart for $V$ in~$X$ and $x\!\in\!U\!\cap\!U'$, 
there exist a neighborhood~$U_x$ of~$x$ in $U\!\cap\!U'$ and a bijection
\BE{NCDoverlap_e0}h_x\!:S_x\lra S_x'  \qquad\hbox{s.t.}\quad
V_i\!\cap\!U_x=V_{h_x(i)}'\!\cap\!U_x~~\forall\,i\!\in\!S_x\,.\EE
We also denote by $h_x$ the induced bijection $\cP(S_x)\!\lra\!\cP(S_x')$.
By~\eref{NCDoverlap_e0}, 
\BE{cNXVagre_e}
\cN_{V_{I-i}}V_I\big|_{V_I\cap U_x}=\cN_{V_{h_x(I)-h_x(i)}'}V_{h_x(I)}'\big|_{V_{h_x(I)}'\cap U_x}
\qquad\forall\,i\!\in\!I\!\subset\!S_x.\EE

\vspace{.15in}

\noindent
We denote by $\Symp^+(X,V)$ the space of all symplectic structures $\om$ on $X$ 
such that $V$ is an NC symplectic divisor in~$(X,\om)$.

\subsection{Regularizations}
\label{NCLreg_subs}

\noindent
Suppose $V\!\subset\!X$ is an NC divisor, 
$(U, \{V_i\}_{i\in S})$ and $(U',\{V_i'\}_{i\in S'})$ are charts for $V$ in~$(X,\om)$, and
$$\big(\cR_I\big)_{I\subset S} \equiv  
\big((\rho_{I;i},\na^{(I;i)})_{i\in I},\Psi_I\big)_{I\subset S}
\quad\hbox{and}\quad
\big(\cR_I'\big)_{I\subset S'} \equiv  
\big((\rho_{I;i}',\na'^{(I;i)})_{i\in I},\Psi_I'\big)_{I\subset S'}$$
are an $\om|_U$-regularization for $\bigcup_{i\in S} V_i $ in~$U$ 
and an $\om|_{U'}$-regularization for $\bigcup_{i\in S'}V_i'$ in~$U'$, respectively. 
We define 
$$\big(\cR_I\big)_{I\subset S}\cong_X\big(\cR_I'\big)_{I\subset S'}$$
if for every $x\!\in\!U\!\cap\!U'$ there exist $U_x$ and~$h_x$ as in~\eref{NCDoverlap_e0} 
such~that 
\BE{UxcRcond_e}\begin{split}
&\big(\rho_{I;i},\na^{(I;i)}\big)\big|_{V_I\cap U_x}
=\big(\rho_{h(I);h(i)}',\na'^{(h(I);h(i))}\big)\big|_{V_{h(I)}\cap U_x}
\quad \forall~i\!\in\!I\!\subset\!S_x  \qquad\hbox{and} \\
&\hspace{.4in}\Psi_I=\Psi_{I'}' ~~\hbox{on}~~
\Dom(\Psi_I)|_{V_I\cap U_x}\cap \Dom(\Psi_{h(I)}')|_{V_{h(I)}\cap U_x}
\quad \forall~I\!\subset\!S_x.
\end{split}\EE

\begin{dfn}\label{NCDregul_dfn}
Let $(X,\om)$ be a symplectic manifold, $V\!\subset\!X$ be an NC symplectic divisor,
and $(U_y,\{V_{y;i}\}_{i\in S_y})_{y\in\cA}$ be a collection of charts for~$V$ in~$X$
as in Definition~\ref{NCD_dfn}.
An \sf{$\om$-regularization for $V$ in~$X$} 
(with respect to the atlas $\cA$) is a~collection
$$\cR\equiv (\cR_{y;I})_{y\in\cA,I\subset S_y} \equiv  
\big((\rho_{y;I;i},\na^{(y;I;i)})_{i\in I},\Psi_{y;I}\big)_{y\in\cA,I\subset S_y}$$
such that $(\cR_{y;I})_{I\subset S_y}$ is an $\om|_{U_y}$-regularization 
for $V_y$ in~$U_y$ as in Definition~\ref{TransCollregul_dfn}\ref{sympregul_it} 
for each $y\!\in\!\cA$ 
 and 
\BE{NCDregul_e2}\big(\cR_{y;I}\big)_{I\subset S_y}\cong_X\big(\cR_{y';I}\big)_{I\subset S_{y'}}
\qquad \forall\,y,y'\!\in\!\cA\,.\EE
\end{dfn}

\vspace{.15in}

\noindent
We call an almost complex structure $J$ on $X$ compatible with a regularization~$\cR$
as in Definition~\ref{NCDregul_dfn} if $J|_{U_y}$ is compatible with the regularization
$(\cR_{y;I})_{y\in\cA,I\subset S_y}$ for $V_y$ in~$U_y$ for each $y\!\in\!\cA$ 
as defined at the end of Section~\ref{SCreg_subs}.
In particular, every open stratum $V^{(r)}\!-\!V^{(r+1)}$ is an almost complex submanifold
of~$X$ with respect to an $\cR$-compatible almost complex structure on~$X$.\\

\noindent
There are natural notions of equivalence classes of regularizations on the level of germs
and families of such equivalence classes; see \cite[Section~4.1]{SympDivConf}. 
We denote by $\Aux(X,V)$ the space of pairs $(\om,\cR)$ consisting of \hbox{$\om\!\in\!\Symp^+(X,V)$}
and the equivalence class of an $\om$-compatible regularization~$\cR$ for~$V$ in~$X$.
Let $\AK(X,V)$ be the space of triples $(\om,\cR,J)$ consisting of \hbox{$(\om,\cR)\!\in\!\Aux(X,V)$}
and an almost complex structure~$J$ on~$X$ compatible with~$\om$ and~$\cR$.
By \cite[Theorem~4.5]{SympDivConf2}, the map~\eref{Aux2Symp_e} is again a weak homotopy equivalence
in the present setting.
It remains straightforward to show that 
the map~\eref{AK2Aux_e} is also a weak homotopy equivalence in the present setting.

\subsection{Constructions}
\label{ConLNC_subs}

\noindent
In this section, we extend the constructions of Section~\ref{ConSC_subs} 
to arbitrary NC divisors in the local perspective. 
We then note that the proof of Theorem~\ref{TXV_thm}\ref{logsplit_it} in Section~\ref{logsplit_subs}
readily extends to the general NC case.\\

\noindent
Suppose $V$ is an NC symplectic divisor in a symplectic manifold~$(X,\om)$,
$(U_y,\{V_{y;i}\}_{i\in S_y})_{y\in\cA}$ is a collection of charts for~$V$ in~$X$ 
as in Definition~\ref{NCD_dfn}, and 
$\cR\!\equiv\!(\cR_y)_{y\in\cA}$ is an $\om$-regularization with respect to the atlas~$\cA$ 
in the sense of Definition~\ref{NCDregul_dfn}.
For $y\!\in\!\cA$ and \hbox{$I'\!\subset\!I\!\subset\!S_y$}, 
let $V_y\!\equiv\!V\!\cap\!U_y$, \hbox{$U_{y;I}^\circ\!\subset\!U_y$} be as in \eref{UIcirc_e},
$\th_{y';I'I}\!=\!\th_{I'I}$ be as in~\eref{thetaII_e}, and
$\vt_{y';I'I}\!=\!\vt_{I'I}$ be as in~\eref{varTII_e}.
Let
$$\pi_y\!:\cO_{\cR_y;U_y}(V_y)\lra U_y \qquad \tn{and}\qquad 
\pi_y\!: T_{\cR_y}U_y(-\log V_y)\lra U_y$$ 
be the complex line bundle \eref{COXVSC_e} and the logarithmic tangent bundle \eref{LogTXVSCasQ_e} 
determined by the $\om|_{U_y}$-regularization~$\cR_y$ for
the NC divisor~$V_y$ for~$(U_y,\om|_{U_y})$.
Let~$s_{\cR;y}$ and
$$\io_{\cR;y}\!: T_{\cR_y}U_y(-\log V_y)\lra TU_y$$
be the associated section of $\cO_{\cR_y;U_y}(V_y)$
and the vector bundle homomorphism~\eref{Ninc_e}, respectively.\\

\noindent
Suppose $y,y'\!\in\!\cA$.
For $x\!\in\!U_y\!\cap\!U_{y'}$, let $U_{yy';x}\!\equiv\!U_x\!\subset\!U_y\!\cap\!U_{y'}$ and
$$h_{y'y;x}\!\equiv\!h_x\!:S_{y;x}\!\equiv\!S_x\lra S_{y';x}\!\equiv\!S_x$$
be as in~\eref{NCDoverlap_e0}. 
By~\eref{cNXVagre_e},
\BE{NCDreg_e0}
\cN_{V_{y;S_{y;x}-i}}V_{y;S_{y;x}}\big|_{V_{y;S_{y;x}}\cap U_{yy';x}}
=\cN_{V_{y';S_{y';x}-h_{y'y;x}(i)}}V_{y';S_{y';x}}\big|_{V_{y';S_{y';x}}\cap U_{yy';x}}\\
\quad\forall~i\!\in\!S_{y;x}.\EE
We can choose $U_{yy';x}$ sufficiently small so that 
\BE{Ux-more_e}
U_{yy';x}\subset U_{y;S_{y;x}}^\circ\!\cap\! U_{y';S_{y';x}}^\circ\EE
and $U_x\!\equiv\!U_{yy';x}$ satisfies~\eref{UxcRcond_e}.
By \eref{cOlofdfnloc_e}, \eref{Ux-more_e}, and \eref{NCDreg_e0},
there are canonical identifications
\BE{compareO_e}\begin{split}
\th_{y'y;x}&:\cO_{\cR_y;S_{y;x}}(V_y)|_{U_{yy';x}}
\stackrel{=}{\lra}
\cO_{\cR_y|_{U_{yy';x}};S_{y;x}}\big(V_y\!\cap\!U_{yy';x}\big)
\stackrel{\approx}{\lra} 
\cO_{\cR_{y'};S_{y';x}}(V_{y'})|_{U_{yy';x}},\\
\vt_{y'y;x}&:T_{\cR_y}U_{S_{y;x}}^{\circ}(-\log V_y)|_{U_{yy';x}}
\stackrel{=}{\lra} T_{\cR_y|_{U_{yy';x}}}U_{yy';x}\big(\!-\log(V_y\!\cap\!U_{yy';x})\!\big)\\
&\hspace{3.2in} 
\stackrel{\approx}{\lra} T_{\cR_{y'}}U_{S_{y';x}}^{\circ}(-\log V_{y'})|_{U_{yy';x}}\,.
\end{split}\EE

\vspace{.2in}

\noindent
Suppose $x'\!\in\!U_{yy';x}$. 
By~\eref{Ux-more_e} and  the uniqueness of~$h_{y'y;x'}$,
$$S_{y;x'}\subset S_{y;x},\qquad S_{y';x'}\subset S_{y';x}, 
\qquad\hbox{and}\qquad h_{y'y;x'}=h_{y'y;x}|_{S_{y;x'}}.$$
Combining these statements with~\eref{thetaII_e}, \eref{varTII_e}, and~\eref{UxcRcond_e},
we obtain
\begin{equation*}\begin{split}
&\th_{y';S_{y';x'}S_{y';x}}\!\circ\!\th_{y'y;x}\!=\!
\th_{y'y;x'}\!\circ\!\th_{y;S_{y;x'}S_{y;x}}\!:\\
&\hspace{1in}
\cO_{\cR_y;S_{y;x}}(V_y)|_{U_{yy';x}\cap U_{yy';x'}}\lra
\cO_{\cR_{y'};S_{y';x'}}(V_{y'})|_{U_{yy';x}\cap U_{yy';x'}},\\
&\vt_{y';S_{y';x'}S_{y';x}}\!\circ\!\vt_{y'y;x}\!=\!
\vt_{y'y;x'}\!\circ\!\vt_{y;S_{y;x'}S_{y;x}}\!:\\
&\hspace{1in}
T_{\cR_y}U_{S_{y;x}}^{\circ}(-\log V_y)|_{U_{yy';x}\cap U_{yy';x'}}\lra
T_{\cR_{y'}}U_{S_{y';x'}}^{\circ}(-\log V_{y'})|_{U_{yy';x}\cap U_{yy';x'}}.
\end{split}\end{equation*}
Along with~\eref{COXVSC_e} and~\eref{LogTXVSCasQ_e}, 
this implies that the collections $\{\th_{y'y;x}\}_{x\in U_y\cap U_{y'}}$ and
$\{\vt_{y'y;x}\}_{x\in U_y\cap U_{y'}}$ determine bundle isomorphisms
\BE{compareO_e5}\begin{split}
\th_{y'y}:\cO_{\cR_y;U_y}(V_y)\big|_{U_y\cap U_{y'}}
&\lra\cO_{\cR_{y'};U_{y'}}(V_{y'})\big|_{U_y\cap U_{y'}},\\
\vt_{y'y}:T_{\cR_y}U_y(-\log V_y)\big|_{U_y\cap U_{y'}}
&\lra T_{\cR_{y'}}U_{y'}(-\log V_{y'})\big|_{U_y\cap U_{y'}},
\end{split}\EE
respectively.\\

\noindent
Suppose $y''\!\in\!\cA$ is another element and $x\!\in\!U_y\!\cap\!U_{y'}\!\cap\!U_{y''}$.
Let
$$U_{yy'y'';x}=U_{yy';x}\!\cap\!U_{yy';x}\!\cap\!U_{y'y'';x}.$$
By the uniqueness of~$h_{y''y;x}$,
$$h_{y''y;x}\!=\!h_{y''y';x}\!\circ\!h_{y'y;x}\!: S_{y;x}\lra S_{y'';x}\,.$$
This implies that the identifications~\eref{compareO_e} satisfy
\begin{equation*}\begin{split}
\th_{y''y;x}\!=\!\th_{y''y';x}\!\circ\!\th_{y'y;x}\!:
\cO_{\cR_y;S_{y;x}}(V_y)|_{U_{yy'y'';x}}&\lra
\cO_{\cR_{y''};S_{y'';x}}(V_{y''})|_{U_{yy'y'';x}},\\
\vt_{y''y;x}\!=\!\vt_{y''y';x}\!\circ\!\vt_{y'y;x}\!:
T_{\cR_y}U_{S_{y;x}}^{\circ}(-\log V_y)|_{U_{yy'y'';x}}&\lra
T_{\cR_y}U_{S_{y'';x}}^{\circ}(-\log V_{y''})|_{U_{yy'y'';x}}.
\end{split}\end{equation*}
Thus, the collections $\{\th_{y'y}\}_{y,y'\in\cA}$ 
and $\{\vt_{y'y}\}_{y,y'\in\cA}$  satisfy the cocycle condition.
The first collection thus determines a complex line bundle
\BE{COXVNC_e}\begin{split}
&\pi\!: \cO_{\cR;X}(V)\!=\!\bigg(\bigsqcup_{y\in\cA}\cO_{\cR_y;U_y}(V_y)
\!\!\bigg)\!\Big/\!\!\sim\,\,\lra X,\quad \pi|_{\cO_{\cR_y;U_y}(V_y)}=\pi_y,\\
&\hspace{.2in}
\cO_{\cR_y;U_y}(V_y)\big|_{U_y\cap U_{y'}}
\ni v\sim \th_{y'y}(v)\in \cO_{\cR_{y'};U_{y'}}(V_{y'})\big|_{U_y\cap U_{y'}} 
\quad \forall~~y,y'\!\in\!\cA.
\end{split}\EE
The second collection similarly determines a vector bundle
\BE{LogTXVNCasQ_e}\begin{split}
&\pi\!: T_\cR X(-\log V)\!=\!\bigg(
\bigsqcup_{y\in\cA}T_{\cR_y}U_y(-\log V_y)\!\!\bigg)\!\Big/\!\!\sim\,\,\lra X,\quad
\pi|_{T_{\cR_y}U_y(-\log V_y)}=\pi_y,\\
&
T_{\cR_y}U_y(-\log V_y)\big|_{U_y\cap U_{y'}}
\ni v\sim \vt_{y'y}(v)\in
T_{\cR_{y'}}U_{y'}(-\log V_{y'})\big|_{U_y\cap U_{y'}}
\quad \forall~y,y'\!\in\!\cA.
\end{split}\EE

\vspace{.15in}

\noindent
By \eref{LS_e}, \eref{ioRSC_e}, and~\eref{UxcRcond_e},
\begin{equation*}\begin{split}
&~~\th_{y'y}\!\circ\!s_{\cR;y}\!=\!s_{\cR;y'}\!:U_y\!\cap\!U_{y'}\lra 
\cO_{\cR_{y'};U_{y'}}(V_{y'})\big|_{U_y \cap U_{y'}}, \\
&\io_{\cR;y}\!=\!\io_{\cR;y'}\!\circ\!\vt_{y'y}\!:
T_{\cR_y}U_y(-\log V_y)\big|_{U_y\cap U_{y'}}\lra TX\big|_{U_y \cap U_{y'}}.
\end{split}\end{equation*}
The collections $\{s_{\cR;y}\}_{y\in\cA}$ and $\{\io_{\cR;y}\}_{y\in\cA}$ thus
determine a section~$s_{\cR}$ of the line bundle~\eref{COXVNC_e} and 
a bundle homomorphism~$\io_{\cR}$ as in~\eref{Ninc_e}.
Since the sections~$s_{\cR;y}$ and the homomorphisms~$\io_{\cR;y}$ satisfy 
the properties of~$s_{\cR}$ and~$\io_{\cR}$ stated in 
Proposition~\ref{OXV_prp}\ref{cOdfn_it} and Theorem~\ref{TXV_thm}\ref{loddfn_it}, so do 
the just constructed sections~$s_{\cR}$ and~$\io_{\cR}$.\\

\noindent
Suppose $J$ is an $\cR$-compatible almost complex structure on~$X$.
For each $y\!\in\!\cA$, 
$J_y\!\equiv\!J|_{U_y}$ is then an $\cR_y$-compatible almost complex structure on~$U_y$
and determines a complex structure on the vector bundle 
$$T_{\cR_y}U_I^{\circ}(-\log V_y)\lra U_I^{\circ}$$
for every $I\!\subset\!S_y$.
Since $J_y\!=\!J_{y'}$ on $U_{yy';x}$, \eref{UxcRcond_e} implies that 
the bundle identification~$\vt_{y'y;x}$ in~\eref{compareO_e} is $\C$-linear.
Thus, $J$ determines a complex structure~$\fI_{\cR;J}$ on the vector bundle~\eref{LogTXVNCasQ_e}
which restricts to the complex structure~$\fI_{\cR_y;J_y}$ on the vector bundle
$$T_{\cR_y}U_y(-\log V_y)\lra U_y$$
for every $y\!\in\!\cA$.
Since the bundle homomorphism~$\io_{\cR_y}$ is $\C$-linear with respect to
the complex structure~$\fI_{\cR_y;J_y}$ on its domain and the complex structure~$J_y$
on its target, we obtain Theorem~\ref{TXV_thm}\ref{logJ_it}.\\

\noindent
By the same reasoning as at the end of Section~\ref{smooth_subs},
the bundles~\eref{COXVNC_e} and~\eref{LogTXVNCasQ_e} also satisfy
the properties in Proposition~\ref{OXV_prp}\ref{cOinv_it}	and Theorem~\ref{TXV_thm}\ref{loginv_it}.\\

\noindent
Suppose $V'$ is another NC symplectic divisor in $(X,\om)$ so that 
$V\!\cup\!V'\!\subset\!X$ is also an NC symplectic divisor
and $V\!\cap\!V'$ contains no open subspace of~$V$.
An atlas for $V\!\cup\!V'$ in~$X$ is a collection of the form
$(U_y,\{V_{y;i}\}_{i\in S_y\sqcup S_y'})_{y\in\cA}$  so that 
$(U_y,\{V_{y;i}\}_{i\in S_y})_{y\in\cA}$ and 
$(U_y,\{V_{y;i}\}_{i\in S_y'})_{y\in\cA}$ are atlases for~$V$ and~$V'$,
respectively.
An $\om$-regularization $\wt\cR\!\equiv\!(\wt\cR_y)_{y\in\cA}$ for such an atlas for 
$V\!\cup\!V'$ restricts to $\om$-regularizations 
$$\cR\equiv(\cR_y)_{y\in\cA} \qquad\hbox{and}\qquad
\cR'\equiv(\cR_y')_{y\in\cA}$$
for the associated atlases for~$V$ and~$V'$ so that 
the $\om|_{U_y}$-regularization~$\wt\cR_y$ for $(V\!\cup\!V')\!\cap\!U_y$ in~$U_y$
restricts to
the $\om|_{U_y}$-regularizations~$\cR_y$ for $V\!\cap\!U_y$ in~$U_y$
and $\cR_y'$ for $V'\!\cap\!U_y$.
We denote by~$\th_{y'y}'$ and~$\vt_{y'y}'$ (resp.~$\wt\th_{y'y}$ and~$\wt\vt_{y'y}$)
the analogues of the transition maps~$\th_{y'y}$ and~$\vt_{y'y}$ in~\eref{compareO_e5}
for the regularization~$\cR'$ (resp.~$\wt\cR$).
For each $y\!\in\!\cA$, let 
\begin{gather*}
\psi_{\cR_y\cR_y'}\!:
\big(\cO_{\wt\cR_y;U_y}(V_y\!\cup\!V_y'),\fI_{\wt\cR_y}\big)\lra
\big(\cO_{\cR_y;U_y}(V_y),\fI_{\cR_y}\big)
\!\otimes\!\big(\cO_{\cR_y';U_y}(V_y'),\fI_{\cR_y'}\big),\\
\io_{\cR_y\wt\cR_y}\!:
T_{\wt\cR_y}U_y\big(\!-\log(V_y\!\cup\!V_y')\!\big)\lra T_{\cR_y}X(-\log V_y)
\end{gather*}
be the bundle isomorphism~\eref{cOsplitNC_e} and 
the bundle homomorphism~\eref{iocRcRpr_e} determined by the regularization~$\wt\cR_y$.\\

\noindent
Since the maps~\eref{cOsplitNC_e} and~\eref{iocRcRpr_e}
 are natural with respect to the restrictions to open subsets,
\begin{equation*}\begin{split}
&\big\{\th_{y'y}\!\otimes\!\th_{y'y}'\big\}\!\circ\!\psi_{\cR_y\cR_y'}
\!=\!\psi_{\cR_{y'}\cR_{y'}'}\!\circ\!\wt\th_{y'y}\!:\\
&\hspace{1in}
\cO_{\wt\cR_y;U_y}(V_y\!\cup\!V_y')\big|_{U_y\cap U_{y'}}\lra
\cO_{\cR_{y'};U_{y'}}(V_{y'})\big|_{U_y\cap U_{y'}}\!\otimes\!
\cO_{\cR_{y'}';U_{y'}}(V_{y'}')\big|_{U_y\cap U_{y'}},\\
&\vt_{y'y}\!\circ\!\io_{\cR_y\wt\cR_y}
\!=\!\io_{\cR_{y'}\wt\cR_{y'}}\!\circ\!\wt\vt_{y'y}\!:
T_{\wt\cR_y}U_y\big(\!-\log(V_y\!\cup\!V_y')\!\big)
\big|_{U_y\cap U_{y'}}\lra T_{\cR_{y'}}U_{y'}(-\log V_{y'})\big|_{U_y\cap U_{y'}}.
\end{split}\end{equation*} 
The collection $\{\psi_{\cR_y\cR_y'}\}_{y\in\cA}$ thus determines an isomorphism
$$\psi_{\cR\cR'}\!:
\big(\cO_{\wt\cR;X}(V\!\cup\!V'),\fI_{\wt\cR}\big)\lra
\big(\cO_{\cR;X}(V),\fI_{\cR}\big)\!\otimes\!\big(\cO_{\cR';X}(V'),\fI_{\cR'}\big)$$
of complex line bundles over~$X$.
In light of  Proposition~\ref{OXV_prp}\ref{cOinv_it}, 
this establishes Proposition~\ref{OXV_prp}\ref{cOsplit_it}.\\

\noindent
The collection $\{\io_{\cR_y\wt\cR_y}\}_{y\in\cA}$ similarly determines 
a homomorphism
$$\io_{\cR\wt\cR}\!:
T_{\wt\cR}X\big(\!-\log(V\!\cup\!V')\!\big)\lra T_\cR X(-\log V) $$
of vector bundles over $X$.
By the properties of the homomorphisms $\io_{\cR\wt\cR}\!\equiv\!\io_{\cR_y\wt\cR_y}$ 
stated in Lemma~\ref{ioRRpr_lmm},
this homomorphism satisfies the same properties with 
$K\!\subset\!S'$, $V_K'^{\circ}$, and \hbox{$V_K'^{\circ}\!\times\!\C^K$} 
replaced by $r\!\in\!\Z^{\ge0}$, $V'^{(r)}\!-\!V'^{(r+1)}$, and 
some rank~$r$ complex vector bundle over~$V'^{(r)}\!-\!V'^{(r+1)}$, respectively.
By the same reasoning as in the paragraph after Lemma~\ref{ioRRpr_lmm},
$\io_{\cR\wt\cR}$ is $\C$-linear with respect to the complex structures~$\fI_{\wt\cR,J}$
on its domain and~$\fI_{\cR,J}$ on its target determined by an $\wt\cR$-compatible almost 
complex structure~$J$ on~$X$.
Furthermore, the $\C$-linear homomorphism $\psi\!=\!\io_{\cR\wt\cR}$ satisfies the conditions 
of Lemma~\ref{bundleisom_lmm} if $V'$ is smooth
(in which case $V'^{(1)}\!-\!V'^{(2)}\!=\!V'$).
In light of Theorem~\ref{TXV_thm}\ref{loginv_it},
this establishes Theorem~\ref{TXV_thm}\ref{logsplit_it}.

\section{Almost complex and symplectic blowups}
\label{blowup_sec}

\noindent
Let $X$ be a smooth manifold and $V\!\subset\!X$ be an \sf{NC smooth divisor}, 
i.e.~a subspace admitting a collection $(U_y,\{V_{y;i}\}_{i\in S_y})_{y\in\cA}$ of charts
as in Definition~\ref{NCD_dfn} with each \hbox{$V_{y;i}\!\subset\!V\!\cap\!\!U_y$}
being a smooth submanifold of~$U_y$ of real codimension~2.
We fix such a collection.
Let $r\!\in\!\Z^+$ be such that \hbox{$V^{(r+1)}\!=\!\eset$}.
We also fix a \sf{Hermitian regularization} 
\BE{blowupcRdfn_e}\cR\equiv (\cR_y)_{y\in\cA} \equiv
\big(\!(\rho_{y;I;i},\na^{(y;I;i)})_{i\in I},\Psi_{y;I}\big)_{y\in\cA,I\subset S_y}\EE
for \sf{$V$ in~$X$}, i.e.~a tuple satisfying the conditions of 
Definitions~\ref{TransCollregul_dfn}\ref{sympregul_it} and~\ref{NCDregul_dfn} 
that do not involve the symplectic form~$\om$.\\

\noindent
By shrinking the open sets $U_y$, we may assume that $|S_y|\!\le\!r$ for all~$y\!\in\!\cA$, 
\hbox{$\Im(\Psi_{y;S_y})\!=\!U_y$} if $y\!\in\!\cA$ with $|S_y|\!=\!r$, 
$$\Im(\Psi_{y;S_y})\!\cap\!\Im(\Psi_{y';S_{y'}})\subset
\Psi_{y;S_y}\big(\Dom(\Psi_{y;S_y})|_{V_{y;S_y}\cap V_{y';S_{y'}}}\big)
\!\cup\!\Psi_{y';S_{y'}}\big(\Dom(\Psi_{y';S_{y'}})|_{V_{y;S_y}\cap V_{y';S_{y'}}}\big),$$
and there exists an open neighborhood~$U_r'$ of $V^{(r)}\!\subset\!X$ so that 
$U_r'\!\cap\!U_{y'}\!=\!\eset$ for all $y'\!\in\!\cA$ with $|S_{y'}|\!<\!r$.
Let 
$$\cA_r=\big\{y\!\in\!\cA\!:|S_y|\!=\!r\big\}, \qquad
\wt\cA=(\cA\!-\!\cA_r)\!\sqcup\!
\big\{\!(y,i)\!:y\!\in\!\cA_r,~i\!\in\!S_y\big\}.$$

\subsection{Smooth complex blowup}
\label{SmothCblowup_subs}

\noindent 
Since $V^{(r+1)}\!=\!\eset$, $V^{(r)}\!\subset\!X$ is a smooth submanifold.
Let 
\BE{pirdfn_e}\pi_r\!:\cN_XV^{(r)}\lra V^{(r)}\EE 
be its normal bundle.
By~\eref{NCDregul_e2}, the Hermitian metrics~$\rho_{y;S_y;i}$ and 
the connections~$\na^{(y;S_y;i)}$ in the complex line bundles~$\cN_{V_{y;S_y-i}}V_{y;S_y}$
with $y\!\in\!\cA_r$ and $i\!\in\!S_y$ determine
a complex structure, a Hermitian metric~$\rho_r$, and 
a compatible connection~$\na^{(r)}$ on~$\cN_XV^{(r)}$.
Furthermore, the~map
\BE{Psirdfn_e}\Psi_r\!:\cN_r'\!\equiv\!
\bigcup_{y\in\cA_r}\!\!\!\Dom(\Psi_{y;S_y})\lra X, \qquad
\Psi_r(v)=\Psi_{y;S_y}(v)~~\forall\,v\!\in\!\Dom(\Psi_{y;S_y}),\,y\!\in\!\cA_r,\EE
is a well-defined regularization for $V^{(r)}$ in~$X$ in the sense of
Definition~\ref{smreg_dfn}.
We note that the complex vector bundle~$\cN_XV^{(r)}$ does not necessarily split 
as a sum of complex line bundles.\\

\noindent
We denote by $\bE\!\equiv\!\P(\cN_XV^{(r)})$ the complex projectivization of~$\cN_XV^{(r)}$ 
and~by 
\BE{gadfn_e}
\wt\pi_0\!:\ga\equiv\big\{\!(\ell,v)\!\in\!\bE\!\times\!\cN_XV^{(r)}\!:v\!\in\!\ell\big\}\lra\bE\EE
the complex tautological line bundle.
Let
\BE{pibEdfn_e}\pi\!:\bE\lra V^{(r)}, \qquad 
\pi(\ell)=\pi_r(v)~~\hbox{if}~(\ell,v)\!\in\!\ga,\EE
be the bundle projection.
The connection~$\na^{(r)}$ and the Hermitian metric~$\rho_r$ on~$\cN_XV^{(r)}$ determine 
a splitting
\BE{TcNXVsplit_e}\begin{split}
T(\cN_XV^{(r)})\big|_{\cN_XV^{(r)}-V^{(r)}}
\approx\pi_r^*TV^{(r)}\big|_{\cN_XV^{(r)}-V^{(r)}}\!\oplus\!
\big\{\!(v,w)\!\in\!\pi_r^*\cN_XV^{(r)}\big|_{\cN_XV^{(r)}-V^{(r)}}\!:w\!\in\!\C v\big\}&\\
\!\oplus\!
\big\{\!(v,w)\!\in\!\pi_r^*\cN_XV^{(r)}\big|_{\cN_XV^{(r)}-V^{(r)}}\!:
\rho_r(v,w)\!=\!0\big\}&
\end{split}\EE
of the vector bundle $T(\cN_XV^{(r)})|_{\cN_XV^{(r)}-V^{(r)}}$
so that the middle summand above is identified with 
the tangent bundle to the orbits of the $\C^*$-action on~$\cN_XV^{(r)}$
and the last summand is its complement in the vertical tangent bundle of~$\pi_r$
restricted to~$\cN_XV^{(r)}\!-\!V^{(r)}$.
By \cite[Lemma~1.1]{anal}, the above splitting is $\C^*$-invariant.
It thus induces a splitting
\BE{TbEsplit_e}T\bE\approx\pi^*TV^{(r)}\!\oplus\!\big(\!\ker\nd\pi\big)\EE
so that the last summand above corresponds to the vertical tangent subbundle of~$\pi$
and a complex structure on the last summand.\\

\noindent
Since $\ga\!\subset\!\pi^*\cN_XV^{(r)}$, the Hermitian metric~$\rho_r$ and 
the compatible connection~$\na^{(r)}$ on~$\cN_XV^{(r)}$ determine 
a Hermitian structure~$(\wt\rho_0,\wt\na^{(0)})$ on the complex line bundle~$\ga$ and
a splitting
\BE{Tgasplit_e}T\ga\approx \wt\pi_0^*T\bE\!\oplus\!\wt\pi_0^*\ga\EE
so that the last summand above corresponds to the vertical tangent bundle of~$\wt\pi_0$.
The composition of this splitting with the splitting~\eref{TbEsplit_e} restricts to
the splitting~\eref{TcNXVsplit_e} under the identification
\BE{TcNident_e}\big(\ga\!-\!\bE,\pi\!\circ\!\wt\pi_0|_{\ga-\bE}\big)
=\big(\cN_XV^{(r)}\!-\!V^{(r)},\pi_r|_{\cN_XV^{(r)}-V^{(r)}}\big),
\qquad (\ell,v)\lra v.\EE

\vspace{.15in}

\noindent
We define the smooth complex blowup $\pi\!:\wt{X}\!\lra\!X$ 
of~$X$ along~$V^{(r)}$ with respect to~$\Psi_r$ by
\begin{gather*}
\wt\cN_0'=\big\{\!(\ell,v)\!\in\!\ga\!:v\!\in\!\cN_r'\big\}, ~~
\wt{X}\equiv\big(\!(X\!-\!V^{(r)})\!\sqcup\!\wt\cN_0'\big)\!\big/\!\!\sim,
~~ \wt\cN_0'\!-\!\bE\ni(\ell,v)\sim\Psi_r(v)\in X\!-\!V^{(r)},\\
\pi([\wt{x}])=\begin{cases}\wt{x},&\hbox{if}~\wt{x}\!\in\!X\!-\!V^{(r)};\\
\Psi_r(v),&\hbox{if}~\wt{x}\!\equiv\!(\ell,v)\!\in\!\wt\cN_0'.
\end{cases}
\end{gather*}
The exceptional divisor $\bE$ is a codimension~2 submanifold of~$\wt{X}$
with a smooth regularization 
$$\wt\Psi_0\!:\wt\cN_0'\lra \wt{X}, \qquad \wt\Psi_0(v)=[v].$$

\subsection{Almost complex blowup}
\label{AlCBl_sub}

\noindent
Suppose now that $J$ is an almost complex structure on~$X$,
$V\!\subset\!X$ is an NC almost complex divisor, 
each \hbox{$V_{y;i}\!\subset\!V\!\cap\!\!U_y$} is 
an almost complex submanifold of $(U_y,J|_{U_y})$ of real codimension~2, 
and the almost complex structure $J$ on $X$ is $\cR$-compatible in 
the sense defined at the end of Section~\ref{NCLreg_subs}. 
The smooth submanifold $V^{(r)}\!\subset\!X$ is then almost complex.
The induced complex structure on its normal bundle agrees with the fiberwise complex structure
determined by the complex line bundles~$\cN_{V_{y;S_y-i}}V_{y;S_y}$
with $y\!\in\!\cA_r$ and $i\!\in\!S_y$.\\

\noindent
Along with the fiberwise complex structure and the connection~$\na^{(r)}$ on~$\cN_XV^{(r)}$, 
$J|_{V^{(r)}}$ determines a complex structure~$J_{\cR;r}$ on the total space 
of~$\cN_XV^{(r)}$ such~that
\BE{PsirJ} J\!\circ\!\nd\Psi_r=\nd\Psi_r\!\circ\!J_{\cR;r}\big|_{\cN_r'}\,.\EE
Along with the splitting~\eref{TbEsplit_e},
$J|_{V^{(r)}}$ determines an almost complex structure $\wt{J}_{\bE}$ on~$\bE$ so that 
the bundle projection~\eref{pibEdfn_e} is $(J,\wt{J}_{\bE})$-holomorphic.
Along with the splitting~\eref{Tgasplit_e}, $\wt{J}_{\bE}$ in turn determines
an almost complex structure~$\wt{J}_{\cR;0}$ on the total space of~$\ga$. 
By the sentence containing~\eref{TcNident_e},
the restrictions of the almost complex structures~$\wt{J}_{\cR;0}$ to $\ga\!-\!\bE$
and~$J_{\cR;r}$ to~$\cN_XV^{(r)}\!-\!V^{(r)}$ agree under the identification~\eref{TcNident_e}.
Furthermore, the projection
$$\pi_2\!:\ga\lra \cN_XV^{(r)}$$
to the second component in~\eref{gadfn_e} is $(J_{\cR;r},\wt{J}_{\cR;0})$-holomorphic.\\

\noindent
We define an almost complex structure~$\wt{J}$ on the blowup~$\wt{X}$ of~$X$ constructed
in Section~\ref{SmothCblowup_subs}~by
$$\wt{J}_{[\wt{x}]}=\begin{cases}J_{\wt{x}},&\hbox{if}~\wt{x}\!\in\!X\!-\!V^{(r)};\\
\wt{J}_{\cR;0}|_{\wt{x}},&\hbox{if}~\wt{x}\!\in\!\wt\cN_0'.
\end{cases}$$
By the conclusion of the previous paragraph and~\eref{PsirJ}, 
the definitions of $\wt{J}$ agree on $\wt\cN_0'\!-\!\bE$.
The exceptional divisor $\bE$ is an almost complex submanifold of $(\wt{X},\wt{J})$.
The almost complex structure~$\wt{J}$ 
 is compatible with the Hermitian regularization 
$(\wt\rho_0,\wt\na^{(0)},\wt\Psi_0)$ for~$\bE$ in~$\wt{X}$.\\

\noindent
Let $\ov{V}\!\subset\!\wt{X}$ be the proper transform of~$V$, 
i.e.~the closure of~$V\!-\!V^{(r)}$, and
$$\wt{V}=\bE\!\cup\!\ov{V}.$$
We show below that~$\wt{V}$ 
is an NC almost complex divisor in $(\wt{X},\wt{J})$ with a collection
$(\wt{U}_y,\{\wt{V}_{y;i}\}_{i\in\wt{S}_y})_{y\in\wt\cA}$  of charts and a regularization
\BE{cRblowup_e}\wt\cR\equiv (\wt\cR_{y;I})_{y\in\wt\cA,I\subset\wt{S}_y} \equiv  
\big(\!(\wt\rho_{y;I;i},\wt\na^{(y;I;i)})_{i\in I},\wt\Psi_{y;I}
\big)_{y\in\wt\cA,I\subset\wt{S}_y}\EE
obtained from the atlas~$(U_y,\{V_{y;i}\}_{i\in S_y})_{y\in\cA}$ and the regularization~$\cR$
for~$V$ in~$X$.
If $y\!\in\!\cA\!-\!\cA_r$, then
$$U_y\subset X\!-\!V^{(r)}=\wt{X}\!-\!\bE.$$
In this case, we simply take 
$$\wt{S}_y=S_y, \quad 
\big(\wt{U}_y,\{\wt{V}_{y;i}\}_{i\in\wt{S}_y}\big)=(U_y,\{V_{y;i}\}_{i\in S_y}\big), \quad
\big(\wt\cR_{y;I}\big)_{I\subset\wt{S}_y}=\big(\cR_{y;I}\big)_{I\subset S_y}.$$

\vspace{.15in}

\noindent
Suppose $y\!\in\!\cA_r$.
Let $\wt{U}_{y}\!\equiv\!\pi^{-1}(U_y)$ be the blowup of~$U_y$ along $V^{(r)}\!\cap\!U_y$.
Since \hbox{$\Im(\Psi_{y;S_y})\!=\!U_y$}, we can identify~$U_y$ with $\Dom(\Psi_{y;S_y})$ via~$\Psi_r$ 
and~$\wt{U}_y$ with $\pi^{-1}(U_y)\!\subset\!\wt\cN_0'$ via~$\wt\Psi_0$.
For  each $i\!\in\!S_y$, let $\ov{V}_{y;i}\!\subset\!\wt{U}_y$ be the proper transform of~$V_{y;i}$
and define
$$\wt{S}_{(y,i)}=\{0\}\!\sqcup\!\big(S_y\!-\!\{i\}\!\big), ~~
\cN_{y;i}=\bigoplus_{j\in S_y-i}\!\!\!\cN_{V_{y;S_y-i}}V_{y;S_y}, ~~
\wt{U}_{(y,i)}=\wt{U}_y\!-\!\ga|_{\P\cN_{y;i}}, ~~
\wt{V}_{(y,i);0}=\bE\!-\!\P\cN_{y;i}.$$
For $j\!\in\!S_y\!-\!i$, let $\wt{V}_{(y,i);j}=\ov{V}_{y;j}\!\cap\!\wt{U}_{(y,i)}$.
We note~that 
\BE{BlDecomp_e}V_{y;i}=\cN_{y;i}\!\cap\!\Dom(\Psi_{y;S_y}), \qquad 
\ov{V}_{y;i}=\ga\big|_{\P\cN_{y;i}}\!\cap\!\wt{U}_y,
\quad\hbox{and}\quad 
\ov{V}\!\cap\!\wt{U}_{(y,i)}=\bigcup_{j\in S_y-i}\!\!\!\!\!\wt{V}_{(y,i);j}\EE
under the above identifications.
Since $\{\wt{V}_{(y,i);j}\}_{j\in\wt{S}_{(y,i)}}$ is a transverse collection of
codimension~2 almost complex submanifolds of~$\wt{U}_{(y,i)}$, 
$\wt{V}\!\cap\!\wt{U}_{(y,i)}$ is an SC almost complex divisor in~$\wt{U}_{(y,i)}$.
Thus, $\wt{V}$ is an NC almost complex divisor in~$(\wt{X},\wt{J})$ and
$(\wt{U}_y,\{\wt{V}_{y;i}\}_{i\in\wt{S}_y})_{y\in\wt\cA}$  is an atlas of local charts
for~$\wt{V}$.

\subsection{Regularizations}
\label{CblowupReg_subs}

\noindent
For  $y\!\in\!\cA_r$ and \hbox{$I'\!\subset\!I\!\subset\!S_y$}, let
$$\cN_{y;I;I'}'\subset \cN_{y;I;I'}\subset 
\cN_{U_y}V_{y;I}\!=\!\bigoplus_{i\in I}\cN_{V_{y;I-i}}V_{y;I}$$
be the analogues of the subspaces $\cN_{I;I'}'\!\subset\!\cN_{I;I'}\!\subset\!\cN_XV_I$
for the regularization $(\Psi_{y;I})_{I\subset S_y}$ 
of $\{V_{y;i}\}_{i\in S_y}$ in~$U_y$ as defined in Section~\ref{SCreg_subs}.
Suppose $i\!\in\!S_y$ and $I\!\subset\!\wt{S}_{(y,i)}$.
If $0\!\in\!I$ (and so $\wt{V}_{(y,i);I}\!\subset\!\bE$), then 
\begin{equation*}\begin{split}
\cN_{\wt{U}_{(y,i)}}\!\wt{V}_{(y,i);I}=&\ga\big|_{\wt{V}_{(y,i);I}}\!\oplus\!
\big(\ga^*\big|_{\wt{V}_{(y,i);I}}\!\otimes_{\C}\!
\pi^*\cN_{U_y}V_{y;I-0}\big|_{\wt{V}_{(y,i);I}}\big)\\
\subset& 
\big\{\!\pi|_{\wt{V}_{(y,i);I}}\!\big\}^{\!*}\cN_{y;S_y;I-0}
\!\oplus\!
\big(\ga^*\big|_{\wt{V}_{(y,i);I}}\!\otimes_{\C}\!
\big\{\!\pi|_{\wt{V}_{(y,i);I}}\!\big\}^{\!*}\cN_{y;S_y;S_y-I}\big)\,.
\end{split}\end{equation*}
In this case, we define
\begin{gather*}
\wt\Psi_{(y,i);I}\!:\big\{\!
\big(\C v',v,u\!\otimes\!w\big)\!\in\!\cN_{\wt{U}_{(y,i)}}\!\wt{V}_{(y,i);I}\!:
v\!+\!u(v)w\!\in\!\cN_{y;S_y}'\!\big\}\lra \wt{U}_{(y,i)}\subset 
\big\{\!\pi|_{\bE}\!\big\}^{\!*}\cN_{U_y}V_{y;S_y}, \\
\wt\Psi_{(y,i);I}\big(\C v',v,u\!\otimes\!w\big)
=\big(\C(v'\!+\!u(v')w),v\!+\!u(v)w\big).
\end{gather*}

\vspace{.15in}

\noindent
Suppose $0\!\not\in\!I$.
Thus,
$\wt{V}_{(y,i);0\sqcup I}\!\equiv\!\wt{V}_{(y,i);I}\!\cap\!\bE$ is a smooth submanifold 
of~$\wt{V}_{(y,i);I}$ with normal bundle~$\ga|_{\wt{V}_{(y,i);0\sqcup I}}$. 
The smooth regularization
$$\wt\Psi_{(y,i);0\sqcup I;I}\!:
\wt\cN_{(y,i);0\sqcup I;I}'\!\equiv\!
\big\{\!(\ell,v)\!\in\!\ga\!:v\!\in\!\cN_{y;S_y;I}'\big\}\lra\wt{V}_{(y,i);I},
\quad \wt\Psi_{(y,i);0\sqcup I;I}(\ell,v)=(\ell,v),$$
of~$\wt{V}_{(y,i);0\sqcup I}$ in~$\wt{V}_{(y,i);I}$ is surjective. 
Let
$$\cO_{\wt\cR;(y,i);I}(\bE)\equiv 
\{\wt\Psi_{(y,i);0\sqcup I;I}^{~-1}\}^{\!*}
\big\{\!\wt\pi_0|_{\wt\cN_{(y,i);0\sqcup I;I}'}\big\}^{\!*}\ga\lra\wt{V}_{(y,i);I}$$
be the analogue of the complex line bundle~\eref{OXVSmooth_e} 
determined by~$\wt\Psi_{(y,i);0\sqcup I;I}$ and the fiberwise complex structure of~$\ga$.
In this case,
$$\cN_{\wt{U}_{(y,i)}}\!\wt{V}_{(y,i);I}=
\cO_{\wt\cR;(y,i);I}(\bE)^*\!\otimes_{\C}\!
\pi^*\cN_{U_y}V_{y;I}\big|_{\wt{V}_{(y,i);I}}
=\cO_{\wt\cR;(y,i);I}(\bE)^*\!\otimes_{\C}\!
\pi^*\cN_{y;S_y;S_y-I}\big|_{\wt{V}_{(y,i);I}}.$$
We define
\begin{gather*}
\wt\Psi_{(y,i);I}\!:
\big\{\!\big(x,(\C v',v),u\big)\!\otimes\!w\!\in\!\cN_{\wt{U}_{(y,i)}}\!\wt{V}_{(y,i);I}\!:
v\!+\!u(v)w\!\in\!\cN_{y;S_y}'\!\big\}\lra \wt{U}_{(y,i)}, \\
\wt\Psi_{(y,i);I}\big(\!(x,(\C v',v),u)\!\otimes\!w\big)
=\big(\C(v'\!+\!u(v')w),v\!+\!u(v)w\big).
\end{gather*}
Since  $(\Psi_{y;I})_{I\subset S_y}$ is a regularization
for $\{V_{y;i}\}_{i\in S_y}$ in~$U_y$, 
$(\wt\Psi_{(y,i);I})_{I\subset\wt{S}_{(y,i)}}$ is a regularization
for $\{\wt{V}_{(y,i);j}\}_{j\in\wt{S}_{(y,i)}}$ in~$\wt{U}_{(y,i)}$.
Since the collection $(\Psi_{y;I})_{y\in\cA,I\subset S_y}$ satisfies the last condition
in~\eref{UxcRcond_e}, 
so does the collection $(\wt\Psi_{y;I})_{y\in\wt\cA,I\subset\wt{S}_y}$.\\

\noindent
Let $y\!\in\!\cA_r$ as before.
The Hermitian structures $(\wt\rho_0,\wt\na^{(0)})$ on~$\ga$ and 
$(\rho_{y;I;j},\na^{(y;I;j)})$ on~$\cN_{V_y;I-j}V_{y;I}$ with $j\!\in\!I\!\subset\!S_y$
determine Hermitian structures $(\wh\rho_{(y,i);I;j},\wh\na^{((y,i);I;j)})$ 
on the complex line bundles
$$\cN_{\wt{V}_{(y,i);I-j}}\wt{V}_{(y,i);I}=
\begin{cases}\ga|_{\wt{V}_{(y,i);I}},&\hbox{if}~j\!=\!0\!\in\!I\!\subset\!\wt{S}_{(y,i)};\\
\ga^*\big|_{\wt{V}_{(y,i);I}}\!\otimes_{\C}\!
\pi^*\cN_{V_{y;S_y-j}}V_{y;S_y}\big|_{\wt{V}_{(y,i);I}},&
\hbox{if}~0,j\!\in\!I\!\subset\!\wt{S}_{(y,i)},\,j\!\neq\!0;\\
\cO_{\wt\cR;(y,i);I}(\bE)^*\!\otimes_{\C}\!
\pi^*\cN_{V_y;I-j}V_{y;I}\big|_{\wt{V}_{(y,i);I}},
&\hbox{if}~j\!\in\!I\!\subset\!\wt{S}_{(y,i)},\,0\!\not\in\!I.
\end{cases}$$
The almost complex structure $\wt{J}|_{\wt{V}_{(y,i);I}}$ and 
the connection~$\wh\na^{((y,i);I;j)}$
determine an almost complex structure~$\wt{J}_{\cR_y;I}$
on the total space of the normal bundle $\cN_{\wt{U}_{(y,i)}}\wt{V}_{(y,i);I}$
of~$\wt{V}_{(y,i);I}$ in~$\wt{U}_{(y,i)}$.
If \hbox{$0\!\not\in\!I\!\subset\!\wt{S}_{(y,i)}$}, Corollary~\ref{RadVecFld_crl} implies~that 
the isomorphism
\BE{cORisom_e}
\cN_{\wt{U}_{(y,i)}}\wt{V}_{(y,i);I}\big|_{\wt{V}_{(y,i);I}-\wt{V}_{(y,i);0\sqcup I}}
\lra \cN_{U_y}V_{y;I}\big|_{V_{y;I}-V_{y;S_y}},\quad
u\!\otimes\!(v,w)\lra u(v)w,\EE
intertwines~$\wt{J}_{\cR_y;I}$ with the almost complex structure~$J_{\cR_y;I}$
determined by the almost complex structure~$J|_{V_{y;I}}$ and 
the connections~$\na^{(y;I;j)}$ with $j\!\in\!I$.
Since the regularization~$\Psi_{y;I}$ intertwines~$J_{\cR_y;I}$ and~$J$,
it follows that the regularization~$\wt\Psi_{(y,i);I}$ intertwines~$\wt{J}_{\cR_y;I}$ and~$\wt{J}$
if $0\!\not\in\!I\!\subset\!\wt{S}_{(y,i)}$.
The same is the case for $I\!=\!\{0\}$ by the definition of~$\wt{J}|_{\wt\cN_0'}$.
Since the differentials $\fD\wt\Psi_{(y,i);I;I-0}$ with 
\hbox{$0\!\in\!I\!\subset\!\wt{S}_{(y,i)}$} are product Hermitian isomorphisms,
it follows that they intertwine the almost complex structures~$\wt{J}_{\cR_y;I}$ 
and~$\wt{J}_{\cR_y;I-0}$.
Thus, the regularization~$\wt\Psi_{(y,i);I}$ intertwines~$\wt{J}_{\cR_y;I}$ and~$\wt{J}$
for all $I\!\subset\!\wt{S}_{(y,i)}$, while $\fD\wt\Psi_{(y,i);I;I'}$ 
intertwines~$\wt{J}_{\cR_y;I}$ and~$\wt{J}_{\cR_y;I'}$ for all 
$I'\!\subset\!I\!\subset\!\wt{S}_{(y,i)}$.\\

\noindent
Below we define smooth functions $f_{(y,i);I}\!:\wt{V}_{(y,i);I}\!\lra\!\R^+$ 
so that the metrics 
\BE{wtrhodfn_e} \wt\rho_{(y,i);I;j}\equiv 
\begin{cases}f_{(y,i);I}\wh\rho_{(y,i);I},&\hbox{if}~j\!\in\!I,~j\!\neq\!0;\\
(1/f_{(y,i);I})\wh\rho_{(y,i);I},&\hbox{if}~j\!=\!0\!\in\!I;
\end{cases}\EE
on the complex line bundles $\cN_{\wt{V}_{(y,i);I-j}}\wt{V}_{(y,i);I}$ 
are preserved by the isomorphisms $\fD\wt\Psi_{(y,i);I;I'}$ with 
\hbox{$j\!\in\!I'\!\subset\!I\!\subset\!\wt{S}_{(y,i)}$}, after shrinking the domain 
of~$\wt\Psi_{(y,i);I}$.
The connection
\BE{wtnadfm_e}\wt\na^{((y,i);I;j)}\equiv 
\begin{cases}
\wh\na^{((y,i);I;j)}\!+\!f_{(y,i);I}^{-1}\partial_{\wt{J}}f_{(y,i);I},
&\hbox{if}~j\!\in\!I,~j\!\neq\!0;\\
\wh\na^{((y,i);I;j)}\!-\!f_{(y,i);I}^{-1}\partial_{\wt{J}}f_{(y,i);I},
&\hbox{if}~j\!=\!0\!\in\!I;
\end{cases}\EE
on $\cN_{\wt{V}_{(y,i);I-j}}\wt{V}_{(y,i);I}$ is compatible with~$\wt\rho_{(y,i);I;j}$.
Along with $\wt{J}|_{\wt{V}_{(y,i);I}}$, it determines 
the same almost complex structure~$\wt{J}_{\cR_y;I}$ on~$\cN_{\wt{V}_{(y,i);I-j}}\wt{V}_{(y,i);I}$
as~$\wh\na^{((y,i);I;j)}$
(because the two connections differ by a $(1,0)$-form).
Since the isomorphisms $\fD\wt\Psi_{(y,i);I;I'}$ preserve the metrics~$\wt\rho_{(y,i);I;j}$
and the almost complex structures~$\wt{J}_{\cR_y;I}$, it follows that they preserve
the connections~$\wt\na^{((y,i);I;j)}$ as~well.\\

\noindent
We choose the functions~$f_{(y,i);I}$ so~that
\begin{gather}\label{fcond_e1a}
f_{(y,i);I}(x)=f_{(y,i');I}(x)
\quad\forall~x\!\in\!\wt{V}_{(y,i);I}\!\cap\!\wt{V}_{(y,i');I},\,
I\!\subset\!\wt{S}_{(y,i)},\wt{S}_{(y,i')},\,i,i'\!\in\!S_y,\,y\!\in\!\cA_r,\\
\label{fcond_e1b}
\begin{split}
&f_{(y,i);I}(x)=f_{(y',h_{y'y;x}(i));h_{y'y;x}(I)}(x)
~~\forall~x\!\in\!\wt{V}_{(y,i);I}\!\cap\!\wt{V}_{(y',h_{y'y;x}(i));h_{y'y;x}(I)},
\,i\!\in\!S_y,\,y,y'\!\in\!\cA_r,\,
\end{split}\end{gather}
with $h_{y'y;x}$ as above~\eref{NCDreg_e0}.
For $I\!\not\ni\!0$, we choose~$f_{(y,i);I}$ so that the isomorphism~\eref{cORisom_e}
identifies the restriction of~$\wt\rho_{(y,i);I;j}$ to 
$\cN_{\wt{V}_{(y,i);I-j}}\wt{V}_{(y,i);I}|_{\wt{V}_{(y,i);I}-U_r'}$
with the restriction of the metric~$\rho_{y;I;j}$ to~$\cN_{V_{y;I-j}}V_{y;I}|_{V_{y;I}-U_r'}$.
Along with the assumption on~$U_r'$,
this implies that the resulting collection~$\wt\cR$ in~\eref{cRblowup_e}
satisfies the first condition in~\eref{UxcRcond_e}.

\begin{lmm}\label{bumpfun_lmm1}
Let $r\!\in\!\Z^+$. 
There exist a smooth function $h\!:\P^{r-1}\!\lra\!\R^+$ and $\de\!\in\!\R^+$
such~that
\begin{enumerate}[label=(\arabic*),leftmargin=*] 

\item $h$ is invariant under the permutations of the homogeneous coordinates on~$\P^{r-1}$
and under the standard $(S^1)^r$-action on~$\P^{r-1}$;

\item  for all $s\!\in\![r]$, $[Z_1,\ldots,Z_s]\!\in\!\P^{s-1}$,
and $[Z_1,\ldots,Z_r]\!\in\!\P^{r-1}$
with $\sum\limits_{i=s+1}^r\!\!\!\!|Z_i|^2\!\le\!\de\sum\limits_{i=1}^s\!|Z_i|^2$,
\BE{bumpfun_e0}\frac{h(Z_1,\ldots,Z_r)}{h(Z_1,\ldots,Z_s,0,\ldots,0)}=
\frac{|Z_1|^2\!+\!\ldots\!+\!|Z_r|^2}{|Z_1|^2\!+\!\ldots\!+\!|Z_s|^2}.\EE

\end{enumerate}
\end{lmm}

\vspace{.2in}

\noindent
This lemma is established in Section~\ref{BlowPf_subs}.
Let $h\!:\P^{r-1}\!\lra\!\R^+$ and $\de\!\in\!\R^+$ be 
as in Lemma~\ref{bumpfun_lmm1}.
Define
$$W_{y;I}=\Big\{\!\big[(v_i)_{i\in S_y}\big]\!\in\!\P\cN_XV^{(r)}|_{V_{y;S_y}}\!:
\sum_{i\in S_y-I}\!\!\!\rho_{y;S_y;i}(v_i)\!<\!\de\sum_{i\in I}\rho_{y;S_y;i}(v_i)\Big\}
~~\forall~I\!\subset\!S_y,\,y\!\in\!\cA_r.$$
We can identify each fiber of~\eref{pirdfn_e}  with $\C^r$ 
respecting the splittings and the metrics.
This induces an identification of each fiber of $\pi\!:\bE\!\lra\!V^{(r)}$ 
with~$\P^{r-1}$.
By the invariance properties of~$h$, the composition of this identification with~$h$
is independent of the choice of the former.
Thus, we obtain a smooth function \hbox{$h_{\bE}\!:\bE\!\lra\!\R^+$} so~that
\BE{hbEprp_e}
\frac{h_{\bE}\big([(v_i)_{i\in S_y}]\big)}{h_{\bE}\big([(v_i)_{i\in I}]\big)}
=\frac{\rho_r\big([(v_i)_{i\in S_y}]\big)}{\rho_r\big([(v_i)_{i\in I}]\big)}
\quad\forall~[(v_i)_{i\in S_y}]\!\in\!W_{y;I},\,I\!\subset\!S_y,\,y\!\in\!\cA_r.\EE

\vspace{.15in}

\noindent
For $y\!\in\!\cA_r$, $i\!\in\!S_y$, and $I\!\subset\!\wt{S}_{(y,i)}$ so that $0\!\in\!I$,
let $U_{y;I}\!\subset\!\wt{U}_y$ be an open neighborhood of~$\wt{V}_{y;I-0}$
so that $U_{y;I}\!\cap\!\bE\!=\!W_{y;I-0}$.
We define
\BE{rhoscaldfn1_e}f_{(y,i);I}=h_{\bE}|_{\wt{V}_{(y,i);I}}.\EE
By~\eref{hbEprp_e},
the restrictions of the differentials $\fD\wt\Psi_{(y,i);I;I'}$ with 
$0\!\in\!I'\!\subset\!I\!\subset\!\wt{S}_{(y,i)}$ to $\wt\Psi_{(y,i);I}^{-1}(U_{y;I})$
preserve the metrics~\eref{wtrhodfn_e}.
By~\eref{rhoscaldfn1_e}, the functions~$f_{(y,i);I}$ 
satisfy~\eref{fcond_e1a} and~\eref{fcond_e1b} whenever \hbox{$0\!\in\!I$}.\\

\noindent
Let $\be\!:\R\!\lra\!\R^{\ge0}$ and $\ve\!:V^{(r)}\!\lra\!\R^+$ be smooth functions so~that
\begin{gather}\label{becond_e}
\be(t)=\begin{cases}1,&\hbox{if}~t\!\le\!1;\\
0,&\hbox{if}~t\!\ge\!2;\end{cases} \qquad\hbox{and}\\
\notag
\bigcup_{y\in\cA_r}\!\!\!
\big\{\!(v_i)_{i\in S_y}\!\in\!\cN_r'|_{V_{y;S_y}}\!:
\rho_{y;S_y;i}(v)\!<\!2\ve\big(\pi_r(\!(v_i)_{i\in S_y})\!\big)~\forall\,i\!\in\!S_y\big\}
\subset U_r'\subset U_r\!=\!\cN_r'.
\end{gather}
Define
\begin{gather*}
U_r''=\bigcup_{y\in\cA_r}\!\!\!
\big\{\!(v_i)_{i\in S_y}\!\in\!\cN_r'|_{V_{y;S_y}}\!:
\rho_{y;S_y;i}(v_i)\!<\!\ve\big(\pi_r(\!(v_i)_{i\in S_y})\!\big)~\forall\,i\!\in\!S_y\big\}
\subset U_r'\subset\cN_r',\\
U_{y;I}=\big\{\!(v_i)_{i\in S_y}\!\in\!\cN_r'|_{V_{y;S_y}}\!:
\rho_{y;S_y;i}(v_i)\!<\!\ve\big(\pi_r(\!(v_i)_{i\in S_y})\!\big)~\forall\,i\!\in\!I\big\}
~~\forall~I\!\subset\!S_y,\,y\!\in\!\cA_r;\\
\be_y\!: \cN_r'|_{V_{y;S_y}}\lra\R^{\ge0}, \quad
\be_y\big(\!(v_i)_{i\in S_y}\big)=
\prod_{j\in S_y}\!\!\be\big(\rho_{y;S_y;j}(v_j)/\ve\big(\pi_r(\!(v_i)_{i\in S_y})\!\big)\!\big)
\quad\forall\,y\!\in\!\cA_r.
\end{gather*}
By the first condition in~\eref{UxcRcond_e}, $\be_y\!=\!\be_{y'}$ on 
$\cN_r'|_{V_{y;S_y}}\!\cap\!\cN_r'|_{V_{y';S_{y'}}}$ for all $y,y'\!\in\!\cA_r$.
Thus, the function
$$\be_r\!:\cN_r'\lra\R^{\ge0}, \qquad \be_r(v)=\be_y(v)
\quad\forall~v\!\in\!\cN_r'|_{V_{y;S_y}},\,y\!\in\!\cA_r,$$
is well-defined and smooth.
It satisfies
\BE{hrprp_e}\be_r|_{U_r''}=1, \quad \be_r|_{\cN_r'-U_r'}=0, \quad
\be_r\big(\!(v_i)_{i\in S_y}\big)\!=\!\be_r\big(\!(v_i)_{i\in I}\big)
~~\forall~[(v_i)_{i\in S_y}]\!\in\!U_{y;I},\,I\!\subset\!S_y,\,y\!\in\!\cA_r.\EE

\vspace{.15in}

\noindent
For $y\!\in\!\cA_r$, $i\!\in\!S_y$, and $I\!\subset\!\wt{S}_{(y,i)}$ so that $0\!\not\in\!I$, 
define
\BE{rhoscaldfn2_e}\begin{split}
f_{(y,i);I}\big(\wt\Psi_{(y,i);0\sqcup I;I}(v)\!\big)&
=\be_r\big(\pi_2(v)\!\big)h_{\bE}(\wt\pi_0(v))
\!+\!\big(1\!-\!\be_r\big(\pi_2(v)\!\big)\!\big)\wt\rho_0(v)\\
&\hspace{2in}
\forall\,v\!\in\!\cN_{(y,i);0\sqcup I;I}'\!
\subset\!\ga\big|_{\wt{V}_{(y,i);0\sqcup I}}.
\end{split}\EE
By~\eref{hbEprp_e} and the last property in~\eref{hrprp_e},
the restrictions of the differentials $\fD\wt\Psi_{(y,i);I;I'}$ with 
\hbox{$I'\!\subset\!I\!\subset\!\wt{S}_{(y,i)}$} such that $0\!\not\in\!I$
to $\wt\Psi_{(y,i);I}^{-1}(U_{y;I}\!\cap\!U_{y;0\sqcup I})$
preserve the metrics~\eref{wtrhodfn_e}.
By~\eref{hbEprp_e} and the first property in~\eref{hrprp_e}, 
the restrictions of the differentials $\fD\wt\Psi_{(y,i);I;I'}$ with 
$I'\!\subset\!I\!\subset\!\wt{S}_{(y,i)}$ such that $0\!\in\!I$
to $\wt\Psi_{(y,i);I}^{-1}(U_{y;I}\!\cap\!U_r'')$
preserve the metrics~\eref{wtrhodfn_e}.
By~\eref{rhoscaldfn2_e}, 
the functions~$f_{(y,i);I}$ satisfy~\eref{fcond_e1a} and~\eref{fcond_e1b} whenever $0\!\not\in\!I$.
By the middle property in~\eref{hrprp_e}, the isomorphism~\eref{cORisom_e}
identifies the restriction of~$\wt\rho_{(y,i);I;j}$ to 
$\cN_{\wt{V}_{(y,i);I-j}}\wt{V}_{(y,i);I}|_{\wt{V}_{(y,i);I}-U_r'}$
with the restriction of the metric~$\rho_{y;I;j}$ to~$\cN_{V_{y;I-j}}V_{y;I}|_{V_{y;I}-U_r'}$
whenever
\hbox{$j\!\neq\!0\!\not\in\!I$}.\\

\noindent
By \cite[Lemma~5.8]{SympDivConf}, we can shrink the domains~$\Dom(\wt\Psi_{(y,i);I})$ 
of~$\wt\Psi_{(y,i);I}$ with $i\!\in\!S_y$ and $I\!\subset\!\wt{S}_{(y,i)}$ 
to open neighborhoods~$\cN_{(y,i);I}''$ of $\wt{V}_{(y,i);I}\!\subset\!\Dom(\wt\Psi_{(y,i);I})$ 
so~that 
$$\wt\Psi_{(y,i);I}(\cN_{(y,i);I}'')\subset
\begin{cases}
U_{y;I}\!\cap\!U_r'',&\hbox{if}~0\!\in\!I;\\
U_{y;I},&\hbox{if}~0\!\not\in\!I;
\end{cases}$$
and the collection $(\wt\Psi_{(y,i);I}|_{\cN_{(y,i);I}''})_{I\subset\wt{S}_{(y,i)}}$ 
is still a regularization for $\{\wt{V}_{(y,i);j}\}_{j\in\wt{S}_{(y,i)}}$ 
in~$\wt{U}_{y,i}$.
Replacing~$\wt\Psi_{y;I}$ with~$\wt\Psi_{y;I}|_{\cN_{y;I}''}$ in~\eref{cRblowup_e}
whenever $y\!\in\!\wt\cA\!-\!\cA$, 
we obtain an almost complex regularization for~$\wt{V}$ in $(\wt{X},\wt{J})$.

\subsection{Proof of Lemma~\ref{blowuplog_lmm}}
\label{blowuplog_subs}

\noindent
The substance of the last claim of Lemma~\ref{blowuplog_lmm} is that the canonical identification
\BE{blowlogid_e}T_{\wt{\cR}}\wt{X}(-\log\wt{V})\big|_{\wt{X}-\bE}
=T_\cR X(-\log V)\big|_{X-V^{(r)}}\EE
extends smoothly to a bundle homomorphism as in~\eref{blowuplog_e} and that 
this bundle homomorphism is an isomorphism.
For each $y\!\in\!\cA_r$, $i\!\in\!S_y$, and $I\!\subset\!\wt{S}_{(y,i)}$ with $0\!\in\!I$,
we verify this over the open subspace $U_{(y,i);I}^{\circ}\!\subset\!\wt\cN_0'$ defined
as in~\eref{UIcirc_e} via the regularization~$\wt\cR$ 
constructed in Section~\ref{CblowupReg_subs}.\\

\noindent
Let $y\!\in\!\cA_r$, $i\!\in\!S_y$, and $I\!\subset\!\wt{S}_{(y,i)}$ be as above,
\hbox{$[v]\!\equiv\![(v_i)_{i\in S_y-I}]\!\in\!\wt{V}_{(y,i);I}$} with $v_i\!\neq\!0$
for all $i\!\in\!S_y\!-\!I$, and $u\!\in\!(\C v)^*\!-\!\{0\}$.
Let $a\!\in\!\C^*$ and $v'\!\in\!\cN_{y;S_y;S_y-I}$ be sufficiently small.
Let 
$$h_{\wt\na^{((y,i);I;0)};av}\!:T_{[v]}\wt{V}_{(y,i);I}\lra T_{a v}\ga
\quad\hbox{and}\quad
h_{\na^{(r)};av}\!:T_{[v]}\wt{V}_{(y,i);I}\lra T_{a v}(\cN_XV^{(r)})$$
be the injective homomorphisms determined by the connections~$\wt\na^{((y,i);I;0)}$
and~$\na^{(r)}$ as in~\eref{nhndfn_e}.
The isomorphism~\eref{TbEsplit_e} gives
$$T_{[v]}\wt{V}_{(y,i);I}=T_{\pi(v)}V^{(r)}\!\oplus\!(\C v)^*\!\otimes_{\C}\!(\C v)^{\perp},$$
where $(\C v)^{\perp}\!\subset\!\cN_{y;S_y;I}|_{\pi(v)}$ is the $\rho_r$-orthogonal complement
of~$\C v$.
By~\eref{wtnadfm_e} and the sentence containing~\eref{TcNident_e}, 
\begin{equation*}\begin{split}
h_{\wt\na^{((y,i);I;0)};av}\big(w,u\!\otimes\!v^{\perp}\big)
&=h_{\na^{(r)};av}(w)\!+\!u(av)v^{\perp}\!+\!
\th(w,u\!\otimes\!v^{\perp})(av)\\
&=h_{\na^{(r)};av}(w)\!+\!a\big(u(v)v^{\perp}\!+\!\th(w,u\!\otimes\!v^{\perp})v\big)
\end{split}\end{equation*}
for some 1-form $\th$ on~$\bE$.
With the notation as in~\eref{varTII_e}, we thus obtain
\begin{equation*}\begin{split}
&\vt_{(y,i);(I-0)I}\big(\!\big(\wt\Psi_{(y,i);I}(av,u\!\otimes\!v'),
(av,u\!\otimes\!v'),(w,u\!\otimes\!v^{\perp})\!\big),(c_i)_{i\in I}\!\big)\\
&\hspace{1in}=
\vt_{y;(I-0)S_y}\big(\!\big(\Psi_{y;S_y}(av\!+\!u(v)v'),av\!+\!u(v)v',w\big),(c_i)_{i\in S_y}\big),
\end{split}\end{equation*}
with $c_i\!\equiv\!c_i([v];w,u\!\otimes\!v^{\perp},c_0)\in\!\C$ for $i\!\in\!S_y\!-\!I$ defined by 
$$\sum_{i\in S_y-I}\!\!\!c_iv_i
=u(v)v^{\perp}\!+\!\big(\th(w,u\!\otimes\!v^{\perp})\!+\!c_0\big)v.$$
The identification~\eref{blowlogid_e} thus extends smoothly over $[v]$ as 
the vector space isomorphism
\begin{gather*}
T_{[v]}\bE\!\oplus\!\C^I\lra T_{\pi([v])}V^{(r)}\!\oplus\!\C^{S_y-I}\!\oplus\!\C^{I-0},\\
\big(w,(c_i)_{i\in I}\big)\lra 
\big(\nd_{[v]}\pi(w),\big(c_i([v];w,c_0)\!\big)_{i\in S_y-I},(c_i)_{i\in I-0}\big).
\end{gather*}

\subsection{Symplectic setting}
\label{SymplBl_subs}

\noindent
Suppose now that $\om$ is a symplectic form on~$(X,V)$
and $\cR$ is an $\om$-regularization for~$V$ in~$X$ in the sense of
Definition~\ref{NCDregul_dfn}.
The smooth submanifold $V^{(r)}\!\subset\!X$ is then symplectic.
Let $\om_r\!\equiv\!\wt\om$ be the closed 2-form on~$\cN_XV^{(r)}$ determined by 
$\om|_{V^{(r)}}$, $\Om\!\equiv\!\om|_{\cN_r}$, 
and~$\na^{(r)}$ as in~\eref{ombund_e2}.
Let $J$ be an $\cR$-compatible almost complex structure on~$X$,
$$\pi\!:\big(\wt{X}\!\equiv\!\big(\!(X\!-\!V^{(r)})\!\sqcup\!\wt\cN_0'\big)\!\big/\!\!\sim,
\wt{J}\big)\lra (X,J)$$
be the corresponding almost complex blowup with the exceptional divisor $\bE\!\equiv\!\P(\cN_XV^{(r)})$
as in Section~\ref{AlCBl_sub}, and $\ov{V}\!\subset\!\wt{X}$ be the proper transform of~$V$.
Suppose also that there exists \hbox{$\ep\!\in\!(0,1)$} such~that
\BE{cNrepcond_e}\cN_r(2\ep)\!\equiv\!\big\{v\!\in\!\cN_XV^{(r)}\!: \rho_r(v)\!<\!2\ep\big\}
\subset\Psi_r^{-1}(U_r')\subset\cN_r'.\EE
This is automatically the case if $V^{(r)}$ is compact.\\

\noindent
The subgroup $S^1\!\subset\!\C^*$ acts on $\cN_r(2\ep)\!\times\!\C$~by
\BE{S1actdfn_e}u\!:\cN_r(2\ep)\!\times\!\C\lra\cN_r(2\ep)\!\times\!\C, \qquad
u\!\cdot\!(v,c)=(uv,c/u).\EE
This $S^1$-action preserves the submanifolds
\begin{equation*}\begin{split}
\wt\cZ(\ep)&\equiv
\big\{\!(v,c)\!\in\!\cN_r(2\ep)\!\times\!\C\!:\rho_r(v)\!-\!|c|^2\!=\!\ep\big\} 
\quad\hbox{and}\quad
\wt\bE_{\ep}\equiv
\big\{\!(v,0)\!\in\!\cN_r(2\ep)\!\times\!\C:\rho_r(v)\!=\!\ep\big\}.
\end{split}\end{equation*}
We identify $\wt\bE_{\ep}$ with the $\sqrt\ep$-sphere bundle of~$\cN_XV^{(r)}$ 
in the obvious way. 
Let
\begin{gather}\notag
\cZ(\ep)=\wt\cZ(\ep)/S^1, \quad \bE_{\ep}\!=\!\wt\bE_{\ep}/S^1, \quad
X_{\ep}=\big(\!\big(X\!-\!\Psi_r(\ov{\cN_r(\ep)})\!\big)
\!\sqcup\!\cZ(\ep)\!\big)\!\big)\!\big/\!\!\sim,\\
\label{SympBlid_e}
\cZ(\ep)\!-\!\bE_{\ep}\ni\big[v,\sqrt{\rho_r(v)\!-\!\ep}\big]\sim\Psi_r(v)\in 
X\!-\!\Psi_r(\ov{\cN_r(\ep)})
\quad\forall~v\!\in\!\cN_r(2\ep)\!-\!\ov{\cN_r(\ep)}.
\end{gather}
By the Symplectic Reduction Theorem \cite[Theorem~23.1]{daSilva},
there is a unique symplectic form~$\om_{r;\ep}$ on the smooth manifold $\cZ(\ep)$ 
so~that
\BE{SympRed_e4} q_{\ep}^*\om_{r;\ep}=
\big(\pi_1^*\om_r\!+\!\pi_2^*\om_{\C}\big)\!\big|_{\wt\cZ(\ep)},\EE
where $q_{\ep}\!:\wt\cZ(\ep)\!\lra\!\cZ(\ep)$ is the quotient projection,
$$\pi_1,\pi_2\!:\cN_r(2\ep)\!\times\!\C\lra \cN_r(2\ep),\C$$
are the component projections, and $\om_{\C}$ is the standard symplectic form on~$\C$.
Since $\Psi_r^*\om\!=\!\om_r$ on~$\cN_r(2\ep)$ and the $S^1$-action~\eref{S1actdfn_e} preserves 
the 2-form $\pi_1^*\om_r\!+\!\pi_2^*\om_{\C}$,
\eref{SympRed_e4} implies that the identification~\eref{SympBlid_e} intertwines~$\om_{r;\ep}$
and~$\om$.
We thus obtain a symplectic form~$\om_{\ep}$ on~$X_{\ep}$ such~that 
$$\om_{\ep}\big|_{X-\Psi_r(\ov{\cN_r(\ep)})}=\om\big|_{X-\Psi_r(\ov{\cN_r(\ep)})}
\qquad\hbox{and}\qquad
\om_{\ep}\big|_{\cZ(\ep)}=\om_{r;\ep}\,.$$
It restricts to a symplectic form on $\bE_{\ep}\!\subset\!X_{\ep}$.\\

\noindent
The $\P^{r-1}$-fiber bundle $\pi_{\ep}\!:\bE_{\ep}\!\lra\!V^{(r)}$ is canonically identified 
with \hbox{$\pi\!:\bE\!\lra\!V^{(r)}$}.
This identification canonically lifts to an identification~of the complex line bundle
$$\wt\pi_0\!:\cN_{X_{\ep}}\bE_{\ep}=\wt\bE_{\ep}\!\times_{S^1}\!\C
\lra \wt\bE_{\ep}/S^1\!\equiv\!\bE_{\ep},
\qquad (v,c)\sim\big(uv,c/u)~~\forall\,(v,c)\!\in\!\wt\bE_{\ep}\!\times\!\C,\,u\!\in\!S^1,$$
with the tautological line bundle $\ga\!\subset\!\pi^*\cN_XV^{(r)}$ as in~\eref{gadfn_e}.
The differential
$$\nd q_{\ep}\!: \wt\bE_{\ep}\!\times\!\C\!=\!\cN_{\wt\cZ(\ep)}\wt\bE_{\ep}
\lra q_{\ep}^*\cN_{X_{\ep}}\bE_{\ep}$$
is an isomorphism of complex line bundles.
It intertwines the fiberwise symplectic form~$\om_{\C}$ with the fiberwise symplectic~form 
$$\om_{\ep}\big|_{\cN_{X_{\ep}}\bE_{\ep}}=\om_{r;\ep}\big|_{\cN_{X_{\ep}}\bE_{\ep}}$$
on~$\cN_{X_{\ep}}\bE_{\ep}$ induced by~$\om_{\ep}$ as below Definition~\ref{smreg_dfn}.
Thus, the complex orientation on~$\cN_{X_{\ep}}\bE_{\ep}$ agrees with the orientation
induced by the symplectic form~$\om_{\ep}$.
It is straightforward to see that the~map
$$\wt\Psi_{\ep;0}\!:\wt\cN_{\ep;0}'\!\equiv\!
\big\{[v,c]\!\in\!\cN_{X_{\ep}}\bE_{\ep}\!:|c|^2\!<\!\ep\big\}\lra 
\cZ(\ep)\subset X_{\ep}, \quad  
\wt\Psi_{\ep;0}\big([v,c]\big)=\big[\sqrt{1\!+\!|c|^2/\ep}\,v,c\big],$$
is a well-defined smooth regularization for~$\bE_{\ep}$ in~$X_{\ep}$.

\begin{rmk}\label{SymplEreg_rmk}
Via the above identification of the complex line bundles~$\cN_{\wt{X}_{\ep}}\bE$ and~$\ga$, 
the Hermitian metric~$\rho_r$ and connection~$\na^{(r)}$ on~$\cN_XV^{(r)}$ 
determine a connection~$\wt\na^{(0)}$ on~$\cN_{\wt{X}_{\ep}}\bE$,
as in Section~\ref{SmothCblowup_subs}.
The standard Hermitian metric on~$\C$ determines a Hermitian metric~$\wt\rho_0$ 
on~$\cN_{\wt{X}_{\ep}}\bE$ compatible with~$\wt\na^{(0)}$ and the fiberwise symplectic 
form~$\wt\om_{\ep}|_{\cN_{\wt{X}_{\ep}}\bE}$.
The metric~$\wt\rho_0$ corresponds to the Hermitian metric $\ep^{-1}\pi^*\rho_r|_{\ga}$ via 
the above identification of the complex line bundles~$\cN_{\wt{X}_{\ep}}\bE$ and~$\ga$. 
For the record, we show in Section~\ref{BlowPf_subs} that 
$(\!(\wt\rho_0,\wt\na^{(0)}),\wt\Psi_{\ep;0})$ is an $\om_{\ep}$-regularization for~$\bE_{\ep}$
in~$X_{\ep}$ in the sense of Definition~\ref{sympreg1_dfn}.
\end{rmk}

\noindent
The open subspaces $X\!-\!\Psi_r(\ov{\cN_r(\ep)})$ of~$X$ and 
$X_{\ep}\!-\!\bE_{\ep}$ of~$X_{\ep}$ are canonically identified.
Let $\ov{V}_{\ep}\!\subset\!X_{\ep}$ be~the closure of~$V\!-\!\Psi_r(\ov{\cN_r(\ep)})$
and $\wt{V}_{\ep}\!=\!\bE_{\ep}\!\cup\!\ov{V}_{\ep}$.
We now show that~$\wt{V}_{\ep}$ is an NC symplectic divisor in~$(X_{\ep},\om_{\ep})$
with a collection $(\wt{U}_y,\{\wt{V}_{y;i}\}_{i\in\wt{S}_y})_{y\in\wt\cA}$ of charts.
If $y\!\in\!\cA\!-\!\cA_r$, then
$$U_y\subset X\!-\!V^{(r)}=X_{\ep}\!-\!\bE_{\ep}.$$
In this case, we again take 
$$\wt{S}_y=S_y  \qquad\hbox{and}\qquad 
\big(\wt{U}_y,\{\wt{V}_{y;i}\}_{i\in\wt{S}_y}\big)=(U_y,\{V_{y;i}\}_{i\in S_y}\big).$$

\vspace{.15in}

\noindent
As before, we identify $\cN_r'\!\subset\!\cN_XV^{(r)}$ with $\Psi_r(\cN_r')\!\subset\!X$ via~$\Psi_r$. 
Let 
$$\wt\cN_{r;\ep}'=\big(\cN_r'\!-\!\ov{\cN_r(\ep)}\big)\!\cup\!\cZ(\ep)\subset X_{\ep}$$
and $\wt\pi_{r;\ep}\!:\wt\cN_{r;\ep}'\!\lra\!V^{(r)}$ be the projection induced by~$\pi_r$. 
Suppose $y\!\in\!\cA_r$.
For  each $i\!\in\!S_y$, 
define $\wt{S}_{(y,i)}$, $\cN_{y;i}$, and $\wt{V}_{(y,i);0}$ as at the end
of Section~\ref{AlCBl_sub},
with~$\bE$ replaced by~$\bE_{\ep}$, and~set
$$\wt{U}_y=\wt\cN_{r;\ep}'|_{V_{y;S_y}}\equiv\wt\pi_{r;\ep}^{-1}(V_{y;S_y}) , \qquad
\ov{V}_{y;i}=\ga|_{\P\cN_{y;i}}\!\cap\!\wt{U}_y\,.$$
For $j\!\in\!S_y\!-\!i$, let $\wt{V}_{(y,i);j}=\ov{V}_{y;j}\!\cap\!\wt{U}_{(y,i)}$
as before.
We again have~\eref{BlDecomp_e}, with~$\ov{V}$ replaced by~$\ov{V}_{\ep}$ in the last statement.
In this case,
$\{\wt{V}_{(y,i);j}\}_{j\in\wt{S}_{(y,i)}}$ is a transverse collection of codimension~2 
symplectic submanifolds of $(\wt{U}_{(y,i)},\om_{r;\ep})$ 
so that their intersection and symplectic orientations agree.
Thus, $\wt{V}_{\ep}\!\cap\!\wt{U}_{(y,i)}$ is an SC symplectic divisor 
in $(\wt{U}_{(y,i)},\om_{r;\ep})$
in the sense of Definition~\ref{SCD_dfn}, 
$\wt{V}_{\ep}$ is an NC symplectic divisor in~$(X_{\ep},\om_{\ep})$, and
$(\wt{U}_y,\{\wt{V}_{y;i}\}_{i\in\wt{S}_y})_{y\in\wt\cA}$  is an atlas of local charts
for~$\wt{V}_{\ep}$.\\

\noindent
Let $f_{\ep}\!:\R\!\lra\!\R$ be a smooth function so that 
$$f_{\ep}'(t)>0~~\forall\,t\!\in\!\R, \qquad
f_{\ep}(t)=\begin{cases}\sqrt{\ep\!+\!t^2/\ep},&\hbox{if}~t\!\le\!\ep/2;\\
t,&\hbox{if}~t\!\ge\!5\sqrt\ep/4.
\end{cases}$$
The map
$$\wt\pi_{\ep}\!:\wt{X}\lra X_{\ep}, \quad
\wt\pi_{\ep}(\wt{x})=\begin{cases}\wt\Psi_{\ep;0}(\wt{x})\!\in\!\cZ(\ep),
&\hbox{if}~\wt{x}\!\in\!\wt\cN_0',~\rho_r(\wt{x})\!<\!\ep^2/4;\\
\Psi_r\big(f_{\ep}(\sqrt{\rho_r(\wt{x})})\frac{\wt{x}}{\sqrt{\rho_r(\wt{x})}}\big),
&\hbox{if}~\wt{x}\!\in\!\wt\cN_0',~\rho_r(\wt{x})\!>\!0;\\
\wt{x}\!\in\!X\!-\!\Psi_r(\ov{\cN_r(25\ep/16)})),
&\hbox{if}~\wt{x}\!\in\!X\!-\!\Psi_r(\ov{\cN_r(25\ep/16)});
\end{cases}$$
is then an orientation-preserving diffeomorphism.
It identifies the NC almost complex divisor $\wt{V}\!\subset\!\wt{X}$ 
with the NC symplectic divisor $\wt{V}_{\ep}\!\subset\!X_{\ep}$.
Thus,
$$\wt\pi_{\ep}^*\om_{\ep}\in \Symp^+(\wt{X},\wt{V}).$$ 

\vspace{.15in}

\noindent
Let $\wt\cR$ be a regularization for $\ov{V}$ in $(\wt{X},\wt{J})$ 
obtained as in Section~\ref{CblowupReg_subs}
from the regularization~$\cR$ for~$V$ in~$(X,\om)$ and thus in~$(X,J)$.
Repeated applications of \cite[Theorem~3.1]{SympDivConf},
starting from the deepest strata of $\ov{V}\!\cap\!\bE$, provide 
a smooth family $(\mu_{\tau})_{\tau\in[0,1]}$ of 1-forms on~$\wt{X}$ so~that
\begin{enumerate}[label=$\bullet$,leftmargin=*]

\item $\wt\om_{\ep;\tau}\!\equiv\!\wt\pi_{\ep}^*\om_{\ep}\!+\!\nd\mu_{\tau}\in \Symp^+(\wt{X},\wt{V})$
for all $\tau\!\in\![0,1]$;

\item $\mu_0\!=\!0$ and $\supp\,\mu_{\tau}\!\subset\!\cN_0'$ for all $\tau\!\in\![0,1]$;

\item a tuple $\wh\cR$ obtained from $\wt\cR$ by restricting the domains of the maps $\wt\Psi_{y;I}$
with $y\!\in\!\wt\cA\!-\!(\cA\!-\!\cA_r)$ is an $\wt\om_{\ep;1}$-regularization for~$\wt{V}$
in~$\wt{X}$.

\end{enumerate}
The bundle isomorphism~\eref{blowuplog_e} thus determines a homotopy class of 
isomorphisms~\eref{blowuplog_e2}
between the log tangent bundles associated with the deformation equivalence classes of~$\om$
in $\Symp^+(X,V)$ and of~$\wt\pi_{\ep}^*\om_{\ep}$ in~$\Symp^+(\wt{X},\wt{V})$.\\

\noindent
The (deformation equivalence class of the) NC symplectic divisor $\wt{V}\!\subset\!\wt{X}$ 
constructed above does not depend on the choices of~$J$, $f_{\ep}$, and~$\ep$.
Since the projection~\eref{Aux2Symp_e} is a weak homotopy equivalence, 
it does not depend on the choices of the regularization~$\cR$ and 
$\om\!\in\!\Symp^+(X,V)$ in the given equivalence class if~$V^{(r)}$ is compact.
Thus, if $V^{(r)}$ is compact, an NC symplectic divisor structure~$[\om]$ on $V\!\subset\!X$
determines an NC symplectic divisor structure~$[\wt\om]$ on $\wt{V}\!\subset\!\wt{X}$.

\subsection{Proofs of technical statements}
\label{BlowPf_subs}

\noindent
We conclude our discussion of blowups with proofs of the claims of
Lemma~\ref{bumpfun_lmm1} and Remark~\ref{SymplEreg_rmk}.

\begin{proof}[{\bf{\emph{Proof of Lemma~\ref{bumpfun_lmm1}}}}] 
Suppose $r\!\ge\!2$ and the claim is true with~$r$ replaced by $r\!-\!1$.
We denote by~$\bS_r$ the group of permutations of the homogeneous coordinates of~$\P^{r-1}$
and by~$\tau_r\!\in\!\bS_r$ the transposition of the last two coordinates.
We identify~$\P^{r-2}$ with the subspace $(Z_r\!=\!0)$ of~$\P^{r-1}$.
It is preserved by the subgroup $\bS_{r-1}\!\subset\!\bS_r$ and
by the $(S^1)^r$-action on~$\P^{r-1}$.\\

\noindent
Let $h\!:\P^{r-2}\!\lra\!\R^+$ and $\de\!\in\!\R^+$ be a smooth function and a positive number
satisfying the conditions in the lemma with~$r$ replaced by~$r\!-\!1$.
Let $U\!\subset\!\P^{r-1}$ be an open neighborhood of~$\P^{r-2}$ preserved by 
the subgroup $\bS_{r-1}\!\subset\!\bS_r$ and by the $(S^1)^r$-action
so~that 
\BE{bumpfun_e3}
U\!\cap\!\tau_r(U) \subset \big\{[Z_1,\ldots,Z_r]\!\in\!\P^{r-1}\!:
|Z_{r-1}|^2\!+\!|Z_r|^2\!\le\!\de\big(|Z_1|^2\!+\!\ldots\!+\!|Z_{r-2}|^2\big)\!\big\}.\EE
We extend $h$ over $U$ by
$$\wt{h}\!:U\lra\R^+, \qquad 
\wt{h}\big([Z_1,\ldots,Z_r]\big)=h\big([Z_1,\ldots,Z_{r-1},0]\big)
\frac{|Z_1|^2\!+\!\ldots\!+\!|Z_r|^2}{|Z_1|^2\!+\!\ldots\!+\!|Z_{r-1}|^2}\,.$$
For each permutation $g\!\in\!\bS_r$ of the homogeneous coordinates of~$\P^{r-1}$, define
$$\wt{h}_g\!: g(U)\lra\R^+, \qquad \wt{h}_g\big([Z]\big)=\wt{h}\big([g^{-1}Z]\big).$$
By the invariance assumptions on~$U$ and~$h$, $\wt{h}_{g_1}\!=\!\wt{h}_{g_2}$ 
if $g_1^{-1}g_2\!\in\!\bS_{r-1}$.\\

\noindent
Suppose $[Z_1,\ldots,Z_r]\!\in\!U\!\cap\!\tau_r(U)$.
By~\eref{bumpfun_e3} and~\eref{bumpfun_e0} with~$r$ replaced by~$r\!-\!1$,
\begin{equation*}\begin{split}
&\wt{h}_{\tau_r}\big([Z_1,\ldots,Z_r]\big)=\wt{h}\big([Z_1,\ldots,Z_{r-2},Z_r,Z_{r-1}]\big)\\
&\qquad=\wt{h}\big([Z_1,\ldots,Z_{r-2},0,0]\big)
\frac{|Z_1|^2\!+\!\ldots\!+\!|Z_{r-2}|^2\!+\!|Z_r|^2}{|Z_1|^2\!+\!\ldots\!+\!|Z_{r-2}|^2}
\frac{|Z_1|^2\!+\!\ldots\!+\!|Z_r|^2}{|Z_1|^2\!+\!\ldots\!+\!|Z_{r-2}|^2\!+\!|Z_r|^2}\\
&\qquad=\wt{h}\big([Z_1,\ldots,Z_{r-2},0,0]\big)
\frac{|Z_1|^2\!+\!\ldots\!+\!|Z_{r-1}|^2}{|Z_1|^2\!+\!\ldots\!+\!|Z_{r-2}|^2}
\frac{|Z_1|^2\!+\!\ldots\!+\!|Z_r|^2}{|Z_1|^2\!+\!\ldots\!+\!|Z_{r-1}|^2}
=\wt{h}\big([Z_1,\ldots,Z_r]\big).
\end{split}\end{equation*}
Thus, $\wt{h}_{\tau_r}\!=\!\wt{h}$ on $U\!\cap\!\tau_r(U)$.
Along with the invariance assumptions on~$U$ and~$h$, 
this implies that $\wt{h}_{g_1}\!=\!\wt{h}_{g_2}$ on $g_1(U)\!\cap\!g_2(U)$
for all $g_1,g_2\!\in\!\bS_r$.\\

\noindent
We thus obtain a well-defined smooth function
$$H\!: W\!\equiv\!\bigcup_{g\in\bS_r}\!\!g(U)\lra\R^+, \qquad
H([Z])=\wt{h}_g([Z])~~\forall\,g\!\in\!\bS_r,~[Z]\!\in\!g(U),$$
which is invariant under the $\bS_r$- and $(S^1)^r$-actions on~$\P^{r-1}$ and 
satisfies the last property in the lemma for some $\de\!\in\!\R^+$.
Let $\be\!:\P^{r-1}\!\lra\![1,0]$ be a smooth function which is invariant under these
two actions, restricts to~1 on a neighborhood of~$\P^{r-2}$, and is supported in~$W$.
The smooth function
$$\be H\!+\!1\!-\!\be\!:\P^{r-1}\lra \R^+$$
then has the desired properties for some $\de\!\in\!\R^+$.
\end{proof}

\begin{proof}[{\bf{\emph{Proof of Remark~\ref{SymplEreg_rmk}}}}] 
Let
\begin{gather*}
\wh\om_{\ep}=\wt\pi_0^*\big(\om_{\ep}|_{\bE_{\ep}}\big)\!+\!
\frac12\nd\io_{\ze_{\cN_{X_{\ep}}\bE_{\ep}}}
\big(\om_{\ep}|_{\cN_{X_{\ep}}\bE_{\ep}} \big)_{\wt\na^{(0)}},\\
\phi,\vph\!:\cN_XV^{(r)}\!-\!V^{(r)}\lra \cN_XV^{(r)}, 
\quad \phi(v)=\frac{v}{\sqrt{\rho_r(v)}}, ~~
\vph(v)=\sqrt{1+\rho_r(v)}\frac{v}{\sqrt{\rho_r(v)}},\\
m_c\!:\cN_XV^{(r)}\lra\cN_XV^{(r)}, \quad m_c(v)=cv, \qquad\forall\,c\!\in\!\C.
\end{gather*}
The composition of the restriction of $\wt\Psi_{\ep;0}$ to $\wt\cN_{\ep;0}'\!-\!\bE_{\ep}$ with
the identification in~\eref{SympBlid_e} is given~by
$$\wt\cN_{\ep;0}'\!-\!\bE_{\ep}\xlra{m_{\sqrt\ep}\circ\vph\circ m_{1/\ep}} 
\cN_r'\!-\!V^{(r)} \xlra{~\Psi_r~} X\!-\!V^{(r)}\,.$$
It thus remains to show that $\vph^*(m_{\sqrt\ep}^*\om_r)\!=\!m_{\ep}^*\wh\om_{\ep}$ on
$m_{1/\ep}(\wt\cN_{\ep;0}'\!-\!\bE_{\ep})\!\subset\!\cN_XV^{(r)}\!-\!V^{(r)}$.\\

\noindent
For each $v\!\in\!\cN_XV^{(r)}\!-\!V^{(r)}$, let 
$$\pi_v,\pi_v^{\perp}\!:
\cN_XV^{(r)}\big|_{\wt\pi_0(v)}\!=\!(\C v)^{\perp}\!\oplus\!(\C v)\lra \C v,(\C v)^{\perp}$$
be the $\rho_r$-orthogonal projections.
Let 
$$\pi_{\na}\!:T_v(\cN_XV^{(r)})\lra \cN_XV^{(r)}\big|_{\pi_r(v)}$$
be the projection corresponding to the decomposition~\eref{TcNXVsplit_e}
determined by the connection~$\na^{(r)}$.
As noted below~\eref{TcNXVsplit_e}, this splitting also encodes 
the decomposition associated with the connection~$\wt\na^{(0)}$. 
By the properties of a connection, the map~$m_c$ preserves the decomposition~\eref{TcNXVsplit_e}
for any $c\!\in\!\C^*$;
see \cite[Lemma~1.1]{anal}. 
Since the connection~$\na^{(r)}$ is compatible with the Hermitian metric~$\rho_r$,
its connection 1-form is purely imaginary in any Hermitian trivialization. 
It follows that the maps~$\phi$ and~$\vph$ also preserve the decomposition~\eref{TcNXVsplit_e};
see the proof of \cite[Lemma~1.1]{anal}.
Thus,
\BE{SymplEreg_e5} \pi_{\na}\!\circ\!m_c=m_c\!\circ\!\pi_{\na}~~\forall\,c\!\in\!\C,
\qquad \pi_{\na}\!\circ\!\phi=\phi\!\circ\!\pi_{\na},
\qquad \pi_{\na}\!\circ\!\vph=\vph\!\circ\!\pi_{\na}\,.\EE
We also note that 
\BE{SymplEreg_e7} 
m_{c*}\ze_{\cN}\!=\!\ze_{\cN}\!\circ\!m_c \quad\hbox{and}\quad
m_c^*\Om=c^2\Om \qquad\forall\,c\!\in\!\R\EE
for any vector bundle $\cN$ and a fiberwise 2-form $\Om$ on~$\cN$.\\

\noindent
Since $\wt\pi_0\!=\!q_{\ep}\!\circ\!m_{\sqrt\ep}\!\circ\!\phi\!\circ\!m_{1/\ep}$,
$\om_{\ep}|_{\C v}\!=\!\ep^{-1}\om_r|_{\C v}$,
and $\ze_{\cN_{X_{\ep}}\bE_{\ep}}\!=\!\ze_{\cN_XV^{(r)}}$ on 
\hbox{$\cN_{X_{\ep}}\bE_{\ep}\!-\!\bE_{\ep}\!=\!\cN_XV^{(r)}\!-\!V^{(r)}$},
the first identity in~\eref{SymplEreg_e5} and~\eref{SymplEreg_e7} give
\begin{equation*}\begin{split}
m_{\ep}^*\wh\om_{\ep}
&=\phi^*m_{\sqrt\ep}^*q_{\ep}^*\big(\om_{\ep}|_{\bE_{\ep}}\big)\!+\!
\frac12\nd\io_{\ze_{\cN_{X_{\ep}}\bE_{\ep}}}
m_{\ep}^*\big(\ep^{-1}\om_r|_{\cN_XV^{(r)}}\!\circ\!\pi_v\big)_{\na^{(r)}}\\
&=\phi^*(m_{\sqrt\ep}^*\om_r)+\!
\frac12\nd\io_{\ze_{\cN_XV^{(r)}}}
\big((m_{\sqrt\ep}^*\om)|_{\cN_XV^{(r)}}\!\circ\!\pi_v\!\circ\!\pi_{\na}\big).
\end{split}\end{equation*}
Thus, it is sufficient to assume that $\ep\!=\!1$ and show that 
\BE{SymplEreg_e9} 
\vph^*\om_r=\phi^*\om_r\!+\!
\frac12\nd\io_{\ze_{\cN_XV^{(r)}}}
\big(\om|_{\cN_XV^{(r)}}\!\circ\!\pi_v\!\circ\!\pi_{\na}\big)\EE
on $\cN_XV^{(r)}\!-\!V^{(r)}$.\\
 
\noindent
By the compatibility of $\om|_{\cN_XV^{(r)}}$ with the Hermitian metric~$\rho_r$
determining the projection~$\pi_v$,
$$\io_{\ze_{\cN_XV^{(r)}}}(\om|_{\cN_XV^{(r)}})
=\io_{\ze_{\cN_XV^{(r)}}}(\om|_{\cN_XV^{(r)}}\!\circ\!\pi_v).$$
We also note that
\begin{gather*}
\big(\io_{\ze_{\cN_XV^{(r)}}}(\om|_{\cN_XV^{(r)}})\!\big)\!\circ\!\phi_*
=\frac{1}{\rho_r(v)}\big(\io_{\ze_{\cN_XV^{(r)}}}(\om|_{\cN_XV^{(r)}})\!\big),\\
\vph_*\bigg(\frac{1\!+\!\rho_r(v)}{\rho_r(v)}\ze_{\cN_XV^{(r)}}\!\!\bigg)
=\ze_{\cN_XV^{(r)}}\!\circ\!\vph,
\quad  
\vph^*\big(\om|_{\cN_XV^{(r)}}\big)= \om|_{\cN_XV^{(r)}}\!+\!
\frac{1}{\rho_r(v)}\om|_{\cN_XV^{(r)}}\!\circ\!\pi_v^{\perp}.
\end{gather*}
Along with the last two identities in~\eref{SymplEreg_e5} and 
$$\om_r=\pi_r^*\big(\om|_{V^{(r)}}\big)\!+\!
\frac12\nd\io_{\ze_{\cN_XV^{(r)}}}\big(\om|_{\cN_XV^{(r)}}\!\circ\!\pi_{\na}\big),$$
the above statements give
\begin{equation*}\begin{split}
\phi^*\om_r&=\pi_r^*\big(\om|_{V^{(r)}}\big)\!+\!
\frac12\nd\bigg(\frac{1}{\rho_r(v)}
\big(\io_{\ze_{\cN_XV^{(r)}}}\om|_{\cN_XV^{(r)}}\big)\!\circ\!\pi_{\na}\!\!\bigg),\\
\vph^*\om_r&=\pi_r^*\big(\om|_{V^{(r)}}\big)\!+\!
\frac12\nd\bigg(\frac{1\!+\!\rho_r(v)}{\rho_r(v)}
\big(\io_{\ze_{\cN_XV^{(r)}}}\om|_{\cN_XV^{(r)}}\big)\!\circ\!\pi_{\na}\!\!\bigg).
\end{split}\end{equation*}
This establishes~\eref{SymplEreg_e9}.
\end{proof}

\section{NC symplectic divisors: global perspective}
\label{NCG_sec}

\noindent
An NC~divisor can also be viewed as the image of a transverse immersion~$\io$ with certain properties.
Following~\cite{SympDivConf2}, we review the global analogues of the notions of 
Sections~\ref{NCLdfn_subs} and~\ref{NCLreg_subs} in 
Sections~\ref{NCGdfn_subs} and~\ref{NCGreg_subs} below.
This global perspective leads to a more succinct notion of regularizations for 
NC~divisors and fits better with global statements, such as Theorem~\ref{TXV_thm}\ref{cTXV_it}.
The local and global perspectives are shown to be equivalent in \cite[Lemma~3.5]{SympDivConf2}.

\subsection{Definition}
\label{NCGdfn_subs}

\noindent
For a finite set $I$, denote by $\bS_I$ the symmetric group of permutations of the elements of~$I$.  
For $k\!\in\!\Z^{\ge0}$, denote by $\bS_k\!\equiv\!\bS_{[k]}$ the $k$-th symmetric group. 
For $k'\!\in\![k]$, there is a natural subgroup
$$\bS_{k'}\times\!\bS_{[k]-[k']}\subset \bS_k\,.$$
We denote its first factor by $\bS_{k;k'}$ and the second by $\bS_{k;k'}^c$.
For each $\si\!\in\!\bS_k$ and $i\!\in\![k]$, let $\si_i\!\in\!\bS_{k-1}$
be the permutation obtained from the bijection
\BE{SkSk-1_e} [k]\!-\!\{i\}\lra [k]\!-\!\{\si(i)\}, \qquad j\lra \si(j),\EE
by identifying its domain and target with $[k\!-\!1]$ in the order-preserving fashions.\\ 

\noindent
For any map $\io\!:\wt{V}\!\lra\!X$ and $k\!\in\!\Z^{\geq 0}$, let
\BE{tViotak_e}
\wt{V}_{\io}^{(k)}=\big\{(x,\wt{v}_1,\ldots,\wt{v}_k)\!\in\!X\!
\times\!(\wt{V}^k\!-\!\De_{\wt{V}}^{(k)})\!:\,\io(\wt{v}_i)\!=\!x~\forall\,i\!\in\![k]\big\},\EE
where $\De_{\wt{V}}^{(k)}\!\subset\!\wt{V}^k$ is the big diagonal
(at least two of the coordinates are the same).
Define
\begin{gather}
\label{iotak_e}  
\io_k\!:\wt{V}_{\io}^{(k)}\lra  X, \qquad \io_k(x,\wt{v}_1,\ldots,\wt{v}_k)=x,\\
\label{Xiotak_e}
V_{\io}^{(k)}=\io_k(\wt{V}_{\io}^{(k)})
=\big\{x\!\in\!X\!:\,\big|\io^{-1}(x)\big|\!\ge\!k\big\}.
\end{gather}
For example,
$$\wt{V}_{\io}^{(0)},V_{\io}^{(0)}=X, \qquad
\wt{V}_{\io}^{(1)}\approx\wt{V}, \qquad  V_{\io}^{(1)}=\io(\wt{V}).$$

\vspace{.2in}

\noindent
For $k',k\!\in\!\Z^{\ge0}$ and $i\!\in\!\Z^+$ with $i,k'\!\le\!k$, define
\begin{alignat}{2}
\label{whiokk_e}
\wt\io_{k;k'}\!:\wt{V}_{\io}^{(k)}&\lra\wt{V}_{\io}^{(k')}, &\quad 
\wt\io_{k;k'}(x,\wt{v}_1,\ldots,\wt{v}_k)&=(x,\wt{v}_1,\ldots,\wt{v}_{k'}),\\
\label{iokcjdfn_e}
\wt\io_{k;k-1}^{(i)}\!: \wt{V}_{\io}^{(k)}&\lra\wt{V}_{\io}^{(k-1)}, &\quad
\wt\io_{k;k-1}^{(i)}(x,\wt{v}_1,\ldots,\wt{v}_k)&=(x,\wt{v}_1,\ldots,\wt{v}_{i-1},\wt{v}_{i+1},\ldots,\wt{v}_k),\\
\label{iokjdfn_e}
\wt\io_k^{(i)}\!:  \wt{V}_{\io}^{(k)} &\lra \wt{V}, &\quad 
\wt\io_k^{(i)}(x,\wt{v}_1,\ldots,\wt{v}_k)&=\wt{v}_i.
\end{alignat}
For example, 
\begin{alignat*}{2}
\wt\io_{k;k'}\!=\!\wt\io_{k'+1;k'}^{(k'+1)}\!\circ\!\ldots\!\circ\!\wt\io_{k;k-1}^{(k)}
\!&:\wt{V}_{\io}^{(k)}\lra\wt{V}_{\io}^{(k')}, &\qquad
\wt\io_{k;1}\!\approx\!\wt\io_k^{(1)}\!&:
\wt{V}_{\io}^{(k)}\lra\wt{V}_{\io}^{(1)}\!\approx\!\wt{V}, \\
\wt\io_{k;0}\!=\!\io_k\!&:\wt{V}_{\io}^{(k)}\lra\wt{V}_{\io}^{(0)}\!=\!X, &\qquad
\wt\io_{1;0}\!\approx\!\io\!&:\wt{V}_{\io}^{(1)}\!\approx\!\wt{V}\lra X.
\end{alignat*}
We define an $\bS_k$-action on $\wt{V}_{\io}^{(k)}$ by requiring that 
\BE{SkVk_e} 
\wt\io_k^{(i)}=\wt\io_k^{(\si(i))}\!\circ\!\si\!: 
\wt{V}_{\io}^{(k)} \lra \wt{V}\EE
for all $\si\!\in\!\bS_k$ and $i\!\in\![k]$.
The diagrams
\BE{Vkdiag_e}
\begin{split}
\xymatrix{\wt{V}_{\io}^{(k)~{}} \ar[rr]^{\wt\io_k^{(i)}} \ar[d]^{\wt\io_{k;k-1}^{(i)}}  
\ar@/_2pc/[dd]_{\wt\io_{k;k'}}  \ar@/^1pc/[rrdd]^{\io_k}
&& \wt{V} \ar[dd]^{\io}  \\
\wt{V}_{\io}^{(k-1)} \ar[rrd]^<<<<<<{\!\io_{k-1}}  
\ar@{-->}[d]^{\wt\io_{k-1;k'}} &&  \\
\wt{V}_{\io}^{(k')}  \ar[rr]^{\io_{k'}}& &X}
\end{split} \hspace{.5in}
\begin{split}
\xymatrix{\wt{V}_{\io}^{(k)}\ar[rr]^{\si} \ar[d]_{\wt\io_{k;k-1}^{(i)}}&&
\wt{V}_{\io}^{(k)}\ar[d]^{\wt\io_{k;k-1}^{(\si(i))}} \\
\wt{V}_{\io}^{(k-1)} \ar[rr]^{\si_i}&&  \wt{V}_{\io}^{(k-1)}}
\end{split}\EE
of solid arrows then commute; the entire first diagram commutes if $i\!>\!k'$.

\begin{eg}\label{NCdiv_eg}
If $V\!=\!\bigcup_{i\in S}\!V_i$ is an SC symplectic divisor, then 
$$\io\!:\wt{V}\!\equiv\!\bigsqcup_{i\in S}\!V_i\lra X$$
restricts by the inclusion on each $V_i$, and 
$$\wt{V}^{(k)}_{\io}\approx \bigsqcup_{I\subset S,\,|I|=k}
\hspace{-.2in}\big(V_I\!\times\!\vec{I}\big),$$
where $\vec{I}\!\subset\!I^k$ is the subcollection of tuples with all entries distinct. 
The action of $\si\!\in\!\bS_k$ on $\wt{V}^{(k)}_{\io}$ is by reordering 
the element of each tuple in~$\vec{I}$:
$$(i_1,\ldots,i_k)\lra (i_{\si^{-1}(1)},\ldots,i_{\si^{-1}(k)}).$$
The maps $\wt\io_{k;k'}\!:\wt{V}_{\io}^{(k)}\lra\wt{V}_{\io}^{(k')}$
in~\eref{whiokk_e} are given~by 
$$V_I\!\times\! \vec{I}\!\ni\!\big(x,(i_1,\ldots,i_k)\!\big) 
\lra \big(x,(i_1,\ldots,i_{k'})\!\big)\!\in\!V_J\!\times\!\vec{J}, 
\qquad \tn{where}~\vec{J}=\vec{I}\!-\!\big\{i_{k'+1},\ldots,i_k\big\}.$$
\end{eg}

\noindent
A smooth map $\io\!:\wt{V}\lra\!X$ is an \sf{immersion} if
the differential $\nd_x\io$ of~$\io$ at~$x$ is injective for all $x\!\in\!\wt{V}$.
This implies~that 
$$\codim\,\io\equiv\dim X-\dim V\ge0.$$
Such a map has a well-defined normal bundle,
\BE{Niota_e}\cN\io\equiv \io^*TX\big/\Im(\nd\io)\lra \wt{V}\,.\EE
If $\io$ is a closed immersion, then the subspace $V_{\io}^{(k)}\!\subset\!X$
and $\wt{V}_{\io}^{(k)}\!\subset\!X\!\times\!\wt{V}^k$ are closed.\\

\noindent
An immersion $\io\!:\wt{V}\lra\!X$  is \sf{transverse} if the homomorphism
$$T_xX\oplus\bigoplus_{i=1}^k T_{\wt{v}_i}\wt{V}\lra \bigoplus_{i=1}^kT_xX, \quad
\big(w,(w_i)_{i\in[k]}\big)\lra \big(w\!+\!\nd_{\wt{v}_i}\io(w_i)\big)_{i\in[k]}\,,$$
is surjective for all $(x,\wt{v}_1,\ldots,\wt{v}_k)\!\in\!\wt{V}_{\io}^{(k)}$ and $k\!\in\!\Z^+$.
By the Inverse Function Theorem, in such a case
\begin{enumerate}[label=$\bullet$,leftmargin=*]

\item each $\wt{V}_{\io}^{(k)}$ is a smooth submanifold of $X\!\times\!\wt{V}^k$,

\item\label{hatimerssion_it} the maps $\wt\io_{k;k-1}$ in \eref{whiokk_e} and
the maps~\eref{iokcjdfn_e} are transverse immersions,

\item the homeomorphisms~$\si$ of $\wt{V}_{\io}^{(k)}$ determined by the elements of~$\bS_k$ 
as in~\eref{SkVk_e} are diffeomorphisms.

\end{enumerate}

\vspace{.15in}

\noindent
By the commutativity of the upper and middle triangles in the first diagram in~\eref{Vkdiag_e}, 
the inclusion of $\Im(\nd\io_k)$ into $\wt\io_k^{(i)*}\Im(\nd\io)$ and
the homomorphism~$\nd\io_{k-1}$  induce homomorphisms
$$
\cN\io_k\lra \io_k^{(i)*}\cN\io, \quad
\cN\wt\io_{k;k-1}^{(i)}\lra \cN\io_k \qquad\forall~i\!\in\![k].
$$
By the Inverse Function Theorem, the resulting homomorphisms
\BE{ImmcNorient_e2}
\cN\io_k\lra \bigoplus_{i\in[k]}\!\wt\io_k^{(i)*}\cN\io \qquad\hbox{and}\qquad
\cN\wt\io_{k;k-1}^{(i)}\lra \wt\io_k^{(i)*}\cN\io~~~\forall\,i\!\in\![k]\EE
are isomorphisms.
If $\wt{V}$ is the disjoint union of submanifolds $V_i\!\subset\!X$,
they correspond to the first two isomorphisms in~\eref{cNorient_e2}. 
For $\si\!\in\!\bS_k$ and $i\!\in\![k]$,
the homomorphisms~$\nd\si$ and~$\nd\si_i$ of the second diagram in~\eref{Vkdiag_e}
induces an isomorphism
\BE{Djsi_e}D_i\si\!:\cN\wt\io_{k;k-1}^{(i)}\lra \cN\wt\io_{k;k-1}^{(\si(i))}\EE
covering~$\si$.\\

\noindent
If $\io\!:\wt{V}\!\lra\!X$ is any immersion between oriented manifolds of even dimensions,
the short exact sequence of vector bundles
\BE{ImmcNorient_e1} 0\lra T\wt{V}\stackrel{\nd\io}\lra \io^*TX\lra \cN\io\lra 0\EE
over $\wt{V}$ induces an orientation on~$\cN\io$. 
If in addition $\io$ is a transverse immersion,
the orientation on~$\cN\io$ induced by the orientations of~$X$ and~$\wt{V}$ induces 
an orientation on~$\cN\io_k$ via the first isomorphism in~\eref{ImmcNorient_e2}.
The orientations of~$X$ and~$\cN\io_k$ then induce an orientation on~$\wt{V}_{\io}^{(k)}$
via the short exact sequence~\eref{ImmcNorient_e1} with $\io\!=\!\io_k$ for all $k\!\in\!\Z^+$,
which we call \sf{the intersection orientation of~$\wt{V}_{\io}^{(k)}$}.
For $k\!=\!1$, it agrees with the original orientation of~$\wt{V}$ under the canonical identification 
$\wt{V}_{\io}^{(1)}\!\approx\!\wt{V}$.\\

\noindent
Suppose $(X,\om)$ is a symplectic manifold.
If $\io\!:\wt{V}\!\lra\!X$ is a transverse immersion such that $\io_k^*\om$ 
is a symplectic form on $\wt{V}_{\io}^{(k)}$ for all $k\!\in\!\Z^+$, then
each $\wt{V}_{\io}^{(k)}$ carries an orientation induced by $\io_k^*\om$,
which we  call the $\om$-orientation.
By the previous paragraph, the $\om$-orientations of~$X$ and $\wt{V}$ 
also induce intersection orientations on all~$\wt{V}_{\io}^{(k)}$.
By definition, the intersection and $\om$-orientations of~$\wt{V}_{\io}^{(1)}$
are the same.

\begin{prp}(\cite[Proposition~3.6]{SympDivConf2})\label{NCD_prp}
Suppose $(X,\om)$ is a symplectic manifold and $\io\!:\wt{V}\!\lra\!X$ is a transverse immersion of codimension~2. Then,  $V=\iota(\wt{V})$ is an NC symplectic divisor in~$(X,\om)$
in the sense of Definition~\ref{NCD_dfn}
if and only if  $\io_k^*\om$ 
is a symplectic form on $\wt{V}_{\io}^{(k)}$ for all $k\!\in\!\Z^+$
and the intersection and $\om$-orientations of~$\wt{V}_{\io}^{(k)}$ are the same.
\end{prp}

\noindent
In the global description of an NC divisor $V\!\subset\! X$, 
the singular locus  $V_{\prt}\!\subset\!X$ is~$V_{\io}^{(2)}$.

\subsection{Regularizations}
\label{NCGreg_subs}

\noindent
Suppose $\io\!:\wt{V}\lra\!X$ is a transverse immersion and $k,k'\!\in\!\Z^{\ge0}$
with $k'\!\le\!k$. 
With the notation as in~\eref{tViotak_e}-\eref{iokjdfn_e}, define
\BE{kk'bundles}
\pi_{k;k'}\!:\cN_{k;k'}\io=\!\bigoplus_{i\in [k]-[k']}\!\!\!\!\!\cN\wt\io_{k;k-1}^{(i)}
\lra \wt{V}_{\io}^{(k)}  \quad\hbox{and}\quad
\pi_{k;k'}^c\!:\cN_{k;k'}^c\io=\!\bigoplus_{i\in [k']}\!\cN\wt\io_{k;k-1}^{(i)}
\lra \wt{V}_{\io}^{(k)}\,.\EE
By the commutativity of the first diagram in~\eref{Vkdiag_e}, 
the homomorphisms~$\nd\wt\io_{k-1;k'}$ and~$\nd\io_{k-1}$ induce homomorphisms
$$\cN_{k;k'}\io\lra \cN\wt\io_{k;k'} \qquad\hbox{and}\qquad 
\cN_{k;k'}^c\io\lra\wt\io_{k;k'}^{\,*}\cN\io_{k'}.$$
By the Inverse Function Theorem, these homomorphisms are isomorphisms.
If $\wt{V}$ is the disjoint union of submanifolds $V_i\!\subset\!X$,
they correspond to the last isomorphism in~\eref{cNorient_e2}
and the first identification in~\eref{cNtot_e}. 
For each $\si\!\in\!\bS_k$, the isomorphisms~\eref{ImmcNorient_e2} and~\eref{Djsi_e} 
induce an isomorphism
\BE{cNioksplit_e} D\si=(D_i\si)_{i\in [k]}\!: 
\cN\io_k\approx\cN_{k;0}\io\equiv\bigoplus_{i\in[k]}\!\cN\wt\io_{k;k-1}^{(i)}
\lra \bigoplus_{i\in[k]}\!\cN\wt\io_{k;k-1}^{(\si(i))}\equiv \cN_{k;0}\io\approx \cN\io_k\EE
lifting the action of $\si$ on~$\wt{V}_{\io}^{(k)}$.
The isomorphism~\eref{cNioksplit_e} permutes the components of the direct sum
so that the subbundles 
$$\cN_{k;k'}\io,\cN_{k;k'}^c\io\subset \cN\io_k$$
are invariant under the action of the subgroup  
\hbox{$\bS_{k'}\!\times\!\bS_{[k]-[k']}$} of~$\bS_k$,
but not under the action of the full group~$\bS_k$. 

\begin{dfn}\label{NCsmreg_dfn}
A \sf{regularization} for an immersion $\io\!:\wt{V}\!\lra\!X$ is a smooth map 
\hbox{$\Psi\!:\cN'\!\lra\!X$} from a neighborhood of~$\wt{V}$ in~$\cN\io$ 
such that for every $\wt{v}\!\in\!\wt{V}$,
there exist a neighborhood $U_{\wt{v}}$ of $\wt{v}$ in~$\wt{V}$ so that 
the restriction of $\Psi$ to $\cN'|_{U_{\wt{v}}}$ is a diffeomorphism onto its~image,
$\Psi(\wt{v})\!=\!\io(\wt{v})$, and the homomorphism
$$ \cN\io|_{\wt{v}}=T_{\wt{v}}^{\ver}\cN\io \lhra T_{\wt{v}}\cN\io
\stackrel{\nd_{\wt{v}}\Psi}{\lra} 
T_{\wt{v}}X\lra \frac{T_{\wt{v}}X}{\Im(\nd_{\wt{v}}\io)}\equiv\cN\io|_{\wt{v}}$$
is the identity.
\end{dfn}

\begin{dfn}\label{NCTransCollReg_dfn}
A \sf{system of regularizations for} a transverse immersion $\io\!:\!\wt{V}\!\lra\!X$
is a tuple $(\Psi_k)_{k\in\Z^{\ge0}}$, where each $\Psi_k$ 
is a regularization for the immersion~$\io_k$, such~that 
\begin{alignat}{2}\label{NCPsikk_e}
&\Psi_k\big(\cN_{k;k'}\io\!\cap\!\Dom(\Psi_k)\big)
=V_{\io}^{(k')}\!\cap\!\Im(\Psi_k) &\qquad 
&\forall~k\!\in\!\Z^{\ge0},~k'\!\in\![k],\\
\label{NCPsikk_e2}
&\hspace{.6in}\Psi_k=\Psi_k\!\circ\!D\si\big|_{\Dom(\Psi_k)} &\qquad 
&\forall~k\!\in\!\Z^{\ge0},~\si\!\in\!\bS_k;\\
\label{NCPsikk_e3}
&\quad\big\{x\!\in\!X\!:|\Psi_k^{-1}(x)|\!\ge\!k\big\}\subset\Im(\Psi_{k+1})
&\qquad &\forall~k\!\in\!\Z^{\ge0}.
\end{alignat}
\end{dfn}

\vspace{.15in}

\noindent
The stratification condition~\eref{NCPsikk_e} replaces the first condition in~\eref{Psikk_e} 
and implies~that there exists a smooth~map
\BE{Psikkprdfn_e}\begin{split}
&\Psi_{k;k'}\!: \cN_{k;k'}'\io\!\equiv\!\cN_{k;k'}\io\!\cap\!\Dom(\Psi_k)\lra\wt{V}_{\io}^{(k')}
\qquad\hbox{s.t.}\\ 
&\quad\Psi_{k;k'}\big|_{\wt{V}_{\io}^{(k)}}=\wt\io_{k;k'}\,,\quad
\Psi_k\big|_{\cN_{k;k'}'\io}=\io_{k'}\!\circ\!\Psi_{k;k'}\,;
\end{split}\EE
see Proposition~1.35 and Theorem~1.32 in~\cite{Warner}.
Similarly to~\eref{wtPsiIIdfn_e}, $\Psi_{k;k'}$ lifts to a (fiberwise) vector bundle isomorphism
\BE{fDPsikk_e}\fD\Psi_{k;k'}\!:  \pi_{k;k'}^*\cN_{k;k'}^c\io|_{\cN_{k;k'}'\io}
\lra\cN\io_{k'}\big|_{\Im(\Psi_{k;k'})}.\EE
This bundle isomorphism preserves the second splitting below in~\eref{kk'bundles} and
is $\bS_{k;k'}$-equivariant and $\bS_{k;k'}^c$-invariant. 
The condition~\eref{NCoverlap_e} below replaces~\eref{overlap_e} in the present setting.

\begin{dfn}\label{NCTransCollregul_dfn}
A \sf{refined regularization} for a transverse immersion $\io\!:\wt{V}\!\lra\!X$
is a system  $(\Psi_k)_{k\in\Z^{\ge0}}$ of regularizations for~$\io$ 
such~that 
\BE{NCoverlap_e}\begin{split}
&\Dom(\Psi_k)\subset\pi_{k;k'}^*\cN_{k;k'}^c\io|_{\cN_{k;k'}'\io}, \quad
\fD\Psi_{k;k'}\big(\Dom(\Psi_k)\big)
=\Dom(\Psi_{k'})\big|_{\Im(\Psi_{k;k'})}, \\
&\hspace{1.5in}\Psi_k=\Psi_{k'}\circ\fD\Psi_{k;k'}|_{\Dom(\Psi_k)}
\end{split}\EE
whenever $0\!\le\!k'\!\le\!k$.
\end{dfn}

\noindent
If $(\Psi_k)_{k\in\Z^{\ge0}}$ is a refined regularization 
for a transverse immersion \hbox{$\io\!:\wt{V}\!\lra\!X$},  then 
\BE{NCDcons_e2}\begin{split}
&\hspace{1in}
\cN_{k;k''}'\io,\pi_{k;k''}^*\cN_{k;k''}^c\io|_{\cN_{k;k''}'\io}\subset
 \pi_{k;k'}^*\cN_{k;k'}^c\io\big|_{\cN_{k;k'}'\io},\\ 
&\Psi_{k;k''}=\Psi_{k';k''}\circ\fD\Psi_{k;k'}\big|_{\cN_{k;k''}'\io}\,, 
\quad
\fD\Psi_{k;k''}=\fD\Psi_{k';k''}\circ
\fD\Psi_{k;k'}\big|_{\pi_{k;k''}^*\cN_{k;k''}^c\io|_{\cN_{k;k''}'\io}}
\end{split}\EE
whenever $0\!\le\!k''\!\le\!k'\!\le\!k$.\\

\noindent
Suppose $(X,\om)$ is a symplectic manifold and
$\io\!:\wt{V}\!\lra\!X$ is an immersion so that $\io^*\om$
is a symplectic form on~$V$.
The normal bundle 
$$\cN\io\equiv \frac{\io^*TX}{\Im(\nd\io)}\approx \big(\Im(\nd\io)\big)^{\om}
\equiv \big\{w\!\in\!T_{\io(\wt{v})}X\!:\,\wt{v}\!\in\!\wt{V},\,
\om\big(w,\nd_x\io(w')\big)\!=\!0~
\forall\,w'\!\in\!T_{\wt{v}}V\big\}$$
of~$\io$ then inherits a fiberwise symplectic form~$\om|_{\cN\io}$ from~$\om$.
We denote the restriction of~$\om|_{\cN\io}$ to a subbundle $L\!\subset\!\cN\io$
by~$\om|_L$.

\begin{dfn}\label{NCsympreg1_dfn}
Suppose $(X,\om)$ is a symplectic manifold,
$\io\!:\wt{V}\!\lra\!X$ is an immersion so that $\io^*\om$ is a symplectic form on~$V$,
and
$$\cN\io=\bigoplus_{i\in I}L_i$$
is a fixed splitting into oriented rank~2 subbundles.
If $\om|_{L_i}$ is nondegenerate for every $i\!\in\!I$, then
an \sf{$\om$-regularization for~$\io$} is a tuple $((\rho_i,\na^{(i)})_{i\in I},\Psi)$, 
where $(\rho_i,\na^{(i)})$ is an $\om|_{L_i}$-compatible Hermitian structure on~$L_i$
for each $i\!\in\!I$ and $\Psi$ is a regularization for~$\io$, such that 
$$\Psi^*\om=\big(\io^*\om\big)_{(\rho_i,\na^{(i)})_{i\in I}}\big|_{\Dom(\Psi)}.$$
\end{dfn}

\begin{dfn}\label{NCSCDregul_dfn}
Suppose $(X,\om)$ is a symplectic manifold and
$\io\!:\wt{V}\!\lra\!X$ is a transverse immersion of codimension~2
so that $\io_k^*\om$ is a symplectic form on $\wt{V}_{\io}^{(k)}$ for each $k\!\in\!\Z^+$.
A \sf{refined $\om$-regularization for~$\io$} is a tuple 
\BE{NCSCDregul_e} \cR\equiv\big(\cR_k\big)_{k\in\Z^{\ge0}}\equiv
\big((\rho_{k;i},\na^{(k;i)})_{i\in[k]},\Psi_k\big)_{k\in\Z^{\ge0}}\EE
such that $(\Psi_k)_{k\in\Z^{\ge0}}$ is a refined regularization for~$\io$,
$\cR_k$ is an $\om$-regularization for~$\io_k$
with respect to the splitting~\eref{cNioksplit_e} for every $k\!\in\!\Z^{\ge0}$, 
\BE{NCSCDregul_e2}\big(\rho_{k;i},\na^{(k;i)}\big) = 
\big\{D_i{\si}\big\}^{\!*}\big(\rho_{k;\si(i)},\na^{(k;\si(i))}\big)
\quad\forall\,k\!\in\!\Z^{\ge0},\,\si\!\in\!\bS_k,\,i\!\in\![k],\EE
and the induced vector bundle isomorphisms~\eref{fDPsikk_e}
are product Hermitian isomorphisms for all $k'\!\le\!k$. 
\end{dfn}

\noindent
An almost complex structure $J$ on $X$ that preserves $\Im\,\nd\io$ determines
an almost complex structure~$J_{\io;k}$ on~$\wt{V}_{\io}^{(k)}$ for every $k\!\in\!\Z^{\ge0}$,
with $J_0\!=\!J$.
The maps~$\io$, $\io_k$, $\wt\io_{k;k'}$, $\wt\io_{k;k-1}^{(i)}$, $\wt\io_k^{(i)}$,
and~$\si$ from~$\wt{V}_{\io}^{(k)}$ respect these almost complex structures.
A refined $\om$-regularization~$\cR$ for~$\io$ 
determines a fiberwise complex structure~$\fI_k$ on~$\cN\io_k$ and a splitting
$$T(\cN\io_k)\approx \pi_{k;0}^*T\wt{V}^{(k)}_\io\!\oplus\! \pi_{k;0}^*\cN\io_k\,,$$
which are preserved by the action of $\bS_k$.
Along with~$J_{\io;k}$, they determine an almost complex structure~$J_{\cR;k}$ on
the total space of~$\cN\io_k$.
We call an almost complex structure~$J$ on~$X$ compatible with an $\om$-regularization~$\cR$ 
as in~\eref{NCSCDregul_e} if
\BE{stdJ_e2} J\big(\Im\,\nd\io)\subset\Im\,\nd\io \qquad\hbox{and}\qquad
J\!\circ\!\nd\Psi_k=\nd\Psi_k\!\circ\!J_{\cR;k}\big|_{\Dom(\Psi_k)}~~\forall\,k\!\in\!\Z^{\ge0}.\EE
Under the correspondence between the local and global perspectives provided by 
Proposition~\ref{NCD_prp},
this notion is the global version of the $\cR$-compatibility defined in Section~\ref{NCLreg_subs}.

\begin{rmk}\label{NCSCDregul_rmk}
Let $(\Psi_k)_{k\in\Z^{\ge0}}$ be a refined regularization
for a transverse immersion \hbox{$\io\!:\wt{V}\!\lra\!X$} as in Definition~\ref{NCTransCollregul_dfn}.
For $k\!\in\!\Z^+$, the limit set $\ov{\Im\,\Psi_k}\!-\!\Im\,\Psi_k$ of~$\Psi_k$
is closed and disjoint from the closed subspace~$V_{\io}^{(k)}$ of~$X$.
There are thus disjoint open neighborhoods~$W_k$ of~$V_{\io}^{(k)}$ and~$W_k'$ of the limit set. 
By \cite[Lemma~5.8]{SympDivConf}, we can shrink the domains of~$\Psi_k$ so~that
$\Im\,\Psi_k\!\subset\!W_k$ for every $k\!\in\!\Z^+$ and the new collection 
$(\Psi_k)_{k\in\Z^{\ge0}}$ is still a refined regularization for~$\io$.
Each map~$\Psi_k$ is then closed (in addition to being open).
\end{rmk}

\subsection{Constructions}
\label{ConGNC_subs}

\noindent
We now describe the vector bundles
$$\cO_{\cR;X}(V)\!=\!\cO_{\cR;X}(\io)\lra X \qquad\hbox{and}\qquad
T_{\cR}X(-\log V)\!=\!T_{\cR}X(-\log\io)\lra X,$$
the section~$s_{\cR}$ of $\cO_{\cR;X}(V)$, and 
the vector bundle homomorphism~$\io_{\cR}$ on~$T_{\cR}X(-\log V)$
constructed in Section~\ref{ConLNC_subs} from the global perspective
on NC divisors and regularizations presented in Sections~\ref{NCGdfn_subs} and~\ref{NCGreg_subs}.
We fix a symplectic manifold~$(X,\om)$, a normalization \hbox{$\io\!:\wt{V}\!\lra\!X$}
of an NC symplectic divisor~$V$ as in Proposition~\ref{NCD_prp}, 
and a refined $\om$-regularization~$\cR$ for~$\io$ in $X$ as in~\eref{NCSCDregul_e}.
For the purposes of constructing a complex structure on $T_{\cR}X(-\log\io)$, we also 
fix an almost complex structure~$J$ on~$X$ compatible with~$\om$ and~$\cR$.\\

\noindent
With the notation as in \eref{tViotak_e} and~\eref{Xiotak_e}, the restrictions
$$\Psi_k\!: \Dom(\Psi_{k})\big|_{\wt{V}^{(k)}_{\io}-\Im(\io_{k+1;k})}
\lra W_k\!\equiv\!\Psi_k\Big(\Dom(\Psi_{k})\big|_{\wt{V}^{(k)}_{\io}-\Im(\io_{k+1;k})}\Big)$$
are not even $\bS_k$-covering maps. 
In other words, there could~exist 
$$v,v'\in\Dom(\Psi_k)\big|_{\wt{V}^{(k)}_{\io}-\Im(\io_{k+1;k})}
\qquad\hbox{s.t.}\qquad
\Psi_k(v)\!=\!\Psi_k(v')\in X
~~\hbox{and}~~ D\si(v)\neq v'~\forall\,\si\!\in\!\bS_k\,;$$
see the left diagram in Figure~\ref{Shrink_fig}.
This makes it difficult to describe a vector bundle on~$W_k$ 
as a pushdown of a vector bundle on $\Dom(\Psi_{k})\big|_{\wt{V}^{(k)}_{\io}-\Im(\io_{k+1;k})}$. 
For this reason, we shrink the last space to
an $\bS_k$-invariant open subspace~$\cN^{\circ}_{k;0}\io$ so~that
\BE{Psikrestr_e}\Psi_k\!: \cN^{\circ}_{k;0}\io
\lra U_k^{\circ}\!\equiv\!\Psi_k\big(\cN^{\circ}_{k;0}\io\big)\EE
is an $\bS_k$-covering map and
the collection $\{U_k^{\circ}\}_{k\in \Z^{\ge0}}$ 
is an open covering of~$X$.
Figure~\ref{Shrink_fig} illustrates this shrinking procedure.\\

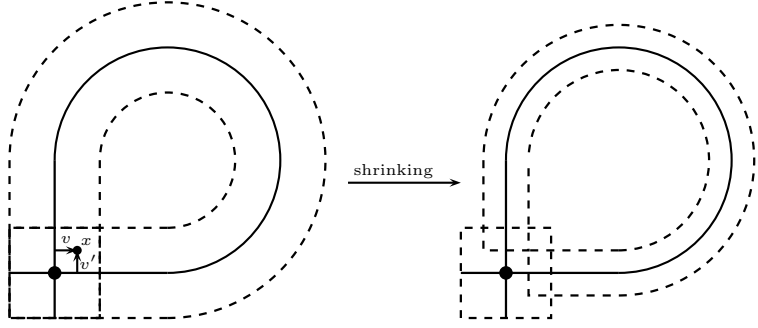
\begin{figure}
\begin{pspicture}(-5,-1)(11,2.4)
\psset{unit=.3cm}
\psline[linewidth=.1](13,-2)(20,-2)
\psline[linewidth=.1](15,-4)(15,3)
\pscircle*(15,-2){.3}
\psarc(20,3){5}{-90}{180}
\psline[linewidth=.1,linestyle=dashed,dash=3pt](17,0)(13,0)(13,-4)(17,-4)(17,0)
\psline[linewidth=.1,linestyle=dashed,dash=3pt](20,-1)(16,-1)(16,-3)(20,-3)
\psline[linewidth=.1,linestyle=dashed,dash=3pt](14,3)(14,-1)(16,-1)(16,3)
\psarc[linewidth=.1,linestyle=dashed,dash=3pt](20,3){6}{-90}{180}
\psarc[linewidth=.1,linestyle=dashed,dash=3pt](20,3){4}{-90}{180}
\psline[linewidth=.1](-7,-2)(0,-2)
\psline[linewidth=.1](-5,-4)(-5,3)
\pscircle*(-5,-2){.3}
\psline{->}(-5,-1)(-4,-1)\rput(-4.5,-.5){\tiny{$v$}}
\psline{->}(-4,-2)(-4,-1)\rput(-3.5,-1.5){\tiny{$v'$}}
\pscircle*(-4,-1){.2}\rput(-3.6,-.6){\tiny{$x$}}
\psarc(0,3){5}{-90}{180}
\psline[linewidth=.1,linestyle=dashed,dash=3pt](-3,0)(-7,0)(-7,-4)(-3,-4)(-3,0)
\psline[linewidth=.1,linestyle=dashed,dash=3pt](0,0)(-7,0)(-7,-4)(0,-4)
\psline[linewidth=.1,linestyle=dashed,dash=3pt](-7,3)(-7,-4)(-3,-4)(-3,3)
\psarc[linewidth=.1,linestyle=dashed,dash=3pt](0,3){7}{-90}{180}
\psarc[linewidth=.1,linestyle=dashed,dash=3pt](0,3){3}{-90}{180}
\psline{->}(8,2)(13,2)
\rput(10,2.5){\tiny{\tn{shrinking}}}
\end{pspicture}
\caption{Shrinking regularizations.}\label{Shrink_fig}
\end{figure}

\noindent
For $k',k\!\in\!\Z^{\ge0}$ with $k'\!\leq\!k$, let 
\hbox{$\pi_{k;k'}\!:\cN_{k;k'}\io\!\lra\!\wt{V}_{\io}^{(k)}$} be as in~\eref{kk'bundles}.
For a continuous function \hbox{$\ve\!:\wt{V}_{\io}^{(k)}\!\lra\!\R^{+}$}, define
$$\cN_{k;k'}\io(\ve)=
\big\{\!(v_i)_{i\in[k]-[k']}\!\in\!\cN_{k;k'}\io\!:
\rho_{k;i}(v_i)\!<\!\ve^2(\pi_{k;k'}(v)\!)~\forall i\!\in\![k]\!-\![k']\big\}.$$
By Remark~\ref{NCSCDregul_rmk} and the proof of \cite[Lemma 5.8]{SympDivConf}, 
there exists a smooth function $\ve\!:X\!\lra\!\R^+$ such~that  
\BE{veGlobal_e}
\cN_{k;0}\io(2^k\ve\!\circ\!\io_k)\subset\Dom(\Psi_k), \quad
\ve\big(\Psi_k(v)\!\big)= \ve\big(\io_k(\pi_{k;0}(v)\!)\!\big)
 ~~\forall\,v\!\in\!\cN_{k;0}\iota(2^k\ve\!\circ\!\io_k),\EE
and $\Psi_k(\ov{2^{k-1}\ve\!\circ\!\io_k})\!\subset\!X$ is closed for every $k\!\in\!\Z^{\ge0}$.
If $V$ is compact, $\ve$ can taken to be a constant.
Define
\BE{NCcomplementVN_e}\begin{split}
&\hspace{.15in}\wt{V}^{(k),\circ}_{\io}
= \wt{V}^{(k)}_{\io}\!-\!\bigcup_{\ell>k}\!\Psi_{\ell;k}
\big(\ov{\cN_{\ell;k}\io(2^{\ell-1}\ve\!\circ\!\io_{\ell})}\big),\\
& \cN^{\circ}_{k;0}\io=\cN_{k;0}\io\big(2^k\ve\!\circ\!\io_k\big)|_{\wt{V}^{(k),\circ}_{\io}},
~~ U_k^\circ=\Psi_k(\cN^{\circ}_{k;0}\io).
\end{split}\EE
By~\eref{NCPsikk_e2} and~\eref{NCPsikk_e3}, the restriction~\eref{Psikrestr_e}
is an $\bS_k$-covering map.\\
 
\noindent
For $k,k'\!\in\!\Z^{\geq 0}$ with $k'\!\le\!k$, let
$$\cN^{k'}_{k;0}\io= \cN^{\circ}_{k;0}\io\!\cap\!\fD\Psi_{k;k'}^{-1}(\cN^{\circ}_{k';0}\io)
\quad\hbox{and}\quad 
\cN^k_{k';0}\io=\fD\Psi_{k,k'}(\cN^{k'}_{k;0}\io)=
\fD\Psi_{k;k'}(\cN^{\circ}_{k;0}\io)\!\cap\!\cN^{\circ}_{k';0}\io.$$
Since the map $\fD\Psi_{k;k'}$ is $\bS_{k;k'}$-equivariant and $\bS_{k;k'}^c$-invariant,
the subspace  $\cN^{k'}_{k;0}\io\!\subset\!\cN^k_{k;0}\io$ 
is \hbox{$\bS_{k;k'}\!\times\!\bS_{k;k'}^c$}-invariant.
The restriction 
\BE{NDcovmap_e}\fD\Psi_{k;k'}\!:\cN^{k'}_{k;0}\io \lra \cN^k_{k';0}\io\EE
is an $\bS_{k;k'}$-equivariant $\bS_{k;k'}^c$-covering map.
By the last equality in~\eref{NCDcons_e2},
\BE{cnk0k'_e}\begin{split}
&\fD\Psi_{k;k'}\big(\cN^{k'}_{k;0}\io\!\cap\!\cN^{k''}_{k;0}\io\big)
=\cN^k_{k';0}\io\!\cap\!\cN^{k''}_{k';0}\io, \quad
\fD\Psi_{k;k''}\big(\cN^{k'}_{k;0}\io\!\cap\!\cN^{k''}_{k;0}\io\big)
=\cN^k_{k'';0}\io\!\cap\!\cN^{k'}_{k'';0}\io,\\
&\hspace{1.2in}
\fD\Psi_{k';k''}\big(\cN^k_{k';0}\io\!\cap\!\cN^{k''}_{k';0}\io\big)
=\cN^k_{k'';0}\io\!\cap\!\cN^{k'}_{k'';0}\io,
\end{split}\EE
whenever $0\!\le\!k''\!\le\!k'\!\le\!k$.\\

\noindent
Suppose 
$$\big\{\wt\pi_k\!:\wt{E}_k\!\lra\!\cN_{k;0}^{\circ}\io\big\}_{k\in\Z^{\ge0}}
\qquad\hbox{and}\qquad
\big\{\wt{F}_{k;k'}\!:\wt{E}_k|_{\cN^{k'}_{k;0}\io}\!\lra\!\wt{E}_{k'}|_{\cN^k_{k';0}\io}
\big\}_{k,k'\in\Z^{\ge0},k'\le k}$$
is a collection of $\bS_k$-equivariant (complex) vector bundles and 
a collection of $\bS_{k;k'}$-equivariant \hbox{$\bS_{k;k'}^c$-invariant}
(smooth) vector bundle maps that lift the covering maps~\eref{NDcovmap_e}, 
restrict to an isomorphism on each fiber, and satisfy
\BE{Cocycle_e}
\wt{F}_{k;k''}\big|_{\cN_{k;0}^{k'}\io \cap \cN_{k;0}^{k''}\io}
=\wt{F}_{k';k''}\!\circ\!\wt{F}_{k;k'}\big|_{\cN_{k;0}^{k'}\io \cap \cN_{k;0}^{k''}\io}
\qquad\forall\,k''\!\le\!k'\!\le\!k.\EE
By~\eref{cnk0k'_e}, the composition on the right-hand side above is well-defined.
Since the map~\eref{NDcovmap_e} is an \hbox{$\bS_k$-covering} map, 
$$E_k\!\equiv\!\wt{E}_k/\bS_k \lra \cN_{k;0}^{\circ}\io/\bS_k\!=\!U_k^{\circ}$$ 
is a vector bundle.
The maps~$\wt{F}_{k;k'}$ induce vector bundle isomorphisms
$$F_{k;k'}\!:E_k|_{U_k^{\circ}\cap U_{k'}^{\circ}}\lra  
E_{k'}|_{U_k^{\circ}\cap U_{k'}^{\circ}}$$
covering the identity on~$U_k^{\circ}\cap U_{k'}^{\circ}$.
By~\eref{Cocycle_e},
$$F_{k;k''}\big|_{U_k^{\circ}\cap U_{k'}^{\circ}\cap U_{k''}^{\circ}}
=F_{k';k''}\!\circ\!F_{k;k'}
\big|_{U_k^{\circ}\cap U_{k'}^{\circ}\cap U_{k''}^{\circ}}.$$
We can thus form a vector bundle
\BE{NCvbdfn_e}\begin{split}
&\pi\!:E\!\equiv\!\bigg(\bigsqcup_{k\in\Z^{\ge0}}\!\!\!\!E_k\bigg)\!\!\Big/\!\!\sim\lra X,
\quad \pi\big([v]\big)=\Psi_k\big(\wt\pi_k(v)\!\big)~~\forall~k\!\in\!\Z^{\ge0},~v\!\in\!E_k,\\
&\qquad E_k\big|_{U_k^{\circ}\cap U_{k'}^{\circ}} \ni w\sim F_{k;k'}(w)
\in E_{k'}\big|_{U_{k}^{\circ}\cap U_{k'}^{\circ}}
\quad\forall~k,k'\!\in\!\Z^{\ge0},~k'\!\le\!k\,.
\end{split}\EE

\vspace{.15in}

\noindent
A collection $\{\wt{s}_k\}_{k\in\Z^{\ge0}}$ of $\bS_k$-equivariant sections of
the vector bundles~$\wt{E}_k$ such~that
\BE{NCscond_e} \wt{F}_{k;k'}\!\circ\!\wt{s}_k\big|_{\cN_{k;0}^{k'}\io}
=\wt{s}_{k'}\!\circ\!\Psi_{k;k'}\big|_{\cN_{k;0}^{k'}\io}
\qquad\forall~k,k'\!\in\!\Z^{\ge0},~k'\!\le\!k,\EE
determines a section~$s$ on the corresponding vector bundle~$E$ in~\eref{NCvbdfn_e} so~that
$$s\big([v]\big)=\big[\wt{s}_k(v)\!\big]
\qquad\forall~k\!\in\!\Z^{\ge0},~v\!\in\!\cN_{k;0}^{\circ}\io\,.$$

\vspace{.15in}

\noindent
If $E'\!\lra\!X$ is a (complex) vector bundle, the $\bS_k$-action on~$\cN_{k;0}^{\circ}\io$ 
lifts to an action on the vector bundle
$$\wt\pi_k'\!:\wt{E}_k'\!\equiv\!\Psi_k^*E'\lra  \cN_{k;0}^{\circ}\io$$
so that $\wt{E}_k'/\bS_k\!=\!E'|_{U_k^{\circ}}$.
The covering maps~\eref{NDcovmap_e} lift to
$\bS_{k;k'}$-equivariant \hbox{$\bS_{k;k'}^c$-invariant} vector bundle maps 
$$\wt{F}_{k;k'}'\!:\wt{E}_k'|_{\cN^{k'}_{k;0}\io}\lra\wt{E}_{k'}'|_{\cN^k_{k';0}\io}$$
that restrict to an isomorphism on each fiber and satisfy~\eref{Cocycle_e}
with~$\wt{F}$ replaced by~$\wt{F}'$ so that the corresponding vector bundle~\eref{NCvbdfn_e}
is canonically identified with~$E'$.
A collection 
$$\big\{\wt\Phi_k\!:\wt{E}_k\!\lra\!\wt{E}_k'\big\}_{k\in\Z^{\ge0}}$$ 
of $\bS_k$-equivariant vector bundle homomorphisms covering the identity on~$\cN_{k;0}\io$
such~that
\BE{NChomcond_e} \wt{F}_{k;k'}\!\circ\!\wt\Phi_k\big|_{\cN_{k;0}^{k'}\io}
=\wt\Phi_{k'}\!\circ\!\wt{F}_{k;k'}\big|_{\cN_{k;0}^{k'}\io}
\qquad\forall~k,k'\!\in\!\Z^{\ge0},~k'\!\le\!k,\EE
determines a vector bundle homomorphism $\Phi\!:E\!\lra\!E'$ covering 
the identity on~$X$ so~that
$$\Phi\big([w]\!\big)=\big[\wt\Phi_k(w)\!\big]
\qquad\forall~k\!\in\!\Z^{\ge0},~w\!\in\!\wt{E}_k\,.$$

\vspace{.15in}

\noindent
For every $k\!\in\!\Z^{\ge0}$, let
\BE{cLk-tag}\begin{split}
\wt\pi_k\!:\cO_{\cR;k}(\io)&= \bigg(\!\!\pi_{k;0}^*\bigotimes_{i\in [k]}\cN\wt\io_{k;k-1}^{(i)} 
\!\bigg)\!\bigg|_{\cN^{\circ}_{k;0}\io}\lra \cN_{k;0}^{\circ}\io, \\
\wt\pi_k\!:T_{\cR;k}(-\log\io)&=
\big(\pi_{k;0}^{\,*}T\wt{V}^{(k),{\circ}}_{\io}\big)\big|_{\cN_{k;0}^{\circ}\io}
\!\oplus\! \big(\cN^{\circ}_{k;0}\io\!\times\!\C^{[k]}\big)\lra\cN_{k;0}^{\circ}\io\,.
\end{split}\EE
The complex structures $\fI_{k;i}$ on $\cN\wt\io_{k;k-1}^{(i)}$ encoded in~$\cR$ determine 
a complex structure on the complex line bundle~$\cO_{\cR;k}(\io)$.
The almost complex structure~$J_{\io;k}$ on~$\wt{V}_k$, induced by~$J$, 
and the standard complex structure on~$\C^{[k]}$
determine a complex structure on the vector bundle~$T_{\cR;k}(-\log\io)$.
The $\bS_k$-action on~$\cN_{k;0}^{\circ}\io$ naturally lifts to both bundles.\\

\noindent
Let $k,k'\!\in\!\Z^{\ge0}$ with $k'\!\le\!k$ and
$$\Pi_k\!: \cN\wt\io_{k;0}\lra\bigotimes_{i\in [k]}\cN\wt\io_{k;k-1}^{(i)} , 
\qquad \Pi_k\big(\!(v_i)_{i\in[k]}\big)=\bigotimes_{i\in[k]}\!v_i\,.$$
We denote by~$\na^{(k)}$ and~$\na^{(k;k')}$
the connections on $\cN\io_k\!\approx\!\cN_{k;0}\io$ and
\hbox{$\cN\wt\io_{k;k'}\!\approx\!\cN_{k;k'}\io$}
induced by the connections~$\na^{(k;i)}$ on the direct summands of these vector bundles.
We write an element \hbox{$v\!\equiv\!(v_i)_{i\in[k]}$} of~$\cN_{k;0}\io$ as
$$v=(v_{k;k'},v_{k;k'}^c) \quad\hbox{with}\quad
v_{k;k'}\!\equiv\!(v_i)_{i\in[k]-[k']}\!\in\!\cN_{k;k'}\io
~~\hbox{and}~~ v_{k;k'}^c\!\equiv\!(v_i)_{i\in[k']}\in\cN_{k;k'}^c\io$$
below.
Let
$$\nh_{\na^{(k)};v}\!:T_{\pi_{k;0}(v)}\wt{V}_{\io}^{(k)}\lra T_v(\cN_{k;0}\io)
\quad\hbox{and}\quad
\nh_{\na^{(k;k')};v_{k;k'}}\!:T_{\pi_{k;k'}(v_{k;k'})}\wt{V}_{\io}^{(k)}
\lra T_{v_{k;k'}}(\cN_{k;k'}\io)$$
be the injective homomorphisms as in~\eref{nhndfn_e} corresponding
to the connections~$\na^{(k)}$ and~$\na^{(k;k')}$.\\

\noindent
By the first equation in~\eref{NCcomplementVN_e}, 
$$ v_i\neq0 \qquad \forall~i\!\in\![k]\!-\![k'],~
(v_j)_{j\in[k]}\!\in\!\cN^{k'}_{k;0}\io
\subset\bigoplus_{j\in[k]}\cN\wt\io_{k;k-1}^{(j)}.$$
Since $\fD\Psi_{k;k'}$ is a product Hermitian isomorphism, the~map
\begin{gather}\label{NCthetaII_e}
\wt\th_{k;k'}\!: \cO_{\cR;k}(\io)\big|_{\cN^{k'}_{k;0}\io}
\lra \cO_{\cR;k'}(\io)\big|_{\cN^k_{k';0}\io},\\
\notag
\wt\th_{k;k'}\big(v,\Pi_k(v_{k;k'},w_{k;k'}^c)\!\big)
=\Big(\fD\Psi_{k;k'}(v),
\Pi_{k'}\big(\fD\Psi_{k;k'}(v_{k;k'},w_{k;k'}^c)\!\big)\!\Big),
\end{gather}
is a well-defined homomorphism of complex line bundles that
lifts the covering map~\eref{NDcovmap_e} and 
restricts to an isomorphism on each fiber.
The~map
\begin{gather}\label{NCvarTII_e}
\wt\vt_{k;k'}\!: T_{\cR;k}(-\log\io)\big|_{\cN^{k'}_{k;0}\io}
\lra T_{\cR;k}(-\log\io)\big|_{\cN^k_{k';0}\io},\\
\notag
\begin{split}
\wt\vt_{k;k'}\big(\!(v,w)\!\oplus\!\big(v,(c_i)_{i\in[k]}\big)\!\!\big)
=\Big(\fD\Psi_{k;k'}(v),\nd_{v_{k;k'}}\Psi_{k;k'}\!
\big(\nh_{\na^{(k;k')};v_{k;k'}}(w)\!+\!\!\sum_{i\in[k]-[k']}\!\!\!\!\!\!c_iv_i\!\big)\!\!\Big)&\\
\!\oplus\!\big(\Psi_{k;k'}(v_{k;k'}),(c_i)_{i\in[k']}\big)&,
\end{split}
\end{gather}
is similarly a well-defined homomorphism of vector bundles that
lifts the covering map~\eref{NDcovmap_e} and 
restricts to an isomorphism on each fiber.
Since $J$ is an $\cR$-compatible almost complex structure on~$X$, this homomorphism is $\C$-linear.\\

\noindent
By the commutativity of the diagrams in Figure~\ref{Shrink_fig},
the bundle homomorphisms~\eref{NCthetaII_e} and~\eref{NCvarTII_e}
are $\bS_{k;k'}$-equivariant \hbox{$\bS_{k;k'}^c$-invariant}. 
By~\eref{NCDcons_e2}, the collections $\{\wt\th_{k;k'}\}_{k'\le k}$
and $\{\wt\vt_{k;k'}\}_{k'\le k}$ satisfy~\eref{Cocycle_e}
with~$\wt{F}$ replaced by~$\wt\th$ and~$\wt\vt$.
The first collection thus determines a complex line bundle
\begin{equation*}\begin{split}
&\pi\!: \cO_{\cR;X}(\io)\!=\!\bigg(\bigsqcup_{k\in\Z^{\ge0}}\!\!\!\!
\cO_{\cR;k}(\io)/\bS_k\!\!\bigg)\!\Big/\!\!\!\sim\,\lra X,~~
\pi\big([w]\big)=\Psi_k\big(\wt\pi_k(w)\!\big)
~\forall\,k\!\in\!\Z^{\ge0},\,w\!\in\!\cO_{\cR;k}(\io),\\
&\hspace{1in}[w]\sim\big[\wt\th_{k;k'}(w)\!\big]
\quad\forall~k,k'\!\in\!\Z^{\ge0},~k'\!\le\!k,~
w\!\in\!\cO_{\cR;k}(\io)\big|_{\cN^{k'}_{k;0}\io}.
\end{split}\end{equation*}
The second collection similarly determines a complex vector bundle
\begin{equation*}\begin{split}
&\pi\!: T_\cR X(-\log\io)\!=\!\bigg(\bigsqcup_{k\in\Z^{\ge0}}\!\!\!\!
T_{\cR;k}(-\log\io)/\bS_k\!\!\bigg)\!\Big/\!\!\sim\,\,\lra X,\\
&\hspace{1.5in}\pi\big([w]\big)=\Psi_k\big(\wt\pi_k(w)\!\big)
~\forall\,k\!\in\!\Z^{\ge0},\,w\!\in\!T_\cR X(-\log\io),\\
&\hspace{-.05in}[w]\sim\big[\wt\vt_{k;k'}(w)\!\big]
\quad\forall~k,k'\!\in\!\Z^{\ge0},~k'\!\le\!k,~
w\!\in\!T_\cR X(-\log\io)\big|_{\cN^{k'}_{k;0}\io}.
\end{split}\end{equation*}

\vspace{.15in}

\noindent
The smooth sections
$$\wt{s}_k\!:\cN_{k;0}^{\circ}\io\lra\cO_{\cR;k}(\io),\qquad
\wt{s}_k(v)=\big(v,\Pi_k(v)\!\big),$$
are $\bS_k$-equivariant and satisfy~\eref{NCscond_e} with~$\wt{F}$ replaced by~$\wt\th$.  
They thus determine a section~$s_{\cR}$ of the complex line bundle~$\cO_{\cR;X}(\io)$.
The smooth bundle homomorphisms
$$\wt\Phi_k\!:T_{\cR;k}(-\log\io)\lra\Psi_k^*TX, \quad
\wt\Phi_k\Big(\!(v,w)\!\oplus\!\big(v,(c_i)_{i\in[k]}\big)\!\!\Big)
=\nd_v\Psi_k\!\Big(\nh_{\na^{(k)};v}(w)\!+\!\!\sum_{i\in[k]}c_iv_i\!\Big),$$
are $\bS_k$-equivariant.
By~\eref{NDcovmap_e}, these bundle homomorphisms satisfy~\eref{NChomcond_e}
with~$\wt{F}$ replaced by~$\wt\vt$.
They thus determine a vector bundle homomorphism
$$\io_{\cR}\!: T_\cR X(-\log\io)\lra TX\,.$$

\subsection{Proof of Theorem~\ref{TXV_thm}\ref{cTXV_it}}
\label{ChNC_sec}

\noindent
Let $(X,\om)$,
$V\!\subset\!X$, $\io\!:\wt{V}\!\lra\!X$, $\cR$, and $J$ be as in Section~\ref{ConGNC_subs}.
We denote the curvature 2-form of a connection~$\na$ on 
a complex vector bundle $E\!\lra\!Y$ over a smooth manifold~by
$$F^{\na}\in\Ga\big(Y;\La^2(T^*Y)\!\otimes_{\R}\!\End_{\C}(E)\!\big).$$
For $i\!\in\!\Z^{\ge0}$, we define 
\begin{gather*}
c_i(\na)\in \Ga\big(Y;\La^{2i}(T^*Y)\!\otimes_{\R}\!\End_{\C}(E)\!\big)
\quad\hbox{and}\quad
c(\na)\in\bigoplus_{i=0}^{\infty}\Ga\big(Y;\La^{2i}(T^*Y)\!\otimes_{\R}\!\End_{\C}(E)\!\big)\\
\hbox{by}\qquad
1\!+\!c_1(\na)\!+\!c_2(\na)\!+\!\ldots
\equiv c(\na)\equiv
\det{}\!_{\C}\!\Big(\!I\!+\!\frac{\fI}{2\pi}F^{\na}\!\Big).
\end{gather*}
By \cite[p206]{Morita},
\BE{FtoC_e}
\big[c^{\na}\!(E)\big]=c(E)\in H^{2*}_{\deR}(Y)
\!\equiv\!\bigoplus_{i=0}^{\infty}H^{2i}_{\deR}(Y).\EE
We compare $c(TX)$ and $c(TX(-\log V)\!)\!=\!c(TX(-\log\io)\!)$ 
at the level of differential forms, which can be done locally;
see Proposition~\ref{DFChern_prp}.
This essentially reduces the computation to the SC~case.
The de Rham cohomology version of~\eref{cTXV_e},
which is equivalent to~\eref{cTXV_e} itself,
follows from~\eref{FtoC_e},
Lemma~\ref{PDVk_lmm}, and Proposition~\ref{DFChern_prp}.\\

\noindent
A connection $\na$ on a complex line bundle $\pi\!:L\!\lra\!Y$ determines 
a horizontal tangent subbundle $TL^{\hor}\!\subset\!TL$ and 
an $\R$-valued \sf{angular 1-form} $\al_{\na}$ on $L\!-\!Y$.
The latter is characterized~by
$$\ker\al_{\na}=\big(TL^{\hor}\!\oplus\!\R\ze_L\big)\big|_{L-Y}
\qquad\hbox{and}\qquad
\al_{\na}\bigg(\frac{\nd}{\nd\th}\ne^{\fI\th}v\Big|_{\th=0}\bigg)=1.$$
By the proof of \cite[Lemma~1.1]{anal},
\BE{dalnadfn_e} \nd\al_{\na}=\pi^*\eta_{\na}\big|_{L-Y}\EE
for some 2-form~$\eta_{\na}$ on~$Y$.

\begin{rmk}\label{angform_rmk}
If $\na$ is compatible with a Hermitian metric~$\rho$ on~$L$, 
then the connection 1-form~$\th_1^1$ in the proof of Lemma~\ref{RadVecFld_lmm}
is purely imaginary.
In the local chart of this proof,
the angular 1-form $\al_{\na}$ of~$\na$ is then given~by
$$\al_{\na}\big|_{(y,z)}=-\fI\frac{\nd z}{z}-\fI\pi^*\th_1^1.$$
Thus, $\eta_{\na}\!=\!-\fI F^{\na}$ above.
\end{rmk}

\noindent
Let $\be\!:\R\!\lra\!\R^{\ge0}$ be as in~\eref{becond_e}.
For a Hermitian metric~$\rho$ and a smooth function $\ve\!:Y\!\lra\!\R^+$, define
\BE{berhovedfn_e}\be_{\rho;\ve}\!:L\lra\R^{\ge0}, \qquad
\be_{\rho;\ve}(v)=\be\big(2\ve(\pi(v)\!)^{-2}\rho(v)\!\big).\EE
By~\eref{dalnadfn_e}, the 2-form
$$\tau_{\rho;\na;\ve}\equiv-\frac{1}{2\pi}\nd\big(\be_{\rho;\ve}\al_{\na}\big)
=-\frac{1}{2\pi}\Big(\!\nd\be_{\rho;\ve}\!\w\!\al_{\na}
\!+\!\be_{\rho;\ve}\pi^*\eta_{\na}\!\Big)$$
is well-defined on the entire total space of~$L$.
This closed 2-form is compactly supported in the vertical direction~and
\BE{tauinteg_e}\int_{\pi^{-1}(y)}\tau_{\rho;\na;\ve}
=-\frac{1}{2\pi}\bigg(\int_0^{\infty}\!\!\nd\be_{\rho;\ve}\!\!\bigg)
\bigg(\int_{S^1}\!\!\nd\th\!\!\bigg)=1 \qquad  \forall~y\!\in\!Y.\EE
Thus, $\tau_{\rho;\na;\ve}$ represents the Thom class of the complex line bundle $\pi\!:L\!\lra\!Y$;
see \cite[p64]{BT}.\\

\noindent
For each $k\!\in\!\Z^+$, let $\ve_k\!:\wt{V}^{(k)}\!\lra\!\R^+$ be the composition 
of the function~$\ve$ in~\eref{NCcomplementVN_e} with~$\io_k$.
For each $i\!\in\![k]$, let
$$\pi_{k;i}\!:\cN\io_k\!\approx\!\cN_{k;0}\io\lra \cN\wt\io_{k;k-1}^{(i)}$$
be the component projection.
For $k'\!\in\!\Z^{\ge0}$, we denote~by 
$\wt\tau_{k;k'}\!\in\!\Om^{2k'}(\cN_{k;0})$ the $k'$-th elementary symmetric polynomial
on the set $\{\pi_{k;i}^{\,*}\tau_{\rho_{k;i};\na^{(k;i)};2\ve_k}\}_{i\in[k]}$, i.e.
$$\wt\tau_{k;k'}=
\sum_{\begin{subarray}{c}i_1,\ldots,i_{k'}\in[k]\\
i_1<\ldots<i_{k'}\end{subarray}}\hspace{-.2in}
\big(\pi_{k;i_1}^{\,*}\tau_{\rho_{k;i_1};\na^{(k;i_1)};2\ve_k}\big)
\!\w\!\ldots\!\w\!
\big(\pi_{k;i_{k'}}^{\,*}\tau_{\rho_{k;i_{k'}};\na^{(k;{i_{k'}})};2\ve_k}\big).$$
This $2k'$-form on the total space of~$\cN\io_k$ is $\bS_k$-invariant and closed.
The $2k$-form
$$\wt\tau_{k;k}\equiv \underset{i\in[k]}{\La}\pi_{k;i}^{\,*}\tau_{\rho_{k;i};\na^{(k;i)};2\ve_k}
\equiv\big(\pi_{k;1}^{\,*}\tau_{\rho_{k;1};\na^{(k;1)};2\ve_k}\big)\!\w\!\ldots\!\w
\big(\pi_{k;k}^{\,*}\tau_{\rho_{k;k};\na^{(k;k)};2\ve_k}\big)$$
is in addition compactly supported in the vertical direction.
By~\eref{tauinteg_e},
it represents the Thom class of the complex vector bundle 
$\pi_{k;0}\!:\cN\io_k\!\lra\!\wt{V}_{\io}^{(k)}$.\\
 
\noindent
For an open subset $U\!\subset\!X$ such that
\BE{PsikUdecomp_e}\begin{split}
&\hspace{1in}\Psi_k^{-1}(U)=\wt{U}_0\!\sqcup\!\wt{U}_1\!\sqcup\!\ldots\!\sqcup\!\wt{U}_\ell
\qquad\hbox{with}\\
&\wt{U}_0\subset\cN\io_k\!-\!\ov{\cN\io_k(\ve_k)} \quad\hbox{and}\quad
\Psi_k\!:\wt{U}_i\lra U~~\hbox{a diffeomorphism}~~\forall~i\!\in\![\ell],
\end{split}\EE
we define 
\BE{pushtaudfn_e}
\Psi_{k*}\wt\tau_{k;k}\big|_U= 
\sum_{i=1}^\ell\! \big\{\!\!\{\Psi_k|_{\wt{U}_i}\}^{-1}\!\big\}^{\!*}\wt\tau_{k;k} 
\in \Om^{2k}(U;\R).\EE
Since $\wt\tau_{k;k}$ vanishes on~$\wt{U}_0$, $\Psi_{k*}\wt\tau_{k;k}|_U$ does not depend 
on an admissible decomposition of~$\Psi_k^{-1}(U)$ as above.
Thus, $\Psi_{k*}\wt\tau_{k;k}|_U$ is well-defined and
$$\big(\Psi_{k*}\wt\tau_{k;k}|_U\big)\big|_{U\cap U'}
=\Psi_{k*}\wt\tau_{k;k}\big|_{U\cap U'}
=\big(\Psi_{k*}\wt\tau_{k;k}|_{U'}\big)\big|_{U\cap U'}$$
for all open subsets $U,U'\!\subset\!X$ with admissible decompositions.
By Lemma~\ref{Psikcov_lmm} at the end of this section, such open subsets cover~$X$.
We thus obtain a closed $2k$-form $\Psi_{k*}\wt\tau_{k;k}$ on~$X$.
If  $\mu$ is a closed form on~$X$, then
\BE{pushtau_e}  \int_X (\Psi_{k*}\wt\tau_{k;k})\!\w\!\mu
=\int_{\cN\io_k}\!\wt\tau_{k;k}\!\w\!(\Psi_k^*\mu)
=\int_{\wt{V}^{(k)}}\io_k^*\mu\,;\EE
the first equality above holds for any differential form~$\mu$ on~$X$.
The next straightforward lemma is also proved at the end of this section.

\begin{lmm}\label{PDVk_lmm}
With the notation as above,
\BE{PullBackTau_e}
\tau_k\!\equiv\!\frac1{k!}\Psi_{k*}\wt\tau_{k;k}=\PD_X\big([V^{(k)}_{\io}]_X\big),
\quad\Psi_k^*\tau_{k'}\big|_{\cN_{k;0}^{\circ}}=\wt\tau_{k;k'}\big|_{\cN_{k;0}^{\circ}}
\qquad\forall~k,k'\!\in\!\Z^+.\EE
\end{lmm}

\begin{prp}\label{DFChern_prp}
There exist connections~$\na$ and~$\na'$ in the complex vector bundles 
$(TX,J)$ and $(T_{\cR}X(-\log\io),\fI_{\cR,J})$ so~that 
\BE{cNabla_e}
c(\na)= c(\na')\big(1\!+\!\tau_1\!+\!\tau_2\!+\!\cdots\big).\EE
\end{prp}

\begin{proof} We construct $\na$ and $\na'$ using the global perspective of 
Section~\ref{ConGNC_subs}.
For $k\!\in\!\Z^{\ge0}$ and $i\!\in\![k]$, let 
$$\be_{k;i}\!\equiv\!\be_{\rho_{k;i};\ve_k}\!:\cN\wt\io_{k;k-1}^{(i)}\lra\R^{\ge0}$$
be a smooth function as in~\eref{berhovedfn_e} and 
$$\Phi_{k;i}\!: \big(\cN\wt\io_{k;k-1}^{(i)}\!-\!\wt{V}_{\io}^{(k)}\big)\!\times\!\C
\lra \pi^*\cN\wt\io_{k;k-1}^{(i)}, \qquad \Phi_{k;i}(v,c)=cv,$$
where $\pi\!:\cN\wt\io_{k;k-1}^{(i)}\!\lra\!\wt{V}_{\io}^{(k)}$ is the bundle projection.
We define a connection~$\na'^{(k;i)}$ in the trivial complex line bundle 
$\cN\wt\io_{k;k-1}^{(i)}\!\times\!\C$ over~$\cN\wt\io_{k;k-1}^{(i)}$ by
$$\na'^{(k;i)}=\be_{k;i}\nd\!+\!(1\!-\!\be_{k;i})\Phi_{k;i}^{\,*}\pi^*\na^{(k;i)}\,;$$
the last summand above is well-defined because $\be_{k;i}\!=\!1$ 
in a neighborhood of the zero section $\wt{V}_{\io}^{(k)}\!\subset\!\cN\wt\io_{k;k-1}^{(i)}$.
We note~that
\BE{CurvatureRelation_e}\begin{split}
F^{\na^{'(k;i)}}
&=\be_{k;i}F^{\nd}\!+\!(1\!-\!\be_{k;i})F^{\Phi_{k;i}^{\,*}\pi^*\na^{(k;i)}}
\!+\!\nd(1\!-\!\be_{k;i})\!\w\!\big(\Phi_{k;i}^{\,*}\pi^*\na^{(k;i)}\!-\!\nd\big)\\
&=0\!+\!(1\!-\!\be_{k;i})\pi^*F^{\na^{(k;i)}}\!-\!\fI\,\nd\be_{k;i}\!\w\!\al_{\na^{(k;i)}}
=\pi^*F^{\na^{(k;i)}}\!+\!2\pi\fI\tau_{\rho_{k;i};\na^{(k;i)};\ve_k};
\end{split}\EE
the last two equalities follow from Remarks~\ref{RadVecFld_rmk} and~\ref{angform_rmk}.\\

\noindent
For each $k\!\in\!\Z^{\ge0}$, the connections~$\na^{(k;i)}$ on
the complex line bundles $\cN\wt\io_{k;k-1}^{(i)}$ determine a splitting
\BE{TcNdecomp_e}T\big(\cN\io_k\big)=
\pi_{k;0}^{\,*}T\wt{V}_{\io}^{(k)}\!\oplus\!
\bigoplus_{i=1}^k\pi_{k;0}^{\,*}\cN\wt\io_{k;k-1}^{(i)}\,.\EE
Let $r\!\in\!\Z^+$ be such that $V_{\io}^{(r+1)}\!=\!\eset$.
By Definition~\ref{NCSCDregul_dfn},
the vector bundle isomorphisms~\eref{fDPsikk_e} are product Hermitian isomorphisms 
for all $k'\!\le\!k$.
Furthermore, \hbox{$\ve_k\!=\!\ve_{k'}\!\circ\!\Psi_{k;k'}$}.
Starting with a $J$-compatible connection
on~$TV_{\io}^{(r)}$ and possibly shrinking the domains of the regularizations~$\Psi_k$, 
we can thus inductively construct a connection~$\na$ on~$TX$ so that $\{\nd\Psi_k\}^*\na$
agrees with the restriction of the connection 
$$\wt\na^{(k)}\equiv\pi_{k;0}^{\,*}\na^{T\wt{V}^{(k)}_{\io}}\!\oplus\! 
\bigoplus_{i=1}^k\!\big(\pi_{k;0}^{\,*}\na^{(k;i)}\!+\!
\fI(\be_{\rho_{k;i};2\ve_k}\!-\!\be_{\rho_{k;i};\ve_k})\al_{\na^{(k;i)}}\!\big)$$
on~\eref{TcNdecomp_e} to~$\Dom(\Psi_k)$ for every $k\!\in\!\Z^+$.
By Remark~\ref{angform_rmk},
the curvature of the last summand~$\wt\na^{(k;i)}$ above satisfies
\begin{equation*}\begin{split}
1\!+\!\frac{\fI}{2\pi}F^{\wt\na^{(k;i)}}&=
\Big(1\!+\!\frac{\fI}{2\pi}\pi_{k;0}^{\,*}F^{\na^{(k;i)}}
\!-\!\tau_{\rho_{k;i};\na^{(k;i)};\ve_k}\Big)\!
\big(1\!+\!\pi_{k;i}^{\,*}\tau_{\rho_{k;i};\na^{(k;i)};2\ve_k}\big)\\
&\hspace{1in}
-\frac1{4\pi}\nd\big(\!(1\!-\!\be_{\rho_{k;i};\ve_k})\al_{\na^{(k;i)}}\big)
\!\w\!\nd\big(\be_{\rho_{k;i};2\ve_k}\al_{\na^{(k;i)}}\big);
\end{split}\end{equation*}
the term on the last line above vanishes because $\be_{\rho_{k;i};\ve_k}\!\equiv\!1$
on~$\supp\,\be_{\rho_{k;i};2\ve_k}$.
Along with~\eref{CurvatureRelation_e}, this~gives
\BE{TXcurv_e}\begin{split}
\Psi_k^*c(\na)\big|_{\cN_{k;0}^{\circ}\io}
&=\pi_{k;0}^{\,*}c\big(\na^{T\wt{V}^{(k)}_{\io}}\big)\!\big|_{\cN_{k;0}^{\circ}\io}
\prod_{i=1}^k\!\!\big(\pi_{k;0}^{\,*}c(\na'^{(k;i)})
\big(1\!+\!\pi_{k;i}^{\,*}\tau_{\rho_{k;i};\na^{(k;i)};2\ve_k}
\big)\!\big)\!\big|_{\cN_{k;0}^{\circ}\io}\\
&=\pi_{k;0}^{\,*}c\big(\na^{T\wt{V}^{(k)}_{\io}}\big)\big|_{\cN_{k;0}^{\circ}\io}
\!\bigg(\prod_{i=1}^k\!\pi_{k;0}^{\,*}c(\na'^{(k;i)})\!\!\bigg)
\Psi_k^*\big(1\!+\!\tau_1\!+\!\tau_2\!+\!\ldots\big)\!\big|_{\cN_{k;0}^{\circ}\io};
\end{split}\EE
the last equality follows from the second statement of Lemma~\ref{PDVk_lmm}.\\

\noindent
By~\eref{NCSCDregul_e2}, the connection
$$\wt\na'^{(k)}\equiv\pi_{k;0}^{\,*}\na^{T\wt{V}^{(k)}_{\io}}\big|_{\cN_{k;0}^{\circ}\io}
\!\oplus\! 
\bigoplus_{i=1}^k\pi_{k;0}^{\,*}\wt\na'^{(k;i)}\big|_{\cN_{k;0}^{\circ}\io}$$
on the complex vector bundle~\eref{cLk-tag} is $\bS_k$-equivariant.
Since $\ve_k\!=\!\ve_{k'}\!\circ\!\Psi_{k;k'}$ and $\fD\Psi_{k;k'}$ is a product Hermitian
isomorphisms, 
the bundle isomorphisms~\eref{NCSCDregul_e2} intertwine these connections.
Thus, they determine a connection~$\na'$ on the complex vector bundle~$T_{\cR}X(-\log\io)$.
Since
$$\Psi_k^*c(\na')\big|_{\cN_{k;0}^{\circ}\io}
=c\big(\wt\na'^{(k)}\big)
=\pi_{k;0}^{\,*}c\big(\na^{T\wt{V}^{(k)}_{\io}}\big)\big|_{\cN_{k;0}^{\circ}\io}
\prod_{i=1}^k\!\pi_{k;0}^{\,*}c(\na'^{(k;i)}),$$
the identity~\eref{cNabla_e} follows from~\eref{TXcurv_e}.
\end{proof}

\begin{lmm}\label{Psikcov_lmm}
Let $k\!\in\!\Z^+$.
Every point $x\!\in\!X$ admits a neighborhood~$U$ as in~\eref{PsikUdecomp_e}.
\end{lmm}

\begin{proof}
Let $\Psi_k^{-1}(x)\!=\!\{\wt{w}_1,\ldots,\wt{w}_{\ell}\}$.
For each $i\!\in\![\ell]$, let $U_i\!\subset\!\wt{V}_{\io}^{(k)}$ be a neighborhood 
of~$\pi_{k;0}(\wt{w}_i)$ so that the restriction 
$$\Psi_k\!:\Dom(\Psi_k)\big|_{U_i}\lra X$$
is a diffeomorphism onto its image; see Definition~\ref{NCsmreg_dfn}.
By Remark~\ref{NCSCDregul_rmk}, 
$$W\equiv X\!-\!\Psi_k\big(\ov{\cN\io_k(\ve_k)}
\big|_{\wt{V}_{\io}^{(k)}-U_1\cup\ldots\cup U_{\ell}}\big)$$
is an open subspace of~$X$.
The open neighborhood 
$$U\equiv W\cap\bigcap_{i=1}^{\ell}\!\Psi_k\big(\Dom(\Psi_k)\big|_{U_i}\big)$$
of $x\!\in\!X$ then satisfies the condition  in~\eref{PsikUdecomp_e}.
\end{proof}

\begin{proof}[{\bf{\emph{Proof of Lemma~\ref{PDVk_lmm}}}}] 
Let $\mu$ be a closed form on~$X$.
We show~that 
\BE{PD1_e}\int_{V^{(k)}_\io}\al = \int_X \tau_k\!\w\!\mu.\EE
Since $\io_k\!:\wt{V}^{(k)}\!\lra\!V^{(k)}$ is a $k!$-covering map outside of 
a codimension~2 subspace, 
$$\int_{V^{(k)}_{\io}}\al =\frac{1}{k!}\int_{\wt{V}^{(k)}}\io_k^*\al\,.$$
Combining this with~\eref{pushtau_e}, we obtain~\eref{PD1_e} and establish the
first statement in~\eref{PullBackTau_e}.\\

\noindent
If $k'\!>\!k$, the left-hand side of the second identity in~\eref{PullBackTau_e} vanishes
because
$$\supp\,\tau_{k'}\subset\Psi_{k'}\big(\ov{\cN_{k;0}(\ve_k/2)}\big) \qquad\hbox{and}\qquad
\Psi_{k'}\big(\ov{\cN_{k;0}(\ve_k/2)}\big)\!\cap\!\Psi_k(\cN_{k;0}^{\circ})=\eset.$$
The right-hand side of the second identity in~\eref{PullBackTau_e} vanishes by definition
in this~case.\\

\noindent
Suppose $k'\!\le\!k$.
Let $S_{k;k'}\!\subset\!\bS_k$ be a collection of representatives for
the right cosets of $\bS_{k;k'}\!\times\!\bS_{k;k'}^c$ in~$\bS_k$
preserving the order of the first $k'\!$~elements.
Since \eref{Psikrestr_e} is an $\bS_k$-covering map, every point of~$U_k^{\circ}$
has a neighborhood~$U$ so~that
$$\Psi_k^{-1}(U)=\bigsqcup_{\si\in\bS_k}\!\si(W)\subset \cN_{k;0}^{\circ}\io$$
for some open subset $W\!\subset\!\cN_{k;0}^{\circ}\io$.
Since $\Psi_k\!=\!\Psi_{k'}\!\circ\!\fD\Psi_{k;k'}$ and $\fD\Psi_{k;k'}$ is
$\bS_{k;k'}^c$-invariant, 
\eref{NCPsikk_e3} implies~that
$$\Psi_{k'}^{-1}(U)=\bigcup_{\vp\in\bS_{k;k'}}
\bigsqcup_{\si\in S_{k;k'}}\!\!\!\!\!\Psi_{k;k'}\big(\vp\si(W)\!\big).$$
Since $\fD\Psi_{k;k'}$ is $\bS_{k;k'}$-equivariant and $\wt\tau_{k';k'}$ is $\bS_{k'}$-invariant,
$$\Psi_k^*\tau_{k'}=\frac{1}{k!}
\sum_{\si\in S_{k;k'}}\sum_{\vp\in\bS_{k'}}\!
\{D\si\}^{\!*}\{\fD\Psi_{k;k'}\}^{\!*}\{D\vp\}^*\wt\tau_{k';k'}
=\sum_{\si\in S_{k;k'}}\!\!\!\{D\si\}^{\!*}\{\fD\Psi_{k;k'}\}^{\!*}\wt\tau_{k';k'}\,.$$
Since $\ep_k\!=\!\ep_{k'}\!\circ\!\Psi_{k;k'}$ and $\fD\Psi_{k;k'}$ is a product Hermitian
isomorphism, it follows~that 
$$\Psi_k^*\tau_{k'}=\sum_{\si\in S_{k;k'}}\!\!\!\{D\si\}^{\!*}
\bigg(\underset{i\in[k']}{\La}\pi_{k;i}^{\,*}\tau_{\rho_{k;i};\na^{(k;i)};2\ve_k}\!\!\bigg).$$
Since $D\si$ is also a product Hermitian isomorphism,
the last expression equals~$\wt\tau_{k;k'}$.
\end{proof}

\begin{rmk}\label{PDVk_rmk}
If $X$ and $V^{(k)}_{\io}$ are not compact, 
the above proof of the first identity in~\eref{PullBackTau_e} goes through 
if $\mu$ is compactly supported.
If $X$ is not compact, but $V^{(k)}_{\io}$ is compact,
then this identity holds in the compactly de Rham cohomology
of~$X$, as well as in the usual de Rham cohomology of~$X$.
\end{rmk}

\section{On the sharpness of~\eref{cTXV_e} and~\eref{blowupchern_e}}
\label{ChernQvsZ_sec}

\noindent
We conclude by establishing the remaining statement of 
Corollary~\ref{blowupchern_crl} and showing that~\eref{blowupchern_e} does not need 
to hold in~$H^*(\wt{X};\Z)$ for arbitrary NC divisors.
The latter implies that~\eref{cTXV_e} does not hold in~$H^*(X;\Z)$
for arbitrary NC divisors either.\\

\noindent
Continuing with the notation and setup as in Lemma~\ref{blowuplog_lmm},
we denote~by
$$\wt\io\!:\bE\lra\wt{X} \qquad\hbox{and}\qquad 
\wt\pi_0\!:\ga\lra\bE$$
the inclusion map and the tautological line bundle, respectively.
For $\mu\!\in\!H_*(V^{(r)};R)$, we denote by 
\hbox{$\bE|_{\mu}\!\in\!H_*(\bE;R)$}
the fiber product of~$\mu$ with~$\pi|_{\bE}$.

\begin{lmm}\label{blowuphom_lmm} If $R$ is a commutative ring with unity,
\begin{gather}\label{BlHom_e5}
\ker\!\big(\pi_*\!: H_*(\wt{X};R)\!\lra\!H_*(X;R)\!\big)=
\wt\io_*\!\big(\!\ker\!\big(\pi_*\!:H_*(\bE;R)\!\lra\!H_*(V^{(r)};R)\!\big)\!\big)
\quad\hbox{and}\\
\label{BlHom_e5b}
\ker\!\big(\pi_*\!:H_*(\bE;R)\!\lra\!H_*(V^{(r)};R)\!\big)
=\bigoplus_{i=1}^{r-1}\!
\big\{\!c_1(\ga)^{r-1-i}\!\cap\!\big(\bE|_{\mu_i}\big)\!:
\mu_i\!\in\!H_{*-2i}(V^{(r)};R)\!\big\}.
\end{gather}
\end{lmm}

\begin{proof} We omit the coefficient ring~$R$ below.
Let $U\!\subset\!X$ be an open neighborhood of~$Y\!\equiv\!V^{(r)}$
which deformation retracts onto~$Y$ and $\wt{U}\!\subset\!\pi^{-1}(U)$.
Since $\bE$ is a deformation retract of~$\wt{U}$, the Mayer-Vietoris sequences
for $\wt{X}\!=\!(\wt{X}\!-\!\bE)\!\cup\!\wt{U}$ and $X\!=\!(X\!-\!Y)\!\cup\!U$
induce a commutative diagram 
$$\xymatrix{\ldots\ar[r]& H_*(\wt{U}\!-\!\bE)\ar[r]\ar[d]^{\id} &
H_*(\wt{X}\!-\!\bE)\!\oplus\!H_*(\bE) \ar[r]\ar[d]|{\id\oplus\pi_*}&
H_*(\wt{X}) \ar[r]\ar[d]^{\pi_*}&  H_{*-1}(\wt{U}\!-\!\bE)\ar[d]^{\id}\ar[r]&\ldots\\
\ldots\ar[r]& H_*(U\!-\!Y)\ar[r]& H_*(X\!-\!Y)\!\oplus\!H_*(Y) \ar[r]&
H_*(X) \ar[r]&  H_{*-1}(U\!-\!Y)\ar[r]&\ldots}$$
of exact sequences of $R$-modules.
This gives~\eref{BlHom_e5}.\\

\noindent
For every $y\!\in\!Y$, the collection
$\{1|_{\bE_y},c_1(\ga)|_{\bE_y},\ldots,c_1(\ga)^{r-1}|_{\bE_y}\}$
is a basis for~$H^*(\bE_y)$.
By \cite[Theorem~5.7.9]{Spanier}, the homomorphism
$$H_*(\bE)\lra \bigoplus_{i=0}^{r-1}\!H_{*-2i}(Y),
\qquad
\wt\mu\lra\big(\pi_*\big(\!c_1(\ga)^i\!\cap\!\wt\mu\big)_{i=0,\ldots,r-1}\big),$$
is thus an isomorphism.
This implies~\eref{BlHom_e5b}.
\end{proof}

\noindent
For each $k\!\in\![r\!-\!1]$, let
$$\wt\eta_k=\pi^*\big(\PD_X([V^{(k)}]_X)\!\big)\!-\!\PD_{\wt{X}}\big([\ov{V}^{(k)}]_{\wt{X}}\big)
\in H^{2k}(\wt{X};R).$$
By Lemma~\ref{blowuphom_lmm}, there exist 
\hbox{$\eta_{k;i}^c\!\in\!H^{2(k-i)}(V^{(r)};R)$}
with $i\!\in\![k]$ so~that
\BE{PTdiff_e} 
\wt\eta_k\!\cap\!\big[\wt{X}\big]
=\sum_{i=1}^k
\wt\io_*\!\Big(\!c_1(\ga)^{i-1}\!\cap\!\big(\bE|_{\eta_{k;i}^c\cap[V^{(r)}]}\big)\!\!\Big)
\in H_*(\wt{X};R).\EE
For example,
\BE{PTdiff_e1} 
\big(\pi^*\big(\PD_X([V^{(1)}]_X)\!\big)\!\big)\!\cap\!\big[\wt{X}\big]
=[\ov{V}^{(1)}]_{\wt{X}}\!+\!r[\bE]_{\wt{X}} \qquad\hbox{if}~r\!\ge\!2;\EE
the coefficient $\eta_{1;1}^c\!=\!r$ above is obtained by intersecting both sides with 
$\wt\io_*(c_1(\ga)^{r-2}\!\cap\!\bE_y)$.\\

\noindent
We set 
\begin{equation*}\begin{split}
\cPD(V)&=1\!+\!\PD_X\big([V^{(1)}]_X\big)\!+\!\PD_X\big([V^{(2)}]_X\big)\!+\!\ldots
\in\bigoplus_{i=0}^{\i}H^{2i}(X;R),\\
\eta_i^c&=\eta_{i;i}^c\!+\!\ldots\!+\!\eta_{r-1;i}^c
\in\bigoplus_{k=0}^{\i}H^{2k}(V^{(r)};R)\quad\forall\,i\!\in\![r\!-\!1],
\qquad \eta_r^c=1.
\end{split}\end{equation*}
The identity~\eref{blowupchern_e} is equivalent to
\BE{blowupchern_e2}\begin{split}
1\!-\!\frac{c(T\wt{X})}{\pi^*c(TX)(1\!+\!\PD_{\wt{X}}([\bE]_{\wt{X}})\!)}
=&\sum_{i=1}^{r-1}
\PD_{\wt{X}}\Big(\wt\io_*\!\Big(\!
c_1(\ga)^{i-1}\!\cap\!\big(\bE|_{\eta_i^c\cap\cPD(V)^{-1}\cap[V^{(r)}]}\big)\!\!\Big)\!\!\Big)\\
&\hspace{.5in}
+\PD_{\wt{X}}\Big(\pi^*\PD_X\big(
\io_*\big(\cPD(V)^{-1}\!\cap\![V^{(r)}]\big)\!\big)\!\Big),
\end{split}\EE
where $\io\!:V^{(r)}\!\lra\!X$ is the inclusion.
Via~\eref{PTdiff_e1}, \eref{blowupchern_e2} in particular gives
\begin{gather}\label{blowupchern_e5a}
c_1(T\wt{X})=\pi^*c_1(TX)\!-\!(r\!-\!1)\PD_{\wt{X}}\big([\bE]_{\wt{X}}\big) 
\hspace{2.5in}\hbox{if}~r\!\ge\!2;\\
\label{blowupchern_e5b}\begin{split}
c_2(T\wt{X})=\pi^*c_2(TX)\!-\!\PD_{\wt{X}}\bigg(
\wt\io_*\!\Big(\!\!\big(\pi^*c_1(TX)\!+\!2c_1(\ga)\!\big)\!\cap\![\bE]\!\Big)
\!+\!\pi^*\big(\PD_X[V^{(2)}]_X\big)&\\
\hspace{1.6in}-
2\bE\big|_{\PD_X([V^{(1)}]_X)\cap[V^{(2)}]}&\bigg)
\quad\hbox{if}~r\!=\!2.
\end{split}\end{gather}

\vspace{.15in}

\noindent
With the dimension of~$X$ and the codimension of the blowup locus~$V^{(r)}$ fixed,
the left-hand side of~\eref{blowupchern_e2} must be of the~form 
\begin{equation*}\begin{split}
&\sum_{i=1}^{r-1}
\PD_{\wt{X}}\Big(\wt\io_*\!\Big(\!
c_1(\ga)^{i-1}\!\cap\!
\big(\bE|_{P_i(c_1(\cN_XV^{(r)}),\ldots,c_r(\cN_XV^{(r)}))\cap[V^{(r)}]}\big)\!\!\Big)\!\!\Big)\\
&\hspace{1in}
+\PD_{\wt{X}}\Big(\pi^*\PD_X\big(
\io_*\big(P_r\big(c_1(\cN_XV^{(r)}),\ldots,c_r(\cN_XV^{(r)})\!\big)
\!\cap\![V^{(r)}]\big)\!\big)\!\Big)
\end{split}\end{equation*}
for some universal polynomials $P_1,\ldots,P_r\!\in\!\Z[c_1,\ldots,c_r]$.
Since~\eref{blowupchern_e} holds with $\Z$-coefficients in the SC case, 
the right-hand side of~\eref{blowupchern_e2} reduces to the above form in this case.
For example,
\BE{intervsc_e}  
\PD_X\big([V^{(k)}]_X\big)\!\cap\![V^{(r)}]=
c_k\big(\cN_XV^{(r)}\big)\!\cap\![V^{(r)}]\in H_*\big(V^{(r)};\Z)\EE
in the SC case.
This also occurs if the branches of~$V$ at~$V^{(r)}$ can be distinguished globally, 
i.e.~$\cN_XV^{(r)}$ splits into~$r$ subbundles which restrict to the subbundles
$\cN_{V_{y;I-i}}V_{y;I}$ with $i\!\in\!I$ for every chart 
$(U_y,\{V_{y;i}\}_{i\in S_y})$ as in Definition~\ref{NCD_dfn} and 
every $I\!\subset\!S_y$ with $|S_y|\!=\!r$.\\

\noindent
Since~\eref{blowupchern_e} holds with $\Q$-coefficients in general, 
the differences
$$\eta_i^c\!\cup\!\cPD(V)^{-1}\!-\!
P_i\big(c_1(\cN_XV^{(r)}),\ldots,c_r(\cN_XV^{(r)})\!\big)\in H_*\big(V^{(r)};\Z)$$
are torsion.
If the torsion in $H_*(\bE;\Z)$ lies in the kernel~$\wt\io_*$, then
the torsion in $H_*(V^{(r)};\Z)$ lies in the kernel~$\io_*$.
In such a case,
\begin{equation*}\begin{split}
\wt\io_*\Big(\bE_{\eta_i^c\cap\cPD(V)^{-1}\cap[V^{(r)}]}\Big)
&=\wt\io_*\Big(\bE_{P_i(c_1(\cN_XV^{(r)}),\ldots,c_r(\cN_XV^{(r)}))\cap[V^{(r)}]}\Big)
\in H_*(\wt{X};\Z),\\
\io_*\big(\cPD(V)^{-1}\!\cap\![V^{(r)}]\big)
&=\io_*\big(P_r\big(c_1(\cN_XV^{(r)}),\ldots,c_r(\cN_XV^{(r)})\!\big)\!\cap\![V^{(r)}]\big)
\in H_*(X;\Z).
\end{split}\end{equation*}
This implies that \eref{blowupchern_e} holds in $H^*(\wt{X};\Z)$ if 
the torsion in $H_*(V^{(r)};\Z)$ lies in the kernel~$\io_*$.\\

\noindent
We now give an example in which  \eref{blowupchern_e} does not hold in $H^*(\wt{X};\Z)$.
Let $\Si$ be a compact connected Riemann surface and $z^*\!\in\!\Si$.
Let $\wt{S}$ be a K3 surface and $\psi\!\in\!\Aut(\wt{S})$ be a fixed-point-free involution 
so~that \hbox{$S\!\equiv\!\wt{S}/\psi$} is an Enriques surface.
Define
\begin{gather*}
\wt\psi\!:\wt{X}\!\equiv\!\wt{S}\!\times\!\Si^2\lra\wt{S}, \quad
\wt\psi(y,z_1,z_2)=\big(\psi(y),z_2,z_1\big), \\
X=\wt{X}/\wt\psi, \qquad
\wt{V}=\wt{S}\!\times\!\Si\!\times\!\{z^*\}\subset\wt{X}.
\end{gather*}
The image $V\!\subset\!X$ of $\wt{V}$ under the quotient map $q\!:\wt{X}\!\lra\!X$
is an NC divisor in~$X$.
Its 2-fold locus $V^{(2)}\!\approx\!S$ is the image of \hbox{$\wt{S}\!\times\!\{z^*\}^2$} 
under~$q$.
Since the 3-fold locus of~$V$ is empty, $r\!=\!2$ in this case.
We show below~that
\BE{blowupchernEG_e1}
\io_*\big(\PD_X\big([V^{(1)}]_X\big)\!\cap\!\big[V^{(2)}\big]\!\big)=0\in H_2(X;\Z),
\quad c_1(\cN_XV^{(2)})\neq0\in H^2(V^{(2)};\Z),\EE
and the torsion of $H_2(V^{(2)};\Z)$ does not vanish under~$\io_*$.
This implies that the last term in~\eref{blowupchern_e5b} is not determined by
the Chern class of~$\cN_XV^{(2)}$.
Thus, \eref{intervsc_e}, \eref{blowupchern_e5b}, \eref{blowupchern_e}, and~\eref{cTXV_e} 
do not hold with $\Z$ coefficients in this case.\\

\noindent
The image of \hbox{$\wt{S}\!\times\!\{z\}^2$} under~$q$ 
is homologous to~$V^{(2)}$ for any $z\!\in\!\Si$.
Since it is disjoint from~$V$ for $z\!\neq\!z^*$,
$$\io_*\big(\PD_X\big([V^{(1)}]_X\big)\!\cap\!\big[V^{(2)}\big]\!\big)=
\PD_X\big([V^{(1)}]_X\big)\!\cap\!\big[V^{(2)}\big]_X=0.$$
This confirms the first statement in~\eref{blowupchernEG_e1}.\\

\noindent
The normal bundle $\cN_XV^{(2)}$ of $V^{(2)}$ in~$X$ is the quotient of the trivial bundle
$\wt{S}\!\times\!\C^2$ by the involution
$$\wt{S}\!\times\!\C^2\lra\wt{S}\!\times\!\C^2, \qquad
(y,c_1,c_2)\lra \big(\psi(y),c_2,c_1\big).$$
Thus, $\cN_XV^{(2)}\!\approx\!L_+\!\oplus\!L_-$, where
$$L_{\pm}=\big\{[y,c,\pm c]\!:y\!\in\!\wt{S},\,c\!\in\!\C\big\}\subset \cN_XV^{(2)}\,.$$
The complex line bundle~$L_+$ over~$S$ is isomorphic to~$S\!\times\!\C$.
The flat complex line bundle~$L_-$ corresponds to the non-trivial homomorphism
$$\pi_1(S)\!=\!\Z_2\lra S^1.$$
This confirms the second statement in~\eref{blowupchernEG_e1}.\\

\noindent
Let \hbox{$f_+\!:\Si_+\!\lra\!S$} be a smooth map from a compact Riemann surface 
representing the unique torsion element of~$H_2(S;\Z)$.
Let \hbox{$f_-\!:\Si_-\!\lra\!S$} be a smooth map from a compact (unorientable) surface 
representing a class in $H_2(S;\Z_2)$ so~that 
$$[f_+]_{\Z_2}\!\cdot_S\![f_-]_{\Z_2}\neq0\in\Z_2,$$
where $\cdot_S$ denotes the $\Z_2$-intersection product on~$S$.
The involution~$\wt\psi$ pulls back to a smooth involution
$$\wt\psi_-\!:\wt{Z}_-\!\equiv\!
\big\{\!(x,y)\!\in\!\Si_-\!\times\!\wt{S}\!:f_-(x)\!=\!q(y)\!\big\}\!\times\!\Si^2
\lra\wt{Z}_-,
\quad \wt\psi_-(x,y,z_1,z_2)=\big(x,\psi(y),z_2,z_1\big).$$
The~map
$$F_-\!:Z_-\!\equiv\!\wt{Z}_-/\wt\psi_-\lra X, \qquad
F_-\big([x,y,z_1,z_2]\big)=\big[y,z_1,z_2\big],$$
is smooth and determines an element of~$H_6(X;\Z)$ so that 
$$\io_*\big([f_+]_{\Z_2}\big)\!\cdot_X\![F_-]_{\Z_2}
=[f_+]_{\Z_2}\!\cdot_S\![f_-]_{\Z_2}\neq0\in\Z_2.$$
This confirms the claim just after~\eref{blowupchernEG_e1}.\\

\vspace{.2in}

\noindent
{\it The University of Iowa, MacLean Hall, Iowa City, IA 52241\\
mohammad-tehrani@uiowa.edu}\\

\noindent
{\it Department of Mathematics, Stony Brook University, Stony Brook, NY 11794\\
markmclean@math.stonybrook.edu, azinger@math.stonybrook.edu}\\

\end{document}